\documentclass[a4paper,10pt,leqno,twoside]{amsbook}
\newcommand{\Dim}{\hskip .1em\mathrm{dim}}

\newcommand{\mb}{\mathbb}
\newcommand{\C}{\mathbb{C}}
\newcommand{\R}{\mathbb{R}}

\newcommand{\ra}{\to}
\newcommand{\rba}{\Rightarrow}

\newcommand{\p}{\partial}

\newcommand{\orl}{\overline}

\newcommand{\vfi}{\varphi}

\newcommand{\ci}{\circ}

\newcommand{\Ind}{\mathrm{Ind}}
\newcommand{\ind}{\mbox{ind}}
\newcommand{\al}{\mathcal}
\newcommand{\scr}{\mathscr}

\newcommand{\var}{\varepsilon}

\newcommand{\bra}{\langle}
\newcommand{\ket}{\rangle}
\newcommand{\ran}{\mathrm{Range}\hskip 1pt}

\newcommand{\bx}[2]{\al{B}(#1,#2)}
\newcommand{\no}[1]{|\hskip -1pt | #1 |\hskip -1pt|}
\newcommand{\dist}{\mathrm{dist}}

\newcommand{\supp}{\hskip .1em\mathrm{supp}\hskip .1em}

\newcommand{\coker}{\hskip .1em\mathrm{coker}\hskip .1em}
\newcommand{\codim}{\hskip .1em\mathrm{codim}\hskip .1em}
\newcommand{\Sf}{\mathrm{sf}}

\newcommand{\Lim}{\hskip .1em\mathrm{Lim}\hskip .1em}

\newcommand{\set}[2]{\left\{\hskip .1em {#1}\mid{#2}\hskip .1em\right\}}
\newcommand{\Splt}{\mathrm{Splt}}

\newcommand{\img}{\mathrm{i}}
\renewcommand{\ind}{\hskip .1em\mathrm{ind}\hskip .1em}
\newcommand{\res}[2]{{#1}_{|{#2}}}
\def\polk#1{\setbox0=\hbox{#1}{\ooalign{\hidewidth
  \lower1.5ex\hbox{`}\hidewidth\crcr\unhbox0}}}
\usepackage{layout}
\usepackage{multicol}
\usepackage{glossaries}
\usepackage[all]{xy}
\usepackage[english]{babel}
\usepackage{amssymb}
\usepackage{amsmath}
\usepackage{amscd}
\usepackage{mathrsfs}
\usepackage{amsthm}
\usepackage{ifthen}
\usepackage{comment}
\usepackage{xcomment}
\usepackage{setspace}
\usepackage[colorlinks,citecolor={black}]{hyperref}
\renewcommand{\graph}{\hskip .1em\mathrm{graph}\hskip .1em}
\newcommand{\re}{\hskip .1em\mathrm{Re}\hskip .1em}
\newcommand{\im}{\hskip .1em\mathrm{Im}\hskip .1em}

\newcommand{\prc}{p}
\newcounter{counterxapp}
\newtheorem{remark}{Remark}[section]
\newtheorem{lemma}[remark]{Lemma}
\newtheorem{theorem}[remark]{Theorem}
\newtheorem{proposition}[remark]{Proposition}

\theoremstyle{definition}
\newtheorem{definition}[remark]{Definition}
\newtheorem{corollary}[remark]{Corollary}
\newtheorem{example}[remark]{Example}
\newcounter{Example}
{\vskip .2em}
\DeclareMathAlphabet{\mathpzc}{OT1}{pzc}{m}{it}
\excludecomment{TAKEOUT}
\includecomment{TAKEIN}
\makeglossaries
\usepackage{fancyhdr}
\usepackage{ifpdf}

\ifpdf
\usepackage[pdftex]{graphicx}
\DeclareGraphicsExtensions{.pdf}
\else
\usepackage[dvips]{graphicx}
\DeclareGraphicsExtensions{.eps}
\fi
\newglossaryentry{labLEF}{name={$ \mathcal{L}(E,F) $},
description={space of bounded linear maps from $ E $ to $ F $},sort=LEF}
\newglossaryentry{labGb}{name={$ G(\mathcal{B}) $},description={},sort=GB}
\newglossaryentry{labBp}{name={$ \mathcal{B}_p $},description={},sort=Bp}
\newglossaryentry{labPB}{name={$ \mathcal{P}(\mathcal{B}) $},
description={space of idempotents},sort=PB}
\newglossaryentry{labQB}{name={$ \mathcal{Q}(\mathcal{B}) $},
description={space of square roots of identity},sort=Q}
\newglossaryentry{labHB}{name={$ \mathcal{H}(\mathcal{B}) $},
description={space of hyperbolic elements},sort=H}
\newglossaryentry{labsx}{name={$ \sigma(x),\sigma_{\mathcal{B}} (x) $},
description={spectrum of $ x $},sort=s}
\newglossaryentry{labpAx}{name={$ p(A;x) $},description={spectral projector},
sort=p}
\newglossaryentry{labseT}{name={$ \sigma_e (T) $},
description={essential spectrum of $ T $},sort=T}
\newglossaryentry{labCE}{name={$ \mathcal{C}(E) $},
description={Calkin algebra},sort=C}
\newglossaryentry{labPs}{name={$ \Psi $},description={},sort=Psi}
\newglossaryentry{labPh}{name={$ \Phi $},description={},sort=Phi}
\newglossaryentry{laborlp}{name={$ \orl{p} $},
description={the projector $ 1 - p $},sort=p}
\newglossaryentry{labGB}{name={$ G_0 (\mathcal{B}) $},
description={connected component of the unit},sort=Gb}
\newglossaryentry{labGrB}{name={$ Gr(\mathcal{B}) $},
description={Grassmannian algebra},sort=G}
\newglossaryentry{labrP}{name={$ r(P) $},
description={the range of a projector $ P $},sort=r}
\newglossaryentry{labPcE}{name={$ \mathcal{P}_c (P;E) $},
description={class of $ P $ for the relation of compact perturbation},sort=P}
\newglossaryentry{labGcE}{name={$ G_c (X;E) $},
description={class of $ X $ for the relation of commensurability},sort=G}
\newglossaryentry{labPeE}{name={$  \mathcal{P}_e (E) $},description={},sort=P}
\newglossaryentry{labGeE}{name={$ G_e (E) $},
description={essential Grassmannian},sort=G}
\newglossaryentry{labprc}{name={$ \prc_e $},description={},sort=p}
\newglossaryentry{labre}{name={$ r_e $},
description={quotient of $ r $ by the relation of compact perturbation},sort=r}
\newglossaryentry{labrc}{name={$ r_c $},
description={the restriction of $ r $ to $ \mathcal{P}(P;E) $},sort=r}
\newglossaryentry{labGLcE}{name={$ GL_c (E) $},
description={the Fredholm group},sort=G}
\newglossaryentry{labdT}{name={$ \deg T $},description={Leray-Schauder degree},
sort=d}
\newglossaryentry{labGkX}{name={$ G_k (X;E) $},
description={linear subspaces with relative dimension $ k $ with $ X $},
sort=G}
\newglossaryentry{labStX}{name={$ St(X;E) $},
description={the Stiefel space of compact perturbations of the inclusion of 
$ X $},sort=S}
\newglossaryentry{labrSt}{name={$ r_{St} $},
description={fibration map of ($ St(X;E),G_0 (X;E) $)},sort=r}
\newglossaryentry{labGpB}{name={$ G_p (\mathcal{B}) $},description={},sort=GpB}
\newglossaryentry{labBsEF}{name={$ \mathcal{B}_s (E,F) $},
description={strictly singular operators from $ E $ to $ F $},sort=B} 
\newglossaryentry{labPQ}{name={$ [P - Q] $},
description={relative dimension of the images of $ P $ and $ Q $},sort=PQ}
\newglossaryentry{labeHE}{name={$ e\mathcal{H}(E) $},
description={space of essentially hyperbolic operators},sort=eHE}
\newglossaryentry{labgpq}{name={$ g(p,q) $},description={},sort=gpq}
\newglossaryentry{labPBp}{name={$ \mathcal{P}_p (\mathcal{B}) $},
description={},sort=PBp}
\newglossaryentry{labev0}{name={$ ev_0 $},
description={evaluation at the point $ t = 0 $},sort=ev}
\newglossaryentry{labFA}{name={$ F_A $},
description={the differential operator $ F_A (u) = u' - Au $},sort=FA}
\newglossaryentry{labpp}{name={$ \pi_p $},description={},sort=pp}
\newglossaryentry{labGsg}{name={$ G_\sigma $},description={},sort=Gsg}
\newglossaryentry{labSfAP}{name={$ \Sf(A;P) $},description={},sort=SfAP}
\newglossaryentry{labSf}{name={$ \Sf $},description={},sort=Sf}
\newglossaryentry{labwf}{name={$ \widehat{f} $},description={},sort=wf}
\newglossaryentry{labdaB}{name={$ \dist(a,B) $},description={},sort=daB}
\newglossaryentry{labHX}{name={$ \scr{H}(X) $},description={Hausd\"orff space},
sort=HX}
\newglossaryentry{labrHAB}{name={$ \rho_{\scr{H}} (A,B) $},
description={Hausd\"orff semi-metric},sort=rHAB}
\newglossaryentry{labdHAB}{name={$ \delta_{\scr{H}} (A,B) $},
description={Hausd\"orff metric},sort=dHAB}
\newglossaryentry{labGE}{name={$ G(E) $},
description={Grassmannian of closed subspaces},sort=GE}
\newglossaryentry{labDY}{name={$ D(Y) $},description={unit disc},sort=DY}
\newglossaryentry{labSY}{name={$ S(Y) $},description={unit sphere},sort=SY}
\newglossaryentry{labrYZ}{name={$ \rho(Y,Z) $},
description={Grassmannian semi-metric},sort=rYZ}
\newglossaryentry{labdYZ}{name={$ \delta (Y,Z) $},
description={Grassmannian metric},sort=dYZ}
\newglossaryentry{labrSYZ}{name={$ \rho_S (Y,Z) $},
description={sphere opening semi-metric},sort=rSYZ}
\newglossaryentry{labdSYZ}{name={$ \delta_S(Y,Z) $},
description={sphere opening metric},sort=dSYZ}
\newglossaryentry{labr1YZ}{name={$ \rho_1 (Y,Z) $},
description={geometric opening semi-metric},sort=r1YZ}
\newglossaryentry{labd1YZ}{name={$ \delta_1 (Y,Z) $},
description={geometric opening metric},sort=d1YZ}
\newglossaryentry{labr0XY}{name={$ r_0 (X,Y) $},
description={Sch\"affer semi-metric},sort=r0XY}
\newglossaryentry{labrXY}{name={$ r(X,Y) $},description={Sch\"affer metric},
sort=rXY}
\newglossaryentry{labGLE}{name={$ GL(E) $},
description={group of invertible operators of $ E $},sort=G}
\newglossaryentry{labEst}{name={$ E^* $},description={topological dual space},
sort=Est}
\newglossaryentry{labSperp}{name={$ S^{\bot} $},description={annihilator},
sort=Sperp}
\newglossaryentry{labGsE}{name={$ G_s (E) $}, 
description={Grassmannian of splitting subspaces},sort=GsE}
\newglossaryentry{labE/Y}{name={$ E/Y $},description={quotient space},
sort=E/Y}
\newglossaryentry{labPYZ}{name={$ P(Y,Z) $},description={projector onto $ Y $
along $ Z $},sort=PYZ}
\newglossaryentry{labgYZ}{name={$ \gamma(Y,Z) $},description={the semi-gap},
sort=gYZ}
\newglossaryentry{labgYZh}{name={$ \hat{\gamma}(Y,Z) $},
description={the minimum gap},sort=gYZh}
\newglossaryentry{labPE}{name={$ \mathcal{P}(E) $},
description={space of projectors},sort=PE}
\newglossaryentry{labLcEF}{name={$ \mathcal{L}_c (E,F) $},
description={compact operators from $ E $ to $ F $},sort=LcEF}
\newglossaryentry{labdimXY}{name={$ \Dim(X,Y) $},
description={relative dimension},sort=dimXY}
\newglossaryentry{labfind}{name={$ \ind(T) $},
description={semi-freholm index of $ T $},sort=find}
\newglossaryentry{labindXY}{name={$ \ind(X,Y) $},
description={semi-freholm index of $ (X,Y) $},sort=indXY}
\newglossaryentry{labXA}{name={$ X_A $},
description={solution of $ U' = AU $ with $ U(0) = 1 $},sort=XA}
\newglossaryentry{labWAs}{name={$ W_A ^s $},description={stable space},
sort=WAs}
\newglossaryentry{labWAu}{name={$ W_A ^u $},description={unstable space},
sort=WAu}
\newglossaryentry{labBl}{name={$ \mathcal{B}_l $},
description={left inverses of $ \mathcal{B} $},sort=Bl}
\newglossaryentry{labBr}{name={$ \mathcal{B}_r $},
description={right inverses of $ \mathcal{B} $},sort=Br}
\newglossaryentry{labAtau}{name={$ A_{\tau} $},
description={the translation path defined as of $ A(t + \tau) $},sort=Atau}
\newglossaryentry{labH+}{name={$ \mathbb{H}^{+} $},
description={complexes with positive real part},sort=H+}
\newglossaryentry{labH-}{name={$ \mathbb{H}^- $},
description={complexes with negative real part},sort=H-}
\newglossaryentry{labphAx0}{name={$ \varphi_{A,x_0} $},description={},sort=pA}
\newglossaryentry{labLA}{name={$ L_A $},description={},sort=LA}
\newglossaryentry{labpsAy}{name={$ \psi_{A,\orl{y}} $},description={},
sort=psAy}
\newglossaryentry{labRA}{name={$ R_A $},description={},sort=RA}
\newglossaryentry{labC0}{name={ $ C_0 (\mathbb{R},E) $ },
description={continuous function vanishing at infinity},sort=C0}
\newglossaryentry{labC10}{name={ $ C^1 _0 (\mathbb{R},E) $ },
description={smooth function vanishing at infinity with their derivative},
sort=C10}
\newglossaryentry{labC10+}{name={ $ C^1 _0 (\mathbb{R}^+,E) $ },
description={functions of $ C^1 _0 $ defined on the half-line},
sort=C10+}%
\newglossaryentry{labC0+}{name={ $ C _0 (\mathbb{R}^+,E) $ },
description={functions of $ C_0 $ defined on the half-line},sort=C0+}
\newglossaryentry{lab1S}{name={$ 1_S $},
description={the characteristic function of $ S $},sort=1S}
\newglossaryentry{labRAPs}{name={ $ R_{A,P_s} ^+ $ },description={},sort=RAPs}
\newglossaryentry{labRAPu}{name={ $ R_{A,P_u} ^- $ },description={},sort=RAPu}
\newglossaryentry{labrAPs}{name={ $ r_{A,P_s} ^+ $ },description={},sort=rAPs}
\newglossaryentry{labrAPu}{name={ $ r_{A,P_u} ^- $ },description={},sort=rAPu}
\newglossaryentry{labSfA}{name={$ \Sf(A) $},
description={the spectral flow of a path $ A $},sort=SfA}
\newglossaryentry{labeHpE}{name={ $ e\mathcal{H}_+ (E) $ },
description={essentially hyperbolic with positive essential spectrum},
sort=eHpE}
\newglossaryentry{labeHmE}{name={ $ e\mathcal{H}_- (E) $ },
description={essentially hyperbolic with negative essential spectrum},
sort=eHmE}
\glossarystyle{long3col}
\fancypagestyle{plain}{%
\fancyhf{}%

}
\makeatletter
\def\cleardoublepage{\clearpage\if@twoside \ifodd\c@page\else%
\hbox{}%
\thispagestyle{empty}% % Empty header styles
\newpage%
\if@twocolumn\hbox{}\newpage\fi\fi\fi}
\makeatother
\begin{document}
\title{Ordinary differential equations in Banach spaces and %
the spectral flow}
\author{Garrisi Daniele}
%\address{Math Sci. Bldg Room \# 302, POSTECH, Hyoja-Dong, Nam-Gu, Pohang, 
%Gyeongbuk, 790-784, Republic of Korea}
%\email{garrisi@postech.ac.kr}
\dedicatory{To my parents and my sister}
%\date{today}
%\keywords{spectral flow, Morse theory, Fredholm index, Grassmannian, 
%hyperplane}
\thanks{This work was supported by Priority 
Research Centers Program through the National Research Foundation of Korea 
(NRF) funded by the Ministry of Education, Science and Technology 
(Grant 2009-0094068) and from the Scuola Normale Superiore of Pisa}
%\layout
\maketitle
\pagestyle{empty}
%\cleardoublepage
%\include{dedica-mass}
\tableofcontents

% needed in order to clear headings
\clearpage
\pagestyle{fancy}
\fancyhf{}
\pagenumbering{roman}
\fancyhead[LE,RO]{\small\thepage}
%\include{introduction-v2}

% --- Here we define actually how fancy style is --- %
\renewcommand{\chaptermark}[1]{%
\markboth{\footnotesize\Roman{chapter}.\, #1}{}}
\fancyhead[CO]{\leftmark}

\renewcommand{\sectionmark}[1]{%
\markright{\footnotesize\S\arabic{section}.\, #1}{}}
\fancyhead[CE]{\rightmark}

\setcounter{page}{0}
\pagenumbering{arabic}

%%%%%%%%%%%%%%%%%%%%%%%%%%%%%%

%%% TEXEXPAND: INCLUDED FILE MARKER ./introduction-v2.tex
\chapter*{Introduction}
Given a path $ \set{A(t)}{t\in\mathbb{R}} $, of linear operators on some 
Banach space $ E $, we consider the differential operator
\begin{gather*}
F_A u = \left(\frac{d}{dt} - A(t)\right)u
\end{gather*}
on suitable spaces of curves $ u\colon\mathbb{R}\rightarrow E $. A classical
question is whether the operator $ F_A $ is Fredholm and what is its index.
If $ A(t) $ is a path of unbounded operators 
the literature is rich. We recall the work of J.~Robbin and D.~Salamon, 
\cite{RS95}, where $ A $ is an asymptotically hyperbolic path of 
unbounded self-adjoint operators
and  defined on a common domain 
$ W\subset H $ compactly included in a Hilbert space $ H $. 
For such paths they prove that the differential operator
\begin{gather*}
F_A\colon L^2 (\mathbb{R},W)\cap W^{1,2} (\mathbb{R},H) \rightarrow 
L^2 (\mathbb{R},H), \ \ \ u\mapsto u' - A u 
\end{gather*}
is Fredholm. The index of $ F_A $ is minus the \textsl{spectral flow} of $ A $,
an integer which counts algebraically the eigenvalues of $ A(t) $
crossing $ 0 $. The result applies to \textsl{Cauchy-Riemann} operators
and it is widely used in Floer homology.
This result has been generalized to Banach spaces with the 
\textsl{unconditional martingale difference} (UMD) property by P.~Rabier in
\cite{Rab04}; the compact inclusion of the domain is still required. In
this setting, the identity
\begin{gather}
\label{ind=-sf}
\ind F_A = - \Sf(A).
\end{gather}
holds. Y.~Latushkin and T.~Tomilov in \cite{LT05} proved the Fredholmness
of the operator $ F_A $ for paths $ A $ with variable domain 
$ D(A(t))\subset E $ with $ E $ reflexive using 
\textsl{exponential dichotomies}. D.~di Giorgio, A.~Lunardi and 
R.~Schnaubelt in \cite{DGLS05} obtained the same results for 
\textsl{sectorial operators} in an arbitrary Banach space and give necessary
and sufficient conditions on the \textsl{stable} and \textsl{unstable spaces}
in order to have the Fredholmness of $ F_A $.
\vskip .2em
For the bounded case the problem has been studied by A. Abbondandolo and
P. Majer in \cite{AM03b}. This setting is suggested by the Morse Theory on a 
Hilbert manifold $ M $: given a vector field $ \xi $ on $ M $ and $ \phi_t $ 
its flow, $ x $ and $ y $ hyperbolic zeroes of $ \xi $ the stable and unstable 
manifolds
\begin{align*}
W_{\xi} ^s (x) &= \set{p\in M}{\lim_{t\rightarrow +\infty} \phi_t (p) = x} 
\\
W_{\xi} ^u (y) &= \set{p\in M}{\lim_{t\rightarrow -\infty} \phi_t (p) = y}
\end{align*}
are immersed sub-manifolds of $ M $, in fact they are sub-manifolds if the
vector field is the gradient of a Morse function on $ M $. It is not
hard to check that the intersection of the stable and unstable manifold 
of two different zeroes is a sub-manifold if, for every curve 
$ u'(t) = \xi(u(t)) $ such that
$ u(+\infty) = x $ and $ u(-\infty) = y $, the differential operator
\begin{gather*}
F_A (v) = v' - A v, \ \ A(t) = D\xi (u(t))
\end{gather*}
is surjective and $ \ker F_A $ splits. Since $ x $ and $ y $ are
hyperbolic zeroes $ A(+\infty) $ and $ A(-\infty) $ are hyperbolic operators. 
In \cite{AM03b} the study of the Fredholm index of such operator is carried 
out by considering the stable and unstable spaces
\begin{align*}
W_A ^s &= \set{x\in E}{\lim_{t\ra +\infty} X_A (t) x = 0}\\
W_A ^u &= \set{x\in E}{\lim_{t\ra -\infty} X_A (t) x = 0},
\end{align*}
where $ X_A $ is the solution of the Cauchy problem $ X' = AX $ with
$ X(0) = I $. If $ A $ is an \textsl{asymptotically hyperbolic} path on
a Hilbert space the following facts hold: \vskip .2em
\textbf{Fact 1}. The stable and unstable spaces $ W_A ^s $ and $ W_A ^u $ are
closed in $ E $ and admit topological complements, \textsc{Proposition} 1.2 of
\cite{AM03b}. \vskip .2em
\textbf{Fact 2}. The evolution of the stable space $ X_A (t) W_A ^s $ converges
to the negative eigenspace of $ A(+\infty) $, and any topological complement
of $ W_A ^s $ converges to the positive eigenspace of $ A(+\infty) $, with
a suitable topology on the set of closed linear subspaces of a Hilbert,
see \textsc{Theorem} 2.1 of \cite{AM03b}. \vskip .2em
\textbf{Fact 3}. If two paths $ A $ and $ B $ have compact difference for
every $ t\in\mathbb{R} $ the stable space $ W_A ^s $ is compact perturbation
of $ W_B ^s $, \textsc{Theorem} 3.6 of \cite{AM03b}. \vskip .2em
\textbf{Fact 4}. The operator $ F_A $ is semi-Fredholm if and only if 
$ (W_A ^s, W_A^u) $ is a semi-Fredholm pair; in this case
$ \ind F_A = \ind (W_A ^s,W_A ^u) $, \textsc{Theorem} 5.1 of \cite{AM03b} 
\vskip .2em
In the bounded setting the spectral flow is defined in 
\cite{Phi96} for paths in $ \mathcal{F}^{sa}(E) $, the set of 
Fredholm and self-adjoint bounded operators. Unlike the unbounded case
described in \cite{RS95} and in \cite{Rab04}, given an asymptotically 
hyperbolic path in $ \mathcal{F}^{sa} (E) $ the equality 
$ \ind F_A = -\Sf(A) $ does not hold in general. Examples are provided in 
\S 7 of \cite{AM03b}. Our purpose is to generalize firstly these facts to an 
arbitrary Banach space $ E $ and, secondly, to define the spectral flow for 
suitable paths and prove that for a class of paths the relation 
(\ref{ind=-sf}) holds.\vskip .2em
In the first chapter we define some metrics on the set
of closed linear subspaces of $ E $, the \textsl{Grassmannian} of $ E $,
denoted by $ G(E) $, and the subset of complemented subspaces, 
denoted by $ G_s (E) $. This is done in order to have a definition of 
\textsl{convergence of subspaces}. Our main reference
is the work of E.~Berkson, \cite{Ber63}. We also establish which pairs of
closed subspaces $ (X,Y) $ are \textsl{compact perturbation of each other}
and the relative dimension for such pairs is defined. In finite-dimensional
spaces, the relative dimension is $ \dim(X) - \dim(Y) $. A definition of
relative dimension exists in Hilbert spaces, refer \cite{AM03b,BDF73},
and we know of an existing definition in Banach spaces for pairs of projectors
$ (P,Q) $ with compact difference, in \cite{ZL99}. 
In this chapter, we extend the definition to arbitrary pairs of closed 
subspaces in every Banach space $ E $. These definitions allow to state 
\textbf{Fact 1} and \textbf{Fact 3}.
\vskip .2em
We use the notation $ \mathcal{P}(E) $ for the space of projectors on
a Banach space $ E $ and $ \mathcal{P}(\mathcal{C}(E)) $ for the space
of projectors of the Banach algebra 
$ \mathcal{C}(E) = \mathcal{L}(E)/\mathcal{L}_c (E) $. In chapter \ref{chap2},
we prove in Theorem \ref{index_homomorphism} that, for every projector $ P $, 
we can define a group homomorphism, namely $ \vfi_P $, on the fundamental 
group of $ \mathcal{P}(\mathcal{C}(E)) $ at the base point $ p(P) $ such that 
the sequence 
\[
\xymatrix{
\pi_1 (\al{P}(E),P) \ar[r]^-{\prc_*} &
\pi_1 (\al{P}(\al{C}),p(P)) \ar[r]^-{\vfi_P} &
\mb{Z}
}
\]
is exact. 
\begin{TAKEOUT}
The properties of such homomorphism are characterized 
 we prove that the restriction of quotient projection 
to the space of projectors $ \mathcal{P}(E) $ onto the projectors of
the the space of \textsl{essentially hyperbolic operators}, which are merely the sum of a hyperbolic operator and a compact 
one, and prove in Theorem \ref{topology} that this space has the homotopy
type of $ \mathcal{P}(\mathcal{C}) $

recall the definitions of some classical 
subspaces of a Banach algebra such as the \textsl{idempotent} elements and 
\textsl{square roots} of unit and the \textsl{Calkin algebra} 
$ \mathcal{C} $, obtained as the 
quotient algebra of bounded operators by the compact ones. We call an 
operator $ A $ \textsl{essentially hyperbolic} if the spectrum of
$ A + \mathcal{L}_c (E)\in\mathcal{C} $ does not meet the imaginary axis. 
In Theorem \ref{topology}, we show that $ e\mathcal{H}(E) $ has the homotopy 
type of the space of idempotent elements of the Calkin algebra. A
homotopy equivalence is 
\begin{gather*}
\Psi\colon e\mathcal{H}(E)\rightarrow\mathcal{P}(\mathcal{C}), 
\ \ A\mapsto P^+ (A + \mathcal{L}_c (E))
\end{gather*}
where $ P^+ $ denotes the eigenprojector relative to the positive spectrum.
Our aim is to define a group homomorphism on the fundamental 
group of $ \mathcal{P}(\mathcal{C}) $, the space of idempotent elements of 
$ \mathcal{C} $, with values in $ \mathbb{Z} $. In fact such 
homomorphism completes the \textsl{long exact sequence} of the fiber bundle 
\[
\prc\colon\mathcal{P}(E)\rightarrow\mathcal{P}(\mathcal{C}), 
\ \ \ P\mapsto P + \mathcal{L}_c (E)
\]
where $ \mathcal{P}(E) $ is the space of projectors on $ E $. If we say that
two projectors are compact perturbation (one of each other) if their
difference is compact, then the function that maps a projector to its range 
preserves the relation of compact perturbation of closed linear subspaces
defined in chapter \ref{chap1}. Hence, given a projector $ P $ onto a closed 
subspace $ X\subset E $, we can consider the equivalence class of 
$ P $ in the space of projectors, denoted by $ \mathcal{P}_c (P;E) $,
and the equivalence class of $ X $ in $ G_s (E) $, denoted by $ G_c (X;E) $.
We denote by $ \mathcal{P}_e (E) $ and $ G_e (E) $ the quotient spaces
respectively. The latter is called \textsl{essential Grassmannian}. The map 
$ r(P) = P(E) $ induces the homotopy equivalences
\begin{gather*}
\mathcal{P}_c (P;E)\rightarrow G_c (X;E), \ 
\mathcal{P}_e (E)\rightarrow G_e (E).
\end{gather*}
These equivalences are well known in Hilbert spaces 
(\cite{AM03a} is our main reference). In order to extend these results to an 
arbitrary Banach space some techniques used by K.~G\polk{e}ba in \cite{Geb68} 
can be adapted. Using the \textsl{Leray-Schauder degree} we prove in
Theorem \ref{components} that the connected components of 
$ \mathcal{P}_c (P;E) $ are in correspondence with $ \mathbb{Z} $. 
Hence the homomorphism is defined as the composition
\begin{gather*}
\xymatrix{
\pi_1 (\mathcal{P}(\mathcal{C}),[P])\ar[r]^(0.5){\partial}
& \pi_0 (\mathcal{P}_c (P;E))\ar[r] 
& \mathbb{Z}}
\end{gather*}
where $ \partial $ is the map induced by the long exact sequence of the
fiber bundle 
$ (\mathcal{P}(E),\mathcal{P}(\mathcal{C}),\prc) $. This
homomorphism will be denoted by $ \varphi $ or called sometimes 
\textsl{index} of the exact sequence. Given $ P $ in $ \mathcal{P}(E) $
\end{TAKEOUT}
We characterize the elements of the image of $ \varphi_P $ and give
a sufficient condition making $ \varphi_P $ injective. 
We have
\begin{itemize}
\item[h1)] $ P $ is connected to a projector $ Q $ such that 
$ Q - P\in\mathcal{L}_c (E) $ and $ \dim(Q,P) = m $ if and only if
$ m\in\text{Im}(\vfi_P) $;
\item[h2)] the connected component of $ P $ in $ \mathcal{P}(E) $ is 
simply-connected.
\end{itemize}
We show some examples of Banach space where no projectors fulfills h1)
and prove in Proposition \ref{surjective} that if $ E $ has a 
complemented subspace, sum of two subspaces isomorphic to each other and
to their closed subspaces of co-dimension $ m $, then a projector on each
of the two factors satisfies condition h1). 
These properties are verified by an orthogonal projection in a Hilbert space
with infinite-dimensional kernel and range. The most common Banach spaces
such as the measure spaces 
$ L\sp p  (\Omega,\mu) $, $ L\sp\infty (\Omega,\mu) $ and spaces of sequences 
$ \ell\sp p,m,c_0 $ 
(see \cite{Mit70,Sch98} for a richer list and references) have a 
projector satisfying conditions h1,2). When two projectors are connected
by a path of projectors, their ranges are isomorphic, thus, in view of h2),
the problem of determining the image of $ \vfi_P $ is strictly related
to the question whether $ \ran P $ is isomorphic to its subspaces of
co-dimension $ m $ or not. As a consequence of the counterexamples of
W.~T.~Gowers and B.~Maurey in \cite{GM93}, of infinite-dimensional Banach
space not isomorphic to its hyperplanes, such $ P $ might not exits. 
Using a construction in \cite{GM97} of a space isomorphic to their hypersquares
(that is, subspace of co-dimension 2), but not hyperplanes, we can show
easily that in some case $ \text{Im}(\vfi_P) = 2\mb{Z}\subset\mb{Z} $.
Using a construction of A.~Douady in \cite{Dou65}, we show that $ \vfi_P $
is not injective even if it can be surjective. 
\vskip .2em
In chapter \ref{chap3} we study basic properties of the Cauchy problem
$ X'(t) = A(t) X(t),X(0) = 1 $. Here $ A(t) $ is a continuous and bounded path 
on $ \mathcal{L}(E) $. We define the stable and unstable spaces of $ A $,
denote by $ W_A \sp s $ and $ W_A \sp u $, 
and prove \textbf{Fact 1,2} in Theorem \ref{rec}. The proof differs from
the one that the authors of \cite{AM03b} used for Hilbert spaces, only
for the lack of a scalar product in (\ref{fin},\ref{last}) which can be fixed 
using results of continuous selection, refer Appendix \ref{app:sections} and
\cite{BG52}. Moreover, the stable manifold is well behaved with respect
to small, Theorem \ref{small_perturbations} and compact perturbations,
Theorem \ref{compact}, that is \textbf{Fact 3}. 
\vskip .2em 
In chapter \ref{chap4} we study the Fredholm properties of $ F_A $, defined
on the space of continuously differentiable functions vanishing at infinity
with their derivatives with values on continuous, vanishing at infinity. 
In Theorem \ref{facts} we prove that $ F_A $ is a semi-Fredholm operator if
and only if the pair $ (W_A \sp s, W_A \sp u ) $ is semi-Fredholm and, if
this is the case, the index is the same. The extension is made with 
slight modification of the argument of \textsc{Theorem} 5.1 of \cite{AM03b}.
\vskip .2em
In chapter \ref{chap5}, we give a definition of spectral flow for
paths in the space of essentially hyperbolic operators, denoted by
$ e\mathcal{H}(E) $. Such definition coincides with the one given by
C.~Zhu and Y.~Long in \cite{ZL99} and improves it making the
spectral flow easier to compute and to produce examples. 
Their definition, and thus ours, generalizes to Banach spaces the definition
known for Hilbert spaces (refer \cite{Phi96}).
In Theorem \ref{f=-p}, we prove that, given a projector $ P $,
the composition $ \Sf_{2P - I}\circ\Phi $, where $ \Sf_{2P - I} $ 
denotes the spectral flow defined on the fundamental group of 
$ (e\mathcal{H}(E),P) $ and $ \Phi $ is the homotopy equivalence defined
in Theorem \ref{topology}, coincides with $ - \varphi_P $. 
Hence, anything holds for the index $ \varphi_P $ is true for the spectral 
flow as well, including conditions h1) and h2). Thus, in contrast with the
behaviour in a separable, infinite-dimensional Hilbert 
space, where the spectral flow is either trivial or an isomorphism, due to 
the examples provided in Chapter \ref{chap2}, we have different behaviours. 
In the last section we prove that for a suitable class of paths in 
$ e\mathcal{H}(E) $, namely the \textsl{essentially splitting} and 
asymptotically hyperbolic ones, there holds
\[
\ind F_A = - \Sf(A).
\]
In \cite{AM03b}, A.~Abbondandolo and P.~Majer guessed that the above
equality holds in more restricted class of operators, where the
positive and negative eigenspaces are fixed
\[
E\sp + (A(t)) = E\sp +,\quad E\sp - (A(t)) = E\sp -.
\]
Our result proves that the guess is correct 
and extends to the class of essentially hyperbolic and essentially
splitting operators. We achieve this result in several steps: 
in Lemma \ref{essential_splitting}, we prove that an asymptotically 
hyperbolic path, $ A $, is essentially 
splitting if and only if the projectors of the set 
$ \set{P^+ (A(t))}{t\in\R} $ are compact perturbation of each other.
In Theorem \ref{sf_of_ess_split} we compute the spectral flow for an 
essentially splitting path
\[
\Sf(A) = - \dim (\ran P^- (A(+(\infty))), \ran P^- (A(-\infty))).
\]
For such paths we also compute the Fredholm index of $ F_A $ in Theorem 
\ref{essential_ind} 
\[
\ind F_A = \dim (\ran P^- (A(+(\infty))), \ran P^- (A(-\infty))).
\]
Thus we obtain the desired equality $ \Sf(A) = -\ind F_A $. 
\vskip .4em
I would like to thank my advisers Alberto Abbondandolo and Pietro Majer
for their aid and suggestions and my parents and my sister for always 
encouraging and helping me.

%%% TEXEXPAND: END FILE ./introduction-v2.tex
%%% TEXEXPAND: INCLUDED FILE MARKER ./chapter1-v2.tex
\chapter{Topology of the Grassmannian}
\label{chap1}
Given a metric space $ (X,d) $, we define the \textsl{Hausd\"orff space}, 
which is the set of bounded and closed subsets endowed with the
distance metric and denoted by $ (\scr{H}(X),d_{\scr{H}}) $.
We can define a metric on the family of closed, linear subspaces of a given 
Banach space $ E $ through the map that associates a linear space $ Y $
with the unit disc of $ Y $. We call this metric space \textsl{Grassmannian} 
and show that equivalent metrics, namely $ \delta_S $ and $ \delta_1 $, can be
defined on it. We show that the topology is well-behaved with respect
to the action of invertible operators of $ E $, to the graph and to the 
annihilator $ Y\mapsto Y\sp\perp $. 
We show that the subset of the linear, complemented subspaces is open. 
The last two sections of the chapter deal with the definition of 
\textsl{relative dimension} for pairs of closed linear subspaces, its
relation with compact perturbations of projectors and Fredholm operators.
We know of an existing definition of relative dimension for pairs of projectors
in \cite{ZL99}. Our main references are \cite{Ber63,Ost94,Kat95}.
\section{The Hausd\"orff metric}
Let $ (X,d) $ be a metric space. Given two subsets of $ A,B\subseteq X $ it is
well defined the distance
\[
\dist(a,B) = \inf_{b\in B} d(a,b).
\glsadd{labdaB}
\]
We denote by $ \scr{H}(X) $ 
\glsadd{labHX}
the family of closed, nonempty and bounded subsets of $ X $. 
It is defined a metric on $ \scr{H}(X) $ as follows: 
let $ A, B $ be two closed and bounded subsets of $ X $. Define
\begin{gather*}
\rho_{\scr{H}} (A,B) = \sup_{a\in A} \dist (a,B), 
\ \delta_{\scr{H}} (A,B) = \max\{\rho_{\scr{H}} (A,B),\rho_{\scr{H}} (B,A)\};
\end{gather*}
\glsadd{labrHAB}
\glsadd{labdHAB}
the second is called \textsl{Hausd\"orff metric}. We show that it has 
all the properties of a metric. It is clearly symmetric; if 
$ \rho_{\scr{H}} (A,B) = 0 $ 
$ A\subset B $ because $ B $ is closed.  Thus $ \delta_{\scr{H}} (A,B) = 0 $ if
and only if $ A = B $. For the triangular inequality let 
$ A,B,C\in\scr{H}(X) $ be closed and bounded subsets of $ X $. 
Given $ \var > 0 $ there exists $ a_1 \in A $ such that
\begin{gather}
\label{diseguaglianza_triangolare}
\rho_{\scr{H}} (A,C) \leq \var + \dist(a_1,C)\leq\var + d(a_1,b) + d(b,c)
\end{gather}
for any $ (b,c)\in B\times C $. Taking $ b_1 \in B $ such that
$ d(a_1,b_1)\leq\var + \dist(a_1,B) $, (\ref{diseguaglianza_triangolare}) 
becomes
\begin{gather*}
\rho_{\scr{H}} (A,C) \leq 2 \var + \dist (a_1,B) + d(b_1,c) 
\end{gather*}
for any $ c\in C $. Taking the infimum over $ C $ we find that 
$ \rho_{\scr{H}} (A,C)\leq \rho_{\scr{H}} (A,B) + \rho_{\scr{H}} (B,C) $.
Finally, suppose that $ \delta_{\scr{H}} (A,C) = \rho_{\scr{H}} (A,C) $. 
Therefore
\begin{gather*}
\delta_{\scr{H}} (A,C) = \rho_{\scr{H}} (A,C)\leq \rho_{\scr{H}} (A,B) + \rho_{\scr{H}} (B,C) \leq \delta_{\scr{H}} (A,B) + \delta_{\scr{H}} (B,C).
\end{gather*}
The following proposition states a relation between the metric spaces 
$ (X,d) $ and $ (\scr{H},\delta_{\scr{H}}) $. The proof of this can also be 
found in \cite{KC60}.
\begin{proposition}
\label{hausdorff_complete}
The application 
$ \delta_{\scr{H}}\colon \scr{H}\times \scr{H}\rightarrow\R^+ $ defines a 
complete metric in $ \scr{H}(X) $ if and only if $ (X,d) $ is complete. 
Moreover if $ \set{A_n}{n\in\mathbb{N}} $ is a converging sequence its limit
is the set of the limits of sequences $ \{a_n\} $ such that $ a_n \in A_n $.
\end{proposition}
\begin{proof}
We have proved that $ \delta_\scr{H} $ is a metric. Given $ a,b\in X $ it 
follows from the definition that $ \delta_{\scr{H}} (\{a\},\{b\}) = d(a,b) $; 
thus, for a Cauchy sequence $ \{a_n\}\subset X $, the sequence 
$ \{\{a_n\}\} $ converges to a closed and bounded subset of $ S\subset X $. 
For every element $ s\in S $ there holds
\begin{gather*}
d(s,a_n) = \dist(s,\{a_n\})\leq \delta_{\scr{H}} (S,\{a_n\})
\end{gather*}
thus $ s $ is the limit of the sequence $ \{a_n\} $. By uniqueness of the
limit $ S $ consist of a single point, thus $ (X,d) $ is complete. 
To prove the converse let $ \{A_n\} $ be a Cauchy sequence in $ \scr{H}(X) $ 
and
$ \var > 0 $; there exists $ n(\var) $ such that for every $ n\geq n(\var) $
\begin{gather*}
\delta_{\scr{H}} (A_{n(\var)},A_n) < \var/2;
\end{gather*}
given $ a\in A_{n(\var)} $ using induction we can build a sequence $ \{a_k\} $
and $ n_k \in\mathbb{N} $ such that
\begin{gather}
\label{sequence}
a_0 = a, \  a_k \in A_{n_k}, \ n_0 = n(\var), \ n_{k + 1} > n_k, 
\ d(a_{k + 1},a_k) < 2^{- (k + 2)}\var;
\end{gather}
then $ \{a_k\} $ is a Cauchy sequence in $ X $ and, since $ X $ is complete,
converges to a limit, say $ x $. Define $ L $ as the set of the elements that
are limits of sequences $ \{a_k\} $ such that $ a_k\in A_{n_k} $. The 
construction above shows that $ L $ is nonempty. 
We prove now that $ A_n $ converges to $ L $; first there exists
$ a_0 \in A_{n(\var)} $ such that
\begin{gather*}
\rho_{\scr{H}} (A_{n(\var)},L) < \var/8 + \dist(a_0,L);
\end{gather*}
let $ \{a_k\} $ be as in (\ref{sequence}) and call $ x $ its limit. Let $ k $
be such that $ d(a_k,a) < \var/8 $. We have
\begin{equation*}
\begin{split}
\rho_{\scr{H}} (A_{n(\var)},L) &< \var/8 + d(a_0,a_k) + d(a_k,x) < \var /4 + 
\sum_{j = 0} ^{k - 1} d(a_{j + 1},a_j) \\
&< \var/4 + \var\sum_{j = 2} ^{\infty} 2^{-j} < \var/2;
\end{split}
\end{equation*}
thus $ \rho_{\scr{H}} (A_n,L) \leq \rho_{\scr{H}} (A_n,A_{n(\var)}) + \rho_{\scr{H}} (A_{n(\var)},L) < \var $
for every $ n\geq n(\var) $. Similarly there exists $ x\in L $ such that
\begin{gather*}
\rho_{\scr{H}} (L,A_{n(\var)}) < \var/8 + \dist(x,A_{n(\var)});
\end{gather*}
by definition of $ L $ there exists a sequence $ a_k $ converging to $ x $ 
such that $ a_k \in A_{n_k} $. Choose $ k(\var) $ such that, for every 
$ k > k(\var) $, we have
\begin{gather*}
d(x,a_k) < \var/4, \ n_k > n(\var);
\end{gather*}
by the triangular inequality, for every $ n > n_{k(\var)}$, we have
\begin{gather*}
\rho_{\scr{H}} (L,A_{n(\var)}) < \var/4 + \dist(x,A_n) + 
\rho_{\scr{H}} (A_n, A_{n(\var)}) < \var,
\end{gather*}
thus $ \delta_{\scr{H}} (L,A_n) < \var $. This proves the completeness of 
$ \scr{H}(X) $. To conclude the proof observe that, since 
$ \rho_{\scr{H}} (L,A_n) $ is an infinitesimal sequence, given $ x\in L $ 
there exists an infinitesimal sequence $ \{\var_n\} $ and $ a_n $ such that 
\begin{gather*}
d(x,a_n) - \var_n < \dist(x,A_n) \leq \rho_{\scr{H}} (L,A_n);
\end{gather*}
taking the limit as $ n\rightarrow \infty $ we prove that 
$ \{a_n\} $ converges to $ x $.
\end{proof}
\section{Metrics on the Grassmannian}
Let $ (E,|\cdot|) $ be a Banach space. We define $ G(E) $ 
\glsadd{labGE}
as the set of the 
closed linear subspaces of $ E $, called \textsl{Grassmannian}. We 
define a complete metric on this set.
To a linear subspace $ Y\subset E $ we have the following
subsets of $ E $ associated to it:
\begin{align*}
D(Y) &= \set{y\in E}{|y|\leq 1},\\
S(Y) &= \set{y\in E}{|y| = 1}, \ (Y\neq 0) ;
\end{align*}
\glsadd{labDY}
\glsadd{labSY}
on $ G(E) $ we consider the metric induced by the inclusion of subsets
$ i\colon G(E)\hookrightarrow \scr{H}(E), \ Y\mapsto D(Y) $. We set
\begin{align*}
\rho (Y,Z) &= \rho_{\scr{H}} (D(Y),D(Z)), \\ 
\delta (Y,Z) &= \delta_{\scr{H}} (D(Y),D(Z)). 
\glsadd{labrYZ}
\glsadd{labdYZ}
\end{align*}
\begin{proposition}
The subset  $ i(G(E)) $ is closed in $ \scr{H}(E) $, hence $ \delta $ is 
complete.
\end{proposition}
\begin{proof}
Let $ Y_n $ be a sequence in $ G(E) $ such that $ D_n = D(Y_n) $ converges to 
$ D\subseteq E $, a nonempty, closed and bounded subset of $ E $. Let 
$ Z $ be the linear vector subspace generated by $ D $. We show that
$ D $ is the unit disc of the space $ Z $. In fact, we have the following 
properties:
\begin{itemize}
\item[p1)] $ 0\in D $;
\item[p2)] provided $ Z\neq\{0\} $ we have $ D\supset S(Z) $;
\item[p3)] $ D $ is \textsl{star-shaped} to $ 0 $, that is $ tx\in D $ for
every $ t\in [0,1] $ if $ x\in D $.
\end{itemize}
All these properties are consequences of Proposition \ref{hausdorff_complete}.
For instance the first follows in that $ 0\in D_k $ for every 
$ k\in\mathbb{N} $. For the second let $ z\in S(Z) $; since $ D $ generates
$ Z $ there are constants $ t_i $ such that
\begin{gather*}
z = t_1 y_1 + \dots + t_n y_n, \ \ y_i\in D.
\end{gather*}
Each of these elements are limits of a sequence $ y_{i,k}\in D_k $, hence,
for every $ k\in\mathbb{N} $, we have
\begin{align*}
z_k &= t_1 y_{1,k} + \dots + t_n y_{n,k}\in Y_k\\
z_k /|z_k| &= \hat{t}_1 y_{1,k} + \dots + \hat{t}_n y_{n,k}\in D_k;
\end{align*}
applying the Proposition \ref{hausdorff_complete} to the second sequence
we find  $ z\in D $. The proof of the third property is similar and
we omit it. From p1)--p3) it follows easily that $ D \supseteq D(Z) $:
given $ z\neq 0 $ in $ D(Z) $ the vector $ \hat{z} = z/|z|\in D $ and, 
since $ D $ is 
star-shaped, $ z\in D $. The inclusion $ D\subseteq D(Z) $ it is just the 
definition of $ Z $, hence $ D(Z) = D $.
To conclude the proof we show that $ Z $ is a closed
subspace of $ E $. Let $ \{z_n\} $ be a sequence converging to $ x\in E $;
if $ x = 0 $ clearly $ x\in Z $. If $ x\neq 0 $ for $ n $ large each term
of the sequence is nonzero. We write
\begin{gather*}
z_n = \hat{z}_n \cdot |z_n|, \ \hat{z}_n\in D;
\end{gather*}
since $ D $ is closed $ \hat{x}\in D $. Thus $ z = |x| \hat{x} $ belongs to
vector space generated by $ D $, hence $ z\in Z $. We have proved that
$ D = i(Z) $.
\end{proof}
Similarly we can consider the inclusion of spheres given by 
$ j\colon G(E)\setminus\{0\}\hookrightarrow\scr{H}(E), \ Y\mapsto S(Y) $ and 
define a metric on $ G(E)\setminus\{0\} $ as follows
\begin{align*}
\rho_S (Y,Z) &= \rho_{\scr{H}} (S(Y),S(Z)), \\
\delta_S (Y,Z) &= \delta_{\scr{H}} (S(Y),S(Z)), \text{ if } Y,Z\neq 0;
\glsadd{labrSYZ}
\glsadd{labdSYZ}
\end{align*}
we extend it to a metric on $ G(E) $ with $ \rho_S (\{0\},\{0\}) = 0 $ and
$ \rho_S (Y,\{0\}) = \rho_S (\{0\},Z) = 1 $. It is also called
\textsl{opening metric} (see \cite{Ber63}, \textsection 2). As above we have 
the following
\begin{proposition}
The subset  $ j(G(E)\setminus\{0\}) $ is closed in $ \scr{H}(E) $, hence 
$ \delta_S $ is complete.
\end{proposition}
The proof is similar to the previous one. It just takes to prove that limits
of sequences of spheres is a sphere.
\begin{proposition}
The metrics $ \delta_S $ and $ \delta_{\scr{H}} $ are equivalent. In particular
the inequalities
\begin{align*}
\rho_S (Y,Z)&\leq 2\rho(Y,Z) \\
\rho(Y,Z)&\leq \rho_S (Y,Z);
\end{align*}
hold.
\end{proposition}
\begin{proof}
To prove the first inequality we will use this fact: for any pair of
vectors $ x\in S(E) $ and $ y\in E\setminus\{0\} $ we have
$ |x - \hat{y} |\leq 2 | x - y | $ where $ \hat{y} = y/|y| $. 
Let $ Y,Z \neq\{0\} $ and $ \var > 0 $. There exists $ y\in S(Y) $ such that, 
for every $ z\in S(Z) $ and $ 0 < r \leq 1 $ there holds 
\begin{gather*}
\rho_{\scr{H}} (S(Y),S(Z))\leq \var + | y - z | = \var + 
| y - \widehat{rz} | \leq \var + 2 | y - rz |;
\end{gather*}
taking the infimum over $ (0,1]\times S(Z) $ we find 
\begin{gather*}
\rho_{\scr{H}} (S(Y),S(Z))\leq \var + 2 \dist (y,D(Z)\setminus\{0\});
\end{gather*}
since $ \rho_{\scr{H}} (S(Y),S(Z))\leq 1 < 2 $ we can write 
\begin{gather*}
\rho_{\scr{H}} (S(Y),S(Z))\leq 
2\min\{1,\var/2 + \dist (y,D(Z)\setminus\{0\})\};
\end{gather*}
since $ |y| = 1 $ the second member of the inequality becomes 
\begin{equation*}
\begin{split}
&2\min\{1,\var/2 + \dist (y,D(Z)\setminus\{0\})\}\\ 
\leq &2 \min\{\var/2 + |y|, \var/2 + \dist (y,D(Z)\setminus\{0\})\};
\end{split}
\end{equation*}
the latter is equal to 
\begin{gather*}
2(\var/2 + \dist(y,D(Z))\leq\var + 2 \dist(y,D(Z)).
\end{gather*}
Taking the supremum over $ S(Y) $ we obtain
\begin{gather*}
\rho_{\scr{H}} (S(Y),S(Z))\leq\var + 2 \rho_{\scr{H}} (S(Y),D(Z))\leq\var + 
2 \rho_{\scr{H}} (D(Y),D(Z)).
\end{gather*}
If $ Y = \{0\} $ and $ Z\neq 0 $ we have
$ \rho_{\scr{H}} (\{0\},S(Z)) = 1 = \delta_{\scr{H}} (D(Z),\{0\}) $, thus we 
have proved that $ \delta_S (Y,Z)\leq 2 \delta (Y,Z) $. \vskip .2em
We prove the second inequality in the case $ Y, Z\neq\{0\} $ first. Suppose
$ \rho (Y,Z)\neq 0 $ and pick $ \var > 0 $ such that
$ 0 < 2\var < \rho(Y,Z) $. There exists $ y\in D(Y) $ such 
that
\begin{gather*}
\rho(Y,Z) < \var/2 + \dist(y,D(Z));
\end{gather*}
in fact this implies $ y\neq 0 $. Set $ \hat{y} = y/|y| $; there exists 
$ \nu\in S(Z) $ such that 
\begin{gather*}
d(\hat{y},\nu) < \var/2 + \dist(\hat{y},S(Z)).
\end{gather*}
Hence the second term of the first inequality is bounded by $ d(y,|y| \nu) $
which is equal to $ |y| d(\hat{y},\nu) $, thus
\begin{equation*}
\begin{split}
\rho(Y,Z) &< \var/2 + |y| d(\hat{y},\nu)\leq\var/2 + 
d(\hat{y},\nu) \\ 
&< \var + \dist(\hat{y},S(Z))\leq\var + \rho_S (Y,Z).
\end{split}
\end{equation*}
If one among $ Y $ and $ Z $ is $ \{0\} $ we have 
$ \rho(Y,\{0\}) = 1 = \rho_S (Y,\{0\}) $. 
\end{proof}
By technical reasons we also define, for two closed subspaces $ Y, Z $
\begin{align*}
\rho_1 (Y,Z) = \sup_{y\in D(Y)} \dist(y,Z), \ 
\delta_1  (Y,Z) = \max\{\rho_1 (Y,Z),\rho_1 (Z,Y)\}.
\glsadd{labr1YZ}
\glsadd{labd1YZ}
\end{align*}
The triangular inequality does not hold for $ \rho_1 $
(see \cite{Ber63}, \textsection 3 for a counterexample). However the
\textsl{weakened triangular inequality} holds, that is
\begin{gather*}
\rho_1 (X,Z)\leq \rho_1 (Y,Z) (1 + \rho_1 (X,Y)) + \rho_1 (X,Y)
\end{gather*}
for every $ X, Y, Z $ (see \cite{Kat95}, Ch. IV, \textsc{Lemma} 2.2)
\footnote{The inequality allows to consider 
$ d_1 (X,Y) = \log (1 + \delta_1 (X,Y)) $ 
which is a metric and induces the same topology as the 
\textsl{neighbourhood topology} generated by $ \delta_1 $.}.
\begin{proposition}
\label{equivalent_metrics}
The topology generated by the neighbourhoods 
\begin{gather*}
\set{U(Y,r)}{Y\in G(E),\ r > 0}, \ U(Y,r) = \set{Z}{\rho_1 (Y,Z) < r }
\end{gather*}
is equivalent to the one induced by the Hausd\"orff metric of the discs. More 
precisely for every $ Y,Z $ 
\begin{gather*}
1/2\cdot\delta(Y,Z)\leq\delta_1 (Y,Z)\leq\delta(Y,Z).
\end{gather*}
\end{proposition}
\begin{proof}
Given $ y\in D(Y) $, $ \dist(y,Z)\leq\dist(y,D(Z)) $, then 
$ \delta_1 (Y,Z)\leq\delta(Y,Z) $. In order to prove the lower estimate 
suppose both $ Y, Z $ are different from the null space. Let $ y\in S(Y) $; 
for every $ z\in S(Z) $ and $ r > 0 $ we have 
\[
\dist(y,S(Z))\leq | y - z | = | y - \widehat{rz} |\leq 2 | y - rz |;
\]
taking the infimum over $ \mathbb{R}^+ \times S(Z) $ we obtain
\[
\dist(y,S(Z))\leq 2 \dist(y,Z\setminus\{0\}).
\] 
Since $ \dist(y,S(Z))\leq 2 $ we can write
\begin{equation*}
\begin{split}
\dist(y,S(Z))&\leq 2\min\{1,\dist(y,Z\setminus\{0\})\} = 
2\min\{|y|,\dist(y,Z\setminus\{0\})\} \\
&=2\dist(y,Z).
\end{split}
\end{equation*}
Then $ \delta_S(Y,Z)\leq 2\delta_1 (Y,Z) $. Since 
$ \delta(Y,Z)\leq\delta_S(Y,Z) $ the proof is complete. 
\end{proof}
We remark that the quantities introduced in this section such as $ \delta $,
$ \delta_1 $ and $ \delta_S $ induce the same topology on $ G(E) $. 
\begin{TAKEOUT}
However,
in literature, there are noticeable metrics that induce different topologies on
$ G(E) $. Of high interest it is the so called \textsl{Sch\"affer metric} or
\textsl{operator opening}. It is defined as follows
\begin{align*}
r_0 (X,Y) &= \inf\{\no{T - I}; T\in GL(E), TX = Y\}, \\
r(X,Y) &= \max\{r_0 (X,Y),r_0 (Y,Z)\}.
\glsadd{labr0XY}
\glsadd{labrXY}
\end{align*}
This metric induces a different topology than $ \delta $. In order
to show this, we need a couple of definitions.
\begin{definition}
A subspace $ X\subset E $ is said \textsl{complemented} if is closed
and there exists $ Y\subset E $ closed such that $ X\oplus Y = E $.
We denote by $ G_s (E) $ the subset of $ G(E) $ of complemented subspaces.
\end{definition}
Using the open map theorem, it can be shown that a closed subspace $ X $ is 
complemented if and only if there exists a bounded operator 
$ P\in\mathcal{L}(E) $ such that $ P\sp 2 = P $ and $ \ran P = X $. 
In the following example we use an argument of V.~I.~Gurarii and A.~S.~Markus 
of \cite{GM65}.
\begin{example}
By \textsc{Theorem} 4.1 of \cite{Ber63}, $ G_s (E) $ is a closed subset
of $ G(E) $ in the Sch\"affer topology, for every Banach space $ E $.
Let $ E $ and $ F $ be two Banach spaces, $ X\subset E $ a splitting closed 
subspace and $ Y \subset F $ a closed non-splitting subspace isomorphic to 
$ X $. In the Banach space $ E \oplus F $ the subspace $ \{0\}\oplus Y $ 
does not have a topological complement. Let $ T $ be an isomorphism of $ Y $ 
onto $ X $. Consider the family of subspaces
\begin{gather*}
Y(\lambda) = \set{(\lambda Ty,y)}{y\in Y}, \ \lambda\in\mathbb{R};
\end{gather*}
since $ T $ is bounded these are closed subspaces; in fact, given a projector
$ P $ with range $ X $, the linear operator 
$ P(\lambda) (v,w) = (\lambda Pv,T^{-1} Pv) $ is a projector with range 
$ Y(\lambda) $. However $ Y(\lambda) $ converges to $ Y(0) $ as 
$ \lambda\rightarrow 0 $ in the Hausd\"orff topology, in fact given
$ y\in D(Y) $ we have
\begin{gather*}
\dist((0,y),Y(\lambda)) = |\lambda||Ty|\leq |\lambda|\no{T}
\end{gather*}
hence $ \rho_1 (Y(0),Y(\lambda))\leq |\lambda|\no{T} $. Similarly it can be 
proved that $ \rho_1 (Y(\lambda),Y(0)) $ converges to zero as 
$ \lambda\rightarrow 0 $. Hence, a sequence in $ G_s (E\oplus F) $, 
namely $ \{Y(\lambda)\} $, converges to an uncomplemented subspace 
of $ E\oplus F $, therefore  $ G_s (E\oplus F) $ is not closed in 
$ G(E\oplus F) $ in the topology induced by $ \delta_1 $. In the next section 
we will prove that $ G_s (E) $ is open in the Hausd\"orff topology.
\end{example}
\end{TAKEOUT}
Since $ \delta $, $ \delta_S $ and $ \delta_1 $ induce the same topology we 
will choose time after time the one that most fits our settings.
\section{Properties of the Hausd\"orff topology}
Given Banach spaces $ E $ and $ F $ we denote by $ \mathcal{L}(E,F) $ 
\glsadd{labLEF}
the space of linear and bounded applications and use the abbreviate notation 
$ \mathcal{L}(E) $ to denote $ \mathcal{L}(E,E) $.
We call \textsl{general linear group} the 
set of invertible bounded operators of $ E $ endowed with the
topology of the norm and denote it by $ GL(E) $. 
\glsadd{labGLE}
In this section we show that 
the choice of the Hausd\"orff metric makes continuous some natural operations 
on $ G(E) $, such as the multiplication by an invertible operator
and the annihilator subspace $ Y\sp\perp $.
\begin{proposition}
\label{continuity_of_images}
Consider the set $ GL(E)\times G(E) $ with the topology induced by the product
metric $ \no{\cdot}\times\delta $. The action of $ GL(E) $ on $ G(E) $ given by
\begin{gather*}
GL(E)\times G(E)\longrightarrow G(E),\ \ (T,Y)\longmapsto T\cdot Y
\end{gather*}
is continuous.
\end{proposition}
\begin{proof}
We will prove that this map is locally Lipschitz. Fix 
$ T\in GL(E) $ and let $ Y,Z $ be two closed subspaces in $ G(E) $. 
Set $ Ty = y' \in D(TY) $ and $ r = \no{T^{-1}} $. Hence $ |y|\leq r $. 
Thus, by Proposition \ref{equivalent_metrics}, we have
\begin{equation*} 
\begin{split}
\label{lip1}
\dist (y',D(TZ)) &\leq 2\dist(y',TZ) = 2r\dist (y'/r, TZ)\leq 
2r \no{T}\dist(y/r, Z) \\
&\leq 2\no{T^{-1}}\no{T}\rho_1 (Y,Z)\leq 2\no{T^{-1}}\no{T}\rho (Y,Z)
\end{split}
\end{equation*}
hence 
\begin{gather}
\label{continuity_of_images:1}
\rho(TY,TZ)\leq 2\no{T^{-1}}\no{T}\rho (Y,Z).
\end{gather}
Now fix $ Y\in G(E) $, $ T $ and $ S $ invertible operators and 
$ y'\in D(TY) $. As above $ |y|\leq r $ and we have
\begin{equation*}
\begin{split}
\dist(y',D(SY))\leq 2\dist(y',SY)\leq 2 \no{T - S} |y|
\leq 2\no{T - S}\no{T^{-1}};
\end{split}
\end{equation*}
taking the supremum over $ D(TY) $ and switching $ T $ and $ S $ we find
\begin{gather}
\label{continuity_of_images:2}
\delta(TY,SY)\leq 2\no{T - S}\max\{\no{T^{-1}},\no{S^{-1}}\}.
\end{gather}
Now choose a point $ (T_0, Y_0)\in GL(E)\times G(E) $ and set 
$ r_0 = \no{T_0 ^{-1}} $; given $ \alpha < 1 $ we claim that
in the neighbourhood 
\begin{gather*}
U = B(T_0,\alpha r_0 ^{-1})\times G(E)
\end{gather*}
the map is Lipschitz. It is not hard to prove that for such radius the norm
of the inverse of every operator is bounded by a constant that depends only
on $ \alpha $ and $ r_0 $. More precisely, using Von Neumann series, it is
simple to find $ r_0/(1 - \alpha) $ as bound. 
Let $ (T,Y) $ and $ (S,Z) $ be two points in $ U $. Hence
\begin{equation*}
\begin{split}
\delta(TY,SZ)&\leq \delta(TY,SY) + 
\delta(SY,SZ) \\
&\leq 2 \max\{\no{T^{-1}},\no{S^{-1}}\}\no{T - S} + 
2\no{S}\no{S} ^{-1}\delta(Y,Z)\\
&\leq \frac{2r_0}{1 - \alpha}\no{T - S} + 
2\alpha r_0 ^{-1} \cdot \frac{r_0}{1 - \alpha}\delta(Y,Z) \\
&\leq\frac{2\max\{\alpha,r_0\}}{1 - \alpha}\Big(\no{T - S} + 
\delta(Y,Z)\Big).
\end{split}
\end{equation*}
\end{proof}
\begin{proposition}
 \label{trk:2}
  If $ \rho_1 (Y,Z) < 1 $ and $ Z\subseteq Y $ then $ Z = Y $.
\end{proposition}
\begin{proof}
If $ Y $ is the null space the proof is trivial. Otherwise let 
$ \rho (Y,Z) = 1 - \var_0 $ and suppose $ S(Y)\setminus Z $ is not empty and
contains an element, say $ y $. Let $ z \in Z $ be such that
\begin{gather*}
\dist(y,Z)\geq | y - z | - \var_0 /2;
\end{gather*}
define $ y_0 = z - y $. Since $ Z\subseteq Y $, $ y_0\in Y $. Thus
$ \dist(\hat{y_0},Z) \geq 1 - \var_0 /2 $, thus
\begin{equation*} 
 \begin{split}
   1 - \var_0 = \rho (Y,Z)\geq \dist (\hat{y_0},Z)\geq 1 - \var_0 /2 
 \end{split} 
 \end{equation*}
which is impossible, then $ Y\subset Z $ and $ Y = Z $.
\end{proof}
\begin{definition}
We denote by $ E^* $ the space $ \mathcal{L}(E,\R) $.
\glsadd{labEst}
It is called \textsl{topological dual} of $ E $ and its elements are
called \textsl{functionals}. 
For any subset $ S\subset E $ we define
\[
S^{\bot} = \{\xi\in E\sp *:\langle\xi,s\rangle = 0\,\forall s\in S\},
\]
\glsadd{labSperp}
and call it \textsl{annihilator} of $ S $.
\end{definition}
The annihilator is a linear, closed  subspace of $ E\sp * $, 
and it is well-behaved with respect to the topology of $ G(E) $.
\begin{proposition}
\label{orthogonal_map}
Given two closed subspaces $ Y $, $ Z $ and $ Y^{\bot} $, $ Z^{\bot} $
its annihilators, we have $ \rho_1 (Y,Z) = \rho_1 (Z^{\bot},Y^{\bot}) $.
\end{proposition}
\begin{proof}
We prove that, for any closed subspace $ Y $, a functional $ \xi\in E^* $ and 
$ x\in E $, the equalities
\begin{align}
\dist(\xi,Y^{\bot}) &= \sup_{D(Y)} | \bra \xi,y \ket | = 
|\res{\xi}{Y}|, \label{orthogonal_map:1}\\ 
\dist(x,Y) &= \sup_{D(Y^{\bot})} | \bra \eta,x\ket | \label{orthogonal_map:2}
\end{align}
hold. The proof of both uses Hahn-Banach theorems of extension of 
functionals, see \cite{Bre83} details. Given $ \var $ there exists 
$ \orl{y}\in D(Y) $ such that, for every $ \eta\in Y^{\bot} $, we can write
\begin{gather*}
|\res{\xi}{Y} | < \var + \bra\xi,\orl{y}\ket = 
\var + \bra\xi - \eta,\orl{y}\ket\leq\var + |\xi - \eta|; 
\end{gather*}
taking the infimum over $ D(Y^{\bot}) $ we get 
$ |\res{\xi}{Y}|\leq\dist(\xi,Y^{\bot}) $. Conversely, given a functional
$ \xi $, by Hahn-Banach, there exists an extension $ \xi_1 $ of 
$ \res{\xi}{Y} $ such that $ |\xi_1| = |\res{\xi}{Y}| $. 
Thus $ \eta = \xi - \xi_1 $ annihilates $ Y $ and we can write
\begin{gather*}
\dist(\xi,Y^{\bot})\leq | \xi - \eta | = | \xi_1 | = | \res{\xi}{Y} |.
\end{gather*}
We prove the second equality. Let $ \var > 0 $. There exists 
$ \eta_1 \in D(Y^{\bot}) $ such that, for every $ y\in Y $
\begin{gather*}
\sup_{D(Y^{\bot})} | \bra \eta,x\ket | < \var + |\bra\eta_1,x\ket| = 
\var + |\bra\eta_1,x - y\ket|\leq\var + | x - y |;
\end{gather*}
taking the infimum over $ Y $ we find
\begin{gather*}
\sup_{D(Y^{\bot})} | \bra \eta,x\ket | \leq\var + \dist(x,Y).
\end{gather*}
To prove the opposite inequality we distinguish two cases. 
If $ x\in Y $ the proof is trivial, because both terms of
(\ref{orthogonal_map:2}) are zero. Suppose $ x\not\in Y $. 
Let $ 0\leq\alpha < 1 $. There exists $ y_{\alpha}\in Y $ such that
\begin{gather*}
\alpha | x - y_{\alpha} | < \dist (x,Y)\leq | x - y_{\alpha} |;
\end{gather*}
since $ x - y_{\alpha}\not\in Y $ we can define a functional  
$ \eta_{\alpha} $ such that its restriction to $ Y $ is zero and 
$ \bra\eta_{\alpha},x - y_{\alpha}\ket = \alpha | x - y_{\alpha} | $. 
By Hahn-Banach for every $ \alpha $ there exists an extension 
$ \tilde{\eta}_{\alpha} $ of $ \eta_{\alpha} $ such that 
$ |\tilde{\eta}_{\alpha}| = |\eta_{\alpha}| $. It is clear by its definition
that $ \tilde{\eta}_{\alpha}\in Y^{\bot} $. Consider 
$ z = \lambda (x - y_{\alpha}) + y $. We have 
\begin{equation*}
\begin{split}
| z | &= |\lambda| \left| x - y_{\alpha} + \frac{y}{\lambda}\right|
\geq |\lambda |\dist(x,Y)\geq \alpha |\lambda | | x - y_{\alpha} | \\
&\geq |\lambda | |\bra\eta_{\alpha}, x - y_{\alpha}\ket| = 
|\bra\eta_{\alpha},z\ket| 
\end{split}
\end{equation*}
then $ |\eta_{\alpha}|\leq 1 $ and $ \tilde{\eta}_{\alpha}\in D(Y^{\bot}) $. 
Therefore
\begin{gather*}
\alpha\dist(x,Y)\leq \alpha| x - y_{\alpha} | = 
| \bra\tilde{\eta}_{\alpha}, x - y_{\alpha}\ket |  = 
|\bra\tilde{\eta}_{\alpha},x\ket| 
\leq\sup_{D(Y^{\bot})} | \bra\tilde{\eta},x\ket|.
\end{gather*}
The equality is proved as $ \alpha\rightarrow 1 $. Now we can prove the
equality claimed in the statement. We have
\begin{equation*}
\rho_1 (Y,Z) = \sup_{D(Y)} \dist(y,Z) = 
\sup_{D(Y)} \sup_{D(Z^{\bot})} | \bra \xi,y \ket |
\end{equation*}
by (\ref{orthogonal_map:2}). Here we switch the order of the supremums. 
By (\ref{orthogonal_map:1}) the last term of the equality is
\begin{equation*}
\sup_{D(Z^{\bot})} \sup_{D(Y)} | \bra \xi,y \ket |
= \sup_{D(Z^{\bot})} \dist (\xi, Y^{\bot}) = \rho(Z^{\bot},Y^{\bot}).
\end{equation*}
\end{proof}
\begin{corollary}
The map $ G(E)\rightarrow G(E^*) $ that associates a subspace with its
annihilator is continuous.
\end{corollary}
\section{The complemented Grassmannian}
We define the \textsl{complemented Grassmannian}, $ G_s (E) $ as
the subset of $ G(E) $ of complemented subspaces, and the space of projectors.
We prove that the former is an open subset and the latter is homeomorphic to 
the \textsl{splitting} space $ \Splt(E) $, defined as the family of pairs 
$ (X,Y) $ which are the complement of each other. We prove that 
$ \mathcal{L}(X,Y) $ is homeomorphic to $ G_s (X\times Y) $ for every $ X,Y $ 
Banach spaces.
\begin{definition}
A closed subspace $ Y\in G(E) $ is said \textsl{complemented} or that
\textsl{splits} if there exists $ Z\in G(E) $ such that $ Y\oplus Z = E $,
called \textsl{complement} of $ Y $.
We call \textsl{projector} a bounded operator $ P\in\mathcal{L}(E) $ 
such that $ P\sp 2 = P $. We introduce the sets
\begin{align*}
G_s (E) &= \set{Y\in G(E)}{Y\text{ is complemented }};\\
\glsadd{labGsE}
\mathcal{P}(E) &= \set{P\in\mathcal{L}(E)}{P\sp 2 = P}
\glsadd{labPE}
\end{align*}
and call them \textsl{complemented Grassmannian} and 
\textsl{space of projectors}, respectively. We will also refer to these
spaces as topological subspaces of $ G(E) $ and $ \mathcal{L}(E) $,
respectively.
\end{definition}
By the open mapping theorem, $ Y\in E $ is complemented if and only
if there exists $ P\in\mathcal{P}(E) $ such that $ \ran P = Y $.
Moreover, to each complement $ Z $, corresponds a unique bounded 
$ P\in\mathcal{P}(E) $ such that 
\[
\ran P = Y,\ \ \ker P = Z.
\]
\begin{definition}
Let $ Y,Z,P $ as above. We call $ P $ 
the \textsl{projector onto $ Y $ along $ Z $} and denote it by $ P(Y,Z) $.
\glsadd{labPYZ}
\end{definition}
Unless $ E $ is an Hilbert space $ G_s (E) \subsetneq G(E) $, refer
\cite{LT71}. Our aim is to prove that $ G_s (E) $ is an open subset of 
$ G(E) $. For this purpose we need to introduce the notion of 
\textsl{minimum gap} between closed spaces 
(see also \cite{Kat95}, Ch. IV, \S 4).
We recall that, for any closed subspace $ Y\in G(E) $, the quotient
space $ E/Y $ 
\glsadd{labE/Y}
is endowed with the norm $ |x + Y | = \dist(x,Y) $ that makes
it a Banach space called \textsl{quotient space}. Moreover the projection to 
the quotient is a bounded operator between two Banach spaces.
\begin{definition}[The minimum gap]
\label{gap}
Let $ Y $ and $ Z $ be two closed subspaces. Set
\begin{gather*}
\gamma (Y,Z) = \inf_{Y\setminus Z} \frac{\mathrm{dist}(y,Z)}%
{\mathrm{dist}(y,Y\cap Z)}
\end{gather*}
if $ Y\neq 0 $, $ \gamma (Y,Z) = 1 $ otherwise. We define the \emph{gap} by
\begin{gather*}
\hat{\gamma}(Y,Z) = \min\{\gamma(Y,Z),\gamma(Z,Y)\}.
\end{gather*}
\glsadd{labgYZ}
\glsadd{labgYZh}
\end{definition}
\begin{lemma}{\rm(cf. \cite{Kat95}, \textsc{Theorem} 4.2, Ch. IV)}
\label{closed_sum}
Let $ Y $ and $ Z $ be closed subspaces of $ E $. Then $ Y + Z $ is
closed in $ E $ if and only if $ \gamma (Y,Z) > 0 $.
\end{lemma}
\begin{proof}
Suppose both spaces are different from $ \{0\} $. We prove the statement
when $ Y\cap Z = \{0\} $. 
Suppose $ X = Y \oplus Z $ is closed and call $ P $ the projector onto 
$ Y $ along $ Z $. Since $ Y\neq\{0\} $ the projector is not zero. Let 
$ x = y + z $. Then
\begin{equation}
\label{gap_norm}
\begin{split}
\no{P} = \sup_{y + z\neq 0}\frac{| y |}{| y + z |} =
\sup_{y\neq 0} \frac{| y |}{\dist (y,Z)};
\end{split}
\end{equation}
taking the inverses in the equation we find then 
$ \no{P}^{-1} = \gamma(Y,Z) $. Suppose, conversely, that 
$ \gamma (Y,Z) > 0 $. Let $ \{y_n\}\subset Y $ and $ \{z_n\}\subset Z $ 
be sequences such that $ x_n = y_n + z_n \ra x\in E $. 
If the sequence $ \{x_n\} $ has a constant 
subsequence, then $ x\in Z $, since both $ \{y_n\} $ and $ \{z_n\} $ are
constants. Otherwise, up to extracting a subsequence we 
can suppose that $ x_n \neq x_m $ whenever $ n\neq m $. Then we can write
\begin{equation*}
\begin{split}
| y_n - y_m | =&\frac{| y_n - y_m |}{| x_n - x_m |} \cdot | x_n - x_m |\leq
\frac{| y_n - y_m |}{\dist(y_n - y_m,Z)} \cdot | x_n - x_m | \\
\leq&\frac{| x_n - x_m |}{\gamma(Y,Z)};
\end{split}
\end{equation*}
since the last term of the inequality is a Cauchy sequence, $ \{y_n\} $ 
(and thus $ \{z_n\} $) converges and $ x = \lim y_n + \lim z_n \in X $. 
Since both $ Y $ and $ Z $ are closed $ x\in Y + Z $.  
For the general case consider the quotient space $ E/(Y\cap Z) $ and 
call $ \pi $ the projection onto the quotient.
Let $ \tilde{Y} = \pi(Y) $ and $ \tilde{Z} = \pi(Z) $; 
these are closed subspaces of $ F $ because $ \pi $ maps closed subspaces of 
$ E $ containing $ \ker\pi $ onto closed subspaces.
Moreover $ \gamma (\tilde{Y},\tilde{Z}) = \gamma (Y,Z) $, in fact
\begin{gather*}
\dist (\tilde{y},\tilde{Z}) = \inf_{z\in Z} \dist (y - z,Y\cap Z) =
\dist (y,Z).
\end{gather*}
The proof carries on as follows: suppose $ Y + Z $ is a closed subspace.
Then $ \pi(Y + Z) = \tilde{Y} + \tilde{Z} $ is closed in the quotient space. 
The space $ \pi(Y) $ and $ \pi(Z) $ have null intersection thus, by the
first part of the proof, $ \gamma(\pi(Y),\pi(Z)) > 0 $ hence 
$ \gamma(Y,Z) > 0 $. The converse is completely similar.
\end{proof}
In the next proposition we prove that $ G_s (E) $ is an open subset of 
$ G(E) $. A proof of this is due to E.~Berkson, 
\cite{Ber63} \textsc{Theorem 5.2}, when $ G(E) $ has the topology induced by
the Sch\"affer metric. However the same proof works for the metric of
geometric opening.
\begin{proposition}
\label{GE_is_open}
Let $ X\in G_s (E) $ be a proper subspace of $ E $. Let 
$ Y\in G_s (E) $ be a complement of $ X $. Denote by $ P $ the 
projector onto $ X $ along $ Y $. If $ Z\in G(E) $ and
\begin{align}
\rho_S (X,Z) &< \gamma(X,Y), \label{GE_is_open:1} \\
\rho_S (Z,X) &< \gamma(Y,X)  \label{GE_is_open:2}
\end{align}
then $ Z\oplus Y = E $. If $ Q $ is the projector onto $ Z $ along $ Y $ the
operator $ I + Q - P $ is invertible and maps $ X $ onto $ Z $. Moreover 
\begin{gather}
\label{GE_is_open:tesi}
\no{P - Q}\leq \no{I - P}\frac{\no{P} \rho_S (X,Z)}%
{1 - \no{P}\rho_S (X,Z)} 
\end{gather}
\end{proposition}
\begin{proof}
First we prove that $ Z\cap Y = \{0\} $ and $ Z + Y $ is closed. In fact,
given $ y\in Z\cap Y $, $ | y | = 1 $, from (\ref{GE_is_open:2}) we can write
\begin{gather*}
\dist(y,X)\leq\dist(y,S(X))\leq\rho_S (Z,X) < \gamma(Y,X)\leq \dist(y,X)
\end{gather*}
which is absurd. To prove that $ Y + Z $ is closed it will suffice to show
that $ \gamma(Z,Y) > 0 $, by Proposition \ref{closed_sum}. Let $ z\in S(Z) $ 
and $ 1 < \alpha $; there exists $ x_{\alpha}\in S(X) $ such that 
\begin{gather*}
\alpha\dist(x_{\alpha},Z)\geq | x_{\alpha} - z |;
\end{gather*}
for any $ y\in Y $ we can write
\begin{equation*}
\begin{split}
| z - y |&\geq | x_{\alpha} - y | - | x_{\alpha} - z |\geq
\dist(x_{\alpha},Y) - \alpha\dist(x_{\alpha},Z) \\
&\geq\gamma(X,Y) - \alpha\rho_S (X,Z);
\end{split}
\end{equation*}
if $ \alpha - 1 $ is small the last term is positive. Taking the infimum over 
$ Y $ and $ S(Z) $ we get $ \gamma(Z,Y) > 0 $, hence $ Y + Z $ is closed.
We prove now that $ Z + Y = E $ by showing that $ X\subseteq Z + Y $.
Let $ x\in X $ and $ \lambda > 1 $; by induction we can build two sequences 
$ \{x_n\}\subset X $, $ \{z_n\}\subset S(Z) $,  such that
\begin{gather}
x_0 = x,\ | x_n - z_n |\leq\lambda\rho_S (X,Z) |x_n|, 
\ x_{n + 1} = P(x_n - z_n)\label{GE_is_open:3}\\
x = \sum_{k = 0} ^n (z_k + y_{k + 1}) + x_{n + 1}\label{GE_is_open:4}
\end{gather}
where $ y_{k + 1} = (I - P) (x_k - z_k) $. For every $ k\in\mathbb{N} $ we also
have, by induction
\begin{gather}
\label{GE_is_open:5}
|x_k |\leq (\lambda\no{P}\rho_S (X,Z))^k |x_0|;
\end{gather}
by (\ref{gap_norm}) $ \no{P} = \gamma(X,Y)^{-1} $ and (\ref{GE_is_open:1}) 
allows us to choose a positive $ \lambda $ such that
\[
\lambda\rho_S (X,Z)\gamma(X,Y)^{-1} < 1. 
\]
Then $ x_k \rightarrow 0 $. Taking the limit in (\ref{GE_is_open:4}) we find
$ x\in\orl{Z + Y} = Z + Y $. The operator $ I + Q - P $ 
maps $ X $ into $ Z $ and fixes $ Y $. Since $ Q $ and $ P $ project along
the same space a direct computation shows that its inverse is $ I - Q + P $,
thus $ (I + Q - P) X = Z $. Choose $ \lambda > 1 $ and $ v\in E $. We apply
the construction made above to $ x = Pv $. By (\ref{GE_is_open:3}) we have
\begin{equation*}
\begin{split}
|y_{k + 1}|&\leq\no{I - P}|x_k - y_k|\leq\lambda\no{I - P}\rho_S (X,Z)|x_k|\\
&\leq\frac{\no{I - P}}{\no{P}} (\lambda\no{P}\rho_S (X,Z))^{k + 1} |x|
\end{split}
\end{equation*}
by (\ref{GE_is_open:4}). If $ \lambda\no{P}\rho_S (X,Z) < 1 $ we have
\begin{gather*}
|(P - Q) P v |\leq \sum_{k = 0} ^{\infty} |y_{k + 1}|
\leq\no{I - P}\frac{\lambda\rho_S (X,Z)}%
{1 - \lambda\no{P}\rho_S (X,Z)} |Pv|.
\end{gather*}
Letting $ \lambda\rightarrow 1 $, since $ (P - Q) v = (P - Q) Pv $, we obtain
(\ref{GE_is_open:tesi}).
\end{proof}
\begin{corollary}
The subset $ G_s (E) $ is open in $ G(E) $ with the topology induced by the
geometric opening.
\end{corollary}
\begin{TAKEOUT}
As we showed in the preceding example, there are Banach 
spaces where the subset of splitting subspaces is not closed in the 
Grassmannian of closed subspaces. For sake of completeness we provide an 
example of Banach $ E $ (non-isomorphic to a Hilbert)
where $ G_s (E) $ is both open and closed. This is the case of
$ l^{\infty} (\mathbb{C}) $. It is known that the closed and splitting 
subspaces of $ l^{\infty} (\mathbb{C}) $ are the non-separable ones. 
If a closed subspace $ X $ is limit
of a sequence of closed and splitting subspaces in a ball centered in $ X $
of radius smaller than $ 1/2 $ there are splitting subspaces. 
We use now a result of E.~Berkson
(\textsc{Theorem} 2.2 of \cite{Ber63}): if $ \delta_S (X,Y) < 1/2 $ then 
the minimum cardinality of a dense subset of $ X $ is the same as that of 
$ Y $. Thus, if $ Y $ splits it is not separable, hence $ X $ is not even 
separable, thus $ X $ splits.
\end{TAKEOUT}
Another consequence of Proposition \ref{GE_is_open} is the following
\begin{proposition}
\label{continuity_of_graph}
Let $ X $ and $ Y $ be Banach spaces. The map 
$ \mathcal{L}(X,Y)\rightarrow G_s (X\times Y) $ that associates an operator 
with its graph is a homeomorphism with the open subset 
\[
\set{Z\in G_s (X\times Y)}{Z\oplus\{0\}\times Y = X\times Y}.
\]
\end{proposition}
\begin{proof}
Since $ S $ is bounded $ \graph(S) $ is closed and it is a topological 
complement of $ \{0\}\times Y $, then it is an element of $ G_s (X\times Y) $.
Hence the map is well defined. 
For any $ S\in\mathcal{L}(X,Y) $ define $ \check{S} (x,y) = (x, y + S x) $; it is
an invertible operator. Since $ \graph(S) = \check{S}(X\times \{0\}) $,
by Proposition \ref{continuity_of_images} the map is continuous and injective.
To prove that it is also open let $ \graph(S) $ be a point in the image. We
show that there exists $ r > 0 $ such that 
$ B(\graph(S),r)\subset\im(\graph) $, with the metric induced by $ \delta_S $.
We choose
\begin{gather*}
r < \hat{\gamma}(\graph(S),\{0\}\times Y);
\end{gather*}
given $ Z\in B(\graph(S),r) $, by Proposition \ref{GE_is_open} $ Z $ is a 
topological complement of $ \{0\}\times Y $. Thus, for every 
$ x\in X\times\{0\} $ there exists a unique $ z\in Z $ such that $ Pz = x $.
Then $ P $ maps isomorphically $ Z $ onto $ X $ and
\begin{gather*}
\graph((I - P) \res{P}{Z} ^{-1}) = Z
\end{gather*}
which concludes the proof.
\end{proof}
Given $ X\in G_s (E) $ and $ Y $ such that $ X\oplus Y $ we can identify
$ X $ with $ X\times\{0\} $, the graph of the null operator. The subset
of topological complements of $ X\times\{0\} $ is open and homeomorphic
to the Banach space $ \mathcal{L}(X,Y) $ by Proposition 
\ref{continuity_of_graph}. Thus we have proved that
\begin{corollary}
$ G_s (E) $ is a topological Banach manifold.
\end{corollary}
\begin{definition}
Define the \textsl{space of splittings} the subset 
\[
\set{(X,Y)\in G_s (E)\times G_s (E)}{X\oplus Y = E} 
\]
endowed with the product metric $ \delta_S\times\delta_S $ and denote it by 
$ \Splt(E) $.
\end{definition}
We can associate to a pair $ (X,Y)\in\Splt(E) $ the projector $ P(X,Y) $.
\begin{proposition}
\label{continuity_of_splits}
The map $ P\colon\Splt(E)\rightarrow\mathcal{P}(E), \ (X,Y)\mapsto P(X,Y) $
is a homeomorphism with its image. 
\end{proposition}
\begin{proof}
First observe that $ P $ is a bijection. Its inverse maps $ P $ to
$ (\ran P,\ker P) $. Suppose $ (X_0,Y_0) = (\ran P_0,\ker P_0) $ and 
$ \var > 0 $. We prove that there exists $ \delta > 0 $ such that 
$ P(B((X_0,Y_0),\delta))\subseteq B(P_0,\var) $. 
More precisely, in a suitable neighbourhood of $ (X_0,Y_0) $, for every 
$ (X,Y) $ we can choose continuously an invertible operator $ U $
that maps $ X_0 $ and $ Y_0 $ onto $ X $ and $ Y $ respectively and
\begin{gather}
\label{continuity_of_splits:4}
\no{U P_0 U^{-1} - P_0} < \var.
\end{gather}
This completes the proof because $ U P_0 U^{-1} $ is a projector with range 
$ X $ and kernel $ Y $. Thus $ U P_0 U ^{-1} $ is the projector onto $ X $ 
along $ Y $. We construct $ U $ and $ \delta $ as follows: as first step we
choose $ \delta_0 < \hat{\gamma}(X_0,Y_0) $. If 
$ \delta_S (X_0,X) < \hat{\gamma}(X_0,Y_0) $ the Proposition \ref{GE_is_open} 
provides us with an operator $ T = I + P(X,Y_0) - P_0 $ and a positive 
constant $ c $ such that
\begin{gather}
\label{continuity_of_splits:5}
T X_0 = X, \ T Y_0 = Y_0, \ \no{T - I} < c \delta_S(X_0,X).
\end{gather}
As second step we build another invertible operator $ S $ that maps 
$ Y_0 $ onto $ T^{-1} Y $ and fixes $ X_0 $, applying the same Proposition.
Hence $ U = T S $ fits our request. This can be done if, for instance, 
$ \delta_S (T^{-1} Y,Y_0) < \hat{\gamma}(X_0,Y_0) $. Using the estimate 
(\ref{continuity_of_images:1}) we write
\begin{gather}
\label{continuity_of_splits:6}
\delta_S (T^{-1} Y,Y_0) = \delta_S(T^{-1} Y, T^{-1} Y_0)
\leq 2\no{T}\no{T^{-1}}\delta_S(Y_0,Y);
\end{gather}
if $ c\delta_S(X,X_0) < 1 $, using Von Neumann series, we can estimate 
$ \no{T^{-1}} $ with $ 1/(1 - \no{I - T}) $. Then the 
(\ref{continuity_of_splits:6}) becomes
\begin{gather}
\label{continuity_of_splits:7}
\delta_S (T^{-1} Y,Y_0)\leq 2\frac{1 + c\hat{\gamma}(X_0,Y_0)}%
{1 - c\hat{\gamma}(X_0,Y_0)} 
\delta_S(Y_0,Y).
\end{gather}
Then, if we choose
\begin{gather}
\label{continuity_of_splits:8}
\delta_S(Y_0,Y) < \frac{1 - c\hat{\gamma}(X_0,Y_0)}%
{2(1 + c\hat{\gamma}(X_0,Y_0))} 
\end{gather}
we have $ \delta_S(T^{-1} Y,Y_0) < \hat{\gamma}(X_0,Y_0) $ and it is possible
to apply \ref{GE_is_open} and such operator $ S $ exists.
By (\ref{GE_is_open:tesi}) and (\ref{continuity_of_splits:7}) we can write
the (\ref{continuity_of_splits:8}) as
\begin{gather}
\label{continuity_of_splits:9}
\no{I - S} < k\delta_S(Y_0,Y).
\end{gather}
If we choose $ \delta_1 = \min\{\delta_0,1,1/8k,1/4c\} $, using 
(\ref{continuity_of_splits:5}) and (\ref{continuity_of_splits:9}) we can
estimate the norm of the operator $ U - I $ from above by
\begin{equation}
\label{continuity_of_splits:10}
\begin{split}
\no{T(S - I) + T - I}&\leq k(1 + c\delta_S (X_0,X)) \delta_S (Y_0,Y) + 
c\delta_S (X_0,X) \\
&\leq 2k \delta_S (Y_0,Y) + c \delta_S (X_0,X)\leq 1/2.
\end{split}
\end{equation}
We can write $ U P_0 U^{-1} - P_0 $ as 
$ (U - I) P_0 U^{-1} + P_0 (U^{-1} - I) $. By (\ref{continuity_of_splits:10})
the norm of $ I - U $ is strictly smaller than $ 1 $. Hence $ \no{U^{-1}} $
can be estimated by $ 1/(1 - \no{I - U}) $ which is smaller than $ 2 $,
still by (\ref{continuity_of_splits:10}). Then 
\begin{equation*}
\begin{split}
\no{U P_0 U^{-1} - P_0}&\leq 4\no{P_0}\no{I - U}
\leq 4\no{P_0} (2k \delta_S (Y_0,Y) + c \delta_S (X_0,X)).
\end{split}
\end{equation*}
Finally we set
\begin{gather*}
\delta = \min\left\{\delta_1,\frac{\var}{4(2k + c)\no{P_0}}\right\}.
\end{gather*}
The continuity of the inverse follows at once: given $ P,Q\in\al{P}(E) $
\begin{equation*}
\begin{split}
\delta_S\times\delta_S((\ran Q,\ker Q),(\ran P,\ker P)) =& 
\delta_S(\ran Q,\ran P) + \delta_S(\ker Q,\ker P) \\
\leq &4\no{P - Q};
\end{split}
\end{equation*}
in fact is Lipschitz.
\end{proof}
\section{Compact perturbation of subspaces}
We define a relation of 
\textsl{compact perturbation} for pairs of closed subspaces and
an integer that we call \textsl{relative dimension}. We prove that it
is well-behaved with respect to the Fredholm index of a pair of subspaces
and that kernels and images of two operators with compact difference are
compact perturbation of each other. When both spaces are complemented, the 
relation of compact perturbation is equivalent to require that for
every pair of projectors $ (P,Q) $, the operators $ (I - P)Q $ and
$ (I - Q)P $ are compact. That generalizes an existing definition
in \cite{ZL99} when $ P - Q $ is compact.\vskip .2em
We need some preliminary concepts about Fredholm operators and
compact operators. We recall some basic definitions and state some useful 
results about Fredholm operators and Fredholm pairs. For more details we 
refer to Appendix B. \vskip .2em
Given a linear operator $ T\colon E\rightarrow F $ we can always
consider the vector spaces $ \ker T $ and $ F/\ran T $. We denote the second
by $ \coker T $.
\begin{definition}
A bounded operator $ T\in\al{L}(E,F) $ is called \emph{semi-Fredholm} if and 
only if $ \ran T $ is closed and either $ \ker T $ or $ \coker T $ has finite 
dimension. We define its \emph{index} as
\[
\ind(T) = \dim\ker T - \dim\coker T.
\]
When only one between $ \ker T $ and $ \coker T $ has finite dimension we 
will write, for short, $ \ind(T) = \infty $ or $ \ind(T) = -\infty $, 
respectively.
If both spaces have finite dimension we say that $ T $ is \emph{Fredholm} and
the index is a integer.
\glsadd{labfind}
\end{definition}
\begin{definition}
A pair $ (X,Y) $ of closed and linear subspaces is said  
\textsl{semi-Fredholm} if and only if $ X + Y $ is closed and either 
$ X\cap Y $ or $ E/(X + Y) $ has finite dimension. We define its \emph{index}
as
\[
\ind(X,Y) = \dim X\cap Y - \codim X + Y.
\glsadd{labindXY}
\]
When only one between $ X\cap Y $ and $ X + Y $ has finite dimension we 
will write $ \ind(X,Y) = \infty $ or $ \ind(X,Y) = -\infty $, 
respectively.
If both $ X\cap Y $ and $ X + Y $ have finite dimension the pair is said 
\emph{Fredholm}.
\end{definition} 
There is a strict relation between (semi)Fredholm pairs and (semi)Fredholm 
operators. Precisely, given closed subspaces $ (X,Y) $ the operator
\begin{gather}
\label{pairing_operator}
F_{X,Y} \colon X\times Y\rightarrow E,\quad (x,y)\mapsto x - y 
\end{gather}
is (semi)Fredholm if and only if $ (X,Y) $ is (semi)Fredholm and 
$ \ind(X,Y) = \ind(F_{X,Y}) $. Given Banach spaces $ E, F $ we denote by 
$ \mathcal{L}_c (E,F) $ the set of compact operators.
\glsadd{labLcEF}
\begin{definition}
An operator $ T\colon E\rightarrow F $ is said \emph{essentially invertible}
if and only if there exists $ S\in\mathcal{L} (F,E) $ and compact operators
$ K\in\mathcal{L}_c (E) $, $ H\in\mathcal{L}_c (F) $ such that
\begin{align*}
S\circ T &= I_E + K \\
T\circ S &= I_F + H.
\end{align*}
\end{definition}
It is not hard to prove that an operator is Fredholm if and only if is
essentially invertible, see Proposition \ref{essential_inverse}. We end
this section with a strong result of perturbation theory.
\begin{theorem}
\label{perturbation}
{\rm(cf. \cite{Kat95}, Ch. IV, \S 5).} Let $ (X,Y) $ be a
semi-Fredholm pair. Then there exists $ \delta > 0 $ such that, 
$ \delta_S (X',X) < \delta $, $ \delta_S (Y',Y) < \delta $ implies that
$ (X',Y') $ is semi-Fredholm and $ \ind (X',Y') = \ind (X,Y) $.
\end{theorem}
\begin{definition}\rm{(cf. \textsc{Definition} 1.1 of \cite{AM01}).}
Two closed subspaces $ X $ and $ Y $ of a Hilbert spaces are
compact perturbation one of each other if the orthogonal projections $ P_X $
and $ P_Y $ have compact difference. This implies that $ X\cap Y^{\bot} $
and $ X^{\bot} \cap Y $ are finite dimensional subspaces and the relative
dimension is defined as
\begin{gather*}
\dim(X,Y) = \dim(X\cap Y^{\bot}) - \dim (X^{\bot}\cap Y).
\end{gather*}
\end{definition}
Our first aim is to define the relative dimension for pairs of closed subspaces
that do not necessarily split.
\begin{definition}[commensurability]
\label{commensurability}
Let $ X, Y\in G(E) $. The pair $ (X,Y) $ is said 
\emph{commensurable} if there are $ F,G\in\mathcal{L}(E) $ such that
\begin{align}
G X&\subset Y, \ G_{|X} = (I + H)_{|X}\label{commensurability:1},\\
F Y&\subset X, \ F_{|Y} = (I + K)_{|Y}\label{commensurability:2}
\end{align}
where $ H $ and $ K $ are compact operators.
\end{definition}
Being commensurable is an equivalence relation. Symmetry and reflectivity 
are obvious. The proof of transitivity reduces to check that products of
compact perturbations of the identity is a compact perturbation of the
identity. From now on when $ X $ is commensurable to $ Y $ we will call
the pair $ (X,Y) $ commensurable.
\begin{proposition}
Let $ (X,Y) $ be a commensurable pair and $ (F,G) $ as above. 
The restrictions of $ F $ and $ G $ to $ Y $ and $ X $, denoted by $ f $
and $ g $ respectively, are the essential inverse, one of each other, hence,
by Proposition \ref{essential_inverse}, are Fredholm operators. Moreover, 
if $ (F',G') $ is another pair
\begin{gather}
\label{relative_dimension}
\ind f = \ind f', \ \ind g = \ind g'.
\end{gather}
\end{proposition}
\begin{proof}
For every $ t\in [0,1] $ consider the convex combinations
$ F_t = (1 - t) F + t F' $, $ G_t = (1 - t)G + tG' $. It is easy to check
that
\begin{align*}
f_t g_t &= {F_t G_t} _{|X} = I_X + k(t),\\
g_t f_t &= {G_t F_t} _{|Y} = I_Y + h(t)
\end{align*}
where $ h $ and $ k $ are continuous paths of compact operators on $ Y $ and
$ X $ respectively. Thus, for every $ t $ the operators $ f_t $ and $ g_t $ 
are the essential inverse one of each other. Taking $ t = 0 $, we obtain the
first part of the statement. By ii) of Proposition \ref{index-is-constant},
continuous paths of Fredholm operators have constant index. Hence
\begin{align*}
\ind f &= \ind f_0 = \ind f_1 = \ind f',\\
\ind g &= \ind g_0 = \ind g_1 = \ind g'.
\end{align*}
\end{proof}
\begin{definition}[relative dimension]
\label{relative_dimension:1}
Let $ (X,Y) $ and $ (F,G) $ be as in the preceding definition. We define the
\emph{relative dimension} of the pair $ \ind g $ and denote it by 
$ \Dim(X,Y) $.
\glsadd{labdimXY}
\end{definition}
The proposition proved above says that this definition does not depend on the
choice of the pair of operators $ (F,G) $. Given $ X,Y,Z $ such that $ (X,Y) $
and $ (Y,Z) $ are commensurable the properties
\begin{align*}
\Dim(X,X) &= 0,\\ \ \Dim(X,Y) &= -\Dim(Y,X),\\
\Dim(X,Z) &= \Dim(X,Y) + \Dim(Y,Z)
\end{align*}
follow from the properties of composition of Fredholm operators stated in
Proposition \ref{sum-of-index}.
We give now a definition of compact perturbation for pair of splitting 
subspaces, useful for building examples.
\begin{definition}[compact perturbation]
\label{compact_perturbation}
Let $ X, Y\in G_s (E) $. We say that they are \emph{compact perturbation}
(one of the each other) if, given two projectors $ P $ and $ Q $ with ranges
$ X $ and $ Y $ respectively, the operators
\begin{gather}
\label{compact_perturbation:1}
(I - P) Q, \ (I - Q) P
\end{gather}
are compact.
\end{definition}
When $ (X,Y) $ is a pair of elements of the Grassmannian of splitting spaces
commensurability and compact perturbation are equivalent.
\begin{proposition}
\label{equivalence_of_definitions}
Let $ X $ and $ Y $ closed and complemented subspaces of $ E $. 
Then $ (X,Y) $ is a commensurable pair if and only if $ X $ is compact 
perturbation of $ Y $.
\end{proposition}
\begin{proof}
Suppose $ X $ is compact perturbation of $ Y $ and let $ P $ and $ Q $ be
two projectors with ranges $ X $ and $ Y $. Clearly $ Q X\subset Y $ and
$ P Y\subset X $. Moreover,
\begin{align*}
Q x &= Q x - x + x = -(I - Q)P x + x \\
P y &= P y - y + y = -(I - P)Q y + y;
\end{align*}
we obtain two restrictions of compact perturbation of the identity,
as the definition of commensurability requires. Conversely let $ F $ and $ G $
be as in Definition \ref{commensurability} and $ (P,Q) $ a pair of projectors
with ranges $ X $ and $ Y $. We check, for instance, that $ (I - P)Q $ is 
compact.
\[
(I - P)Q = (I - P)(Q - F Q) + (I - P)F Q = (I - P) K Q + 0.
\]
Similarly $ (I - Q)P $ is compact. 
\end{proof}
For sake of simplicity we will sometimes use the notation
$ \dim(P,Q) $ or $ [P - Q] $ instead of $ \dim(\ran P,\ran Q) $.
\glsadd{labPQ}%
Let $ H $ be a Hilbert 
space and $ (X,Y) $ a pair of two closed subspaces 
that are compact perturbation one of each other. Call $ P_X $ and $ P_Y $
the orthogonal projections. By (\ref{compact_perturbation:1}) 
$ P_{Y^{\bot}} P_X $ and $ P_{X^{\bot}} P_Y $ 
are compact operators. Therefore 
\begin{equation*}
\begin{split}
P_X - P_Y &= (P_Y + P_{Y^{\bot}}) P_X - P_Y (P_X + P_{X^{\bot}}) = \\
&= P_{Y^{\bot}} P_X - P_Y P_{X^{\bot}} = 
P_{Y^{\bot}} P_X - (P_{X^{\bot}} P_Y)^* \in\mathcal{L}_c (E).
\end{split}
\end{equation*}
Hence $ P_X $ and $ P_Y $ have compact difference and the Definition
\ref{compact_perturbation} coincides with the one known for Hilbert spaces.
The relative dimension can be computed as
\begin{gather*}
\Dim(X,Y) = \dim\ker\res{P_Y}{X} - \coker\res{P_Y}{X} = 
\dim (X\cap Y^{\bot}) - \dim (X^{\bot}\cap Y)
\end{gather*}
which coincides with the definition of relative dimension in Hilbert spaces. 
In the following example we compute the relative dimension in some special 
case. 
\begin{example}
\label{ex:finite-dim}
Let $ V_0 $ and $ W_0 $ be finite dimensional subspaces and $ V_1 $
and $ W_1 $ topological complements of $ V_0 $ and $ W_0 $
respectively. We prove, using the result of Proposition 
\ref{equivalence_of_definitions}, that $ (V_0,W_0) $ and $ (V_1,W_1) $ are 
commensurable pairs and compute their relative dimension. Let $ P $ and $ Q $
be two projectors onto $ V_0 $ and $ W_0 $. Denote by $ q $ the restriction of 
$ Q $ to $ V_0 $. It is a linear map between finite dimensional subspaces, 
hence
\begin{align*}
\dim V_0 = \dim\ker q + \dim\ran q = \dim\ker q + \dim W_0 - \coker q
\end{align*}
and the Fredholm index of $ q $ is the difference of the dimensions of 
$ V_0 $ and $ W_0 $. Now consider the pairs $ (V_1,E) $ and $ (E,W_1) $ and 
the pairs of projectors $ (I - P,I) $, $ (I,I - Q) $ Thus
\begin{align*}
\dim(V_1,E) &= \ind \res{I}{V_1} = -\codim V_1 \\
\dim(E,W_1) &= \ind Q = \codim W_1 
\end{align*}
hence $ \dim(V_1,W_1) = \codim W_1 - \codim V_1 $. 
\end{example}
\begin{example}
In general it is not true that topological complements of two 
commensurable subspaces are commensurable. Given two splittings of the space
\begin{gather*}
X\oplus X' = E = Y\oplus Y', \ \ P = P(X,X'), \ \ Q = P(Y,Y')
\end{gather*}
with $ X $ and $ X' $ compact perturbations of $ Y $ and $ Y' $ respectively, 
from the relations (\ref{compact_perturbation}) it follows that
\begin{gather*}
P - Q = (I - Q) P + P(I - Q)
\end{gather*}
is a compact operator. This is unlikely to happen even when $ X $ and $ Y $ 
are the same space. For instance let $ X\subset E $ be a splitting subspace 
with a topological complement $ X' $ such that 
$ \mathcal{L}_c (X',X)\subsetneq\mathcal{L}(X',X) $. For 
any $ L\in\mathcal{L}(X',X)\setminus\mathcal{L}_c (X',X) $ define
\begin{gather*}
P(L) (x,y) = (x + Ly,0);
\end{gather*}
it is easy to check that $ P(L) $ is a projector with range $ X $ and
$ P(L) - P $ is not compact. However for a given pair of two commensurable
splitting subspaces a pair of projectors with compact difference always exists
and we prove it in the next theorem. This is equivalent to find topological
commensurable complements.
\end{example}
In the next proposition we describe the relation between the relative dimension
and the Fredholm index of Fredholm pairs.
\begin{proposition}
\label{transitivity_of_dimension}
If $ X $ is compact perturbation of $ Y $ and $ (Y,Z) $ is a Fredholm pair, 
then $ (X,Z) $ is Fredholm and $ \ind(X,Z) = \dim(X,Y) + \ind(Y,Z) $.
\end{proposition}
\begin{proof}
Let $ P $ and $ Q $ be projectors with ranges $ X $ and $ Y $ respectively. The
restrictions $ p $ and $ q $ to $ Y $ and $ X $ are Fredholm operators; we have
\begin{equation}
\begin{split}
\label{transitivity_of_dimension:1}
F_{X,Z} (x,z) &= x - z = x - Qx + Qx - z \\
&= (I - Q)Px + Qx - z = (I - Q)Px + F_{Y,Z}(Qx,z) \\
&= ((I - Q)P,0_Z)\cdot(x,z) + F_{Y,Z}\ci (q,I)\cdot(x,z).
\end{split}
\end{equation}
Since $ F_{Y,Z} $ and $ (q,I) $ are Fredholm their composition is Fredholm; 
the first summand of the last equation is compact. Hence $ F_{X,Z} $ is a 
compact perturbation of a Fredholm operator and therefore Fredholm by 
Proposition \ref{t+k_is_fredholm} and
\begin{gather*}
\ind F_{X,Z} = \ind F_{Y,Z}\ci (q,I) = \ind F_{Y,Z} + \ind (q,I) = 
\ind (Y,Z) + \dim(X,Y)
\end{gather*}
by Proposition \ref{sum-of-index}.
\end{proof}
\begin{example}
We use Proposition \ref{transitivity_of_dimension} with in example that shows 
that for commensurable pairs there is not a result like the Theorem 
\ref{perturbation}, that is, they are not stable by small perturbation: 
consider a pair $ (X,Y) $ such that 
\begin{enumerate}
\item $ X $ is isomorphic to $ Y $,
\item $ X\oplus Y = E $ has infinite dimension;
\end{enumerate}
let $ f\colon Y\rightarrow X $ be an isomorphism and $ \graph(f) $ its graph.
For every integer $ n $ consider the sequence of subspaces
\begin{gather*}
Y_n = \graph(nf);
\end{gather*}
since $ Y_n $ is graph of a bounded operator $ X\oplus Y_n = E $. It is easy
to check that $ Y_n $ converges to $ X $. Thus there can be no open
neighbourhood of $ X $ in $ G_s (E) $ made of compact perturbations of
$ X $. In fact for $ n $ large $ Y_n $ would be contained in such
neighbourhood and $ (X,Y_n) $ would be a commensurable pair; since $ (X,Y_n) $
is a Fredholm pair also, by Proposition \ref{transitivity_of_dimension} we
would have proved that $ (X,X) $ is a Fredholm pair which happens only if
$ X\oplus Y $ has finite dimension, in contradiction with hypothesis ii).
\end{example}
The preceding Proposition suggests a definition of the relative dimension
that involves the Fredholm index. Precisely, suppose $ X $ is compact 
perturbation of $ Y $. Let $ Z $ be a topological complement of $ Y $. Then
$ (Y,Z) $ is a Fredholm pair. By Proposition \ref{transitivity_of_dimension}
$ (X,Z) $ is a Fredholm pair and
\begin{gather}
\label{alternative}
\ind(X,Z) = \ind(Y,Z) + \dim(X,Y) = \dim(X,Y).
\end{gather}
This definition, together with the Theorem \ref{perturbation} will allows us to
state in the next chapter a stability result of the relative dimension for 
closed and splitting subspaces.
\begin{theorem}
\label{suitable_projector}
Let $ X $ be a splitting subspace, compact perturbation of $ Y $. Then there
are topological complements $ X' $ and $ Y' $ that are compact perturbation 
one of each other and
\begin{gather*}
\dim(X,Y) = -\dim(X',Y')
\end{gather*}
\end{theorem}
\begin{proof}
Let $ P $ and $ Q $ be projectors with ranges $ X $ and $ Y $ respectively. 
As consequence of the Proposition \ref{transitivity_of_dimension}
the pair $ (X,\ker Q) $ is a Fredholm. Let $ Z $ be a topological complement 
of $ X\cap\ker Q $ in $ \ker Q $ and $ R\subset E $ a finite dimensional 
complement of $ X + \ker Q $ in $ E $. Then
\begin{gather*}
X\oplus Z \oplus R = E, \ \ P_X + P_Z + P_R = I_E;
\end{gather*}
we claim that $ P_X $ and $ Q $ have compact difference. We write
\begin{gather*}
P_X - Q = (I - Q) P_X + (P_X - Q) P_Z + (P_X - Q) P_R;
\end{gather*}
the first term of the right member is compact by definition of compact
perturbation, the second is $ 0 $, the third has finite rank. Hence
\begin{gather*}
Q(I - P_X), \ \ P_X (I - Q) 
\end{gather*}
are compact operators. It is not hard to prove that for all the pairs of 
projectors $ (P',Q') $ onto $ X' $ and $ Y' $ respectively, compactness of
(\ref{compact_perturbation:1}) holds, thus $ \ker P_X $ and $ \ker Q $ are 
commensurable spaces. To compute the
relative dimension we use restrictions of the operators $ Q $ and $ I - Q $.
We can write
\begin{gather*}
\dim(X,Y) + \dim(X',Y') = \ind\res{Q}{X} + \ind\res{(I - Q)}{X'} 
= \ind I_E = 0.
\end{gather*}
\end{proof}
The next Proposition follows the one known for Hilbert spaces, due to
A.~Abbondandolo and P.~Majer (refer \textsc{Proposition} 2.2 of \cite{AM01}).
\begin{proposition}\label{cp:2}
Let $ T, S\in\mathcal{L}(E,F) $ be operators with compact difference and 
closed images. If the kernels and the images split
\begin{TAKEOUT}
\footnote{Although we cannot think to any good reason why this result 
shouldn't be true for operators with compact difference, no matter if the 
kernels and images split or not, we were not able to find a proof to achieve 
this improvement.}
\end{TAKEOUT}
$ \ker T $ and $ \ran T $ are compact 
perturbation of $ \ker S $ and $ \ran T $ respectively and
the relation
\[
\dim (\ker T, \ker S) = - \dim (\ran T, \ran S).
\]
holds.
\end{proposition}
\begin{proof}
Since kernels and images split we can write 
\begin{align*}
\ker T \oplus Y(T) &= E = \ker S \oplus Y(S)\\
 Z(T)\oplus\ran T  &= F = Z(S)\oplus\ran S
\end{align*}
Since $ T $ and $ S $ are isomorphism of $ Y(T) $ with $ \ran T $ and 
$ Y(S) $ with $ \ran S $ respectively, we can define
operators $ T' $ and $ S' $ on $ F $ with values in $ E $ such that
\begin{align*}
T' T &= P(Y(T),\ker T),\ \ S' S = P(Y(S),\ker S) \\
T T' &= P(\ran T,Z(T)), \ \ S S' = P(\ran S, Z(S));
\end{align*}
set $ P(T) = P(\ker T,Y(T)) $, $ P(S) = P(\ker S,Y(S)) $ and $ K = T - S $.
Then 
\begin{gather*}
(I - P(S)) P(T) = S' S P(T) = S' (S - T) P(T) + S' T P(T) = S' K P(T) + 0
\end{gather*}
is a compact operator. Set $ Q(T) = P(\ran T,Z(T)) $, 
$ Q(S) = P(\ran S,Z(S)) $. Then
\begin{equation*}
\begin{split}
(I - Q(S)) Q(T) &= (I - Q(S)) T T' = (I - Q(S)) (T - S) T' + 
(I - P(S)) S T' \\ 
&= 0 + (I - P(S)) K T'
\end{split}
\end{equation*}
is compact. By Theorem \ref{suitable_projector}, up to changing the 
topological complements of $ \ker T $ and $ \ran T $, we can suppose that 
our projectors have compact difference. Hence
\begin{align*}
&\dim (\ker T,\ker S) = -\dim(Y(S),Y(T)) \\ 
=& -\ind\res{(I - P(T))}{Y(S)} ^{Y(T)} = 
-\ind\res{T (I - P(T))}{Y(S)} ^{\ran T} \\
&\dim (\ran T,\ran S) = \ind\res{Q(T)}{\ran S} ^{\ran T} = 
\ind\res{(Q(T)S)}{Y(S)} ^{\ran T};
\end{align*}
observe that the operator
\begin{gather*}
K_1 = T(I - P(T)) - Q(T) S = TT'T - T T'S = TT' (T - S)
\end{gather*}
is compact. Therefore
\[
\begin{split}
\dim (\ker T,\ker S) &= -\ind\res{(Q(T) S + K_1)}{Y(S)} ^{\ran T} \\
&= -\ind\res{(Q(T)S)}{Y(S)} ^{\ran T} = 
-\dim (\ran T,\ran S).
\end{split}
\]
When $ P - Q $ is compact, we know of an existing definition of relative
dimension in \cite{BDF73} for Hilbert spaces and in \cite{ZL99}, for Banach 
spaces and projectors $ P,Q $ with compact difference. In the latter, given 
two projectors, they denote the relative dimension by $ [P - Q] $. We will 
also use this notation in \S \ref{chap5}.
\glsadd{labPQ}
\end{proof}
\subsubsection*{The non-complemented case}
In the technique used in the proposition above requires that the kernels
and images split. We think this restriction can be removed. We also guess
that whenever $ X $ is complemented and $ Y $ is commensurable to $ Y $,
then $ Y $ also splits.
%%% TEXEXPAND: END FILE ./chapter1-v2.tex
%%% TEXEXPAND: INCLUDED FILE MARKER ./chapter2-v2.tex
\chapter{Homotopy type of Grassmannians}
\label{chap2}
We define the
\textsl{essentially hyperbolic} operators on a Banach space $ E $, that we 
will denote by $ e\mathcal{H}(E) $, and prove the existence of a group 
homomorphism
\[
\pi_1 (e\mathcal{H}(E),2P - I)\rightarrow\mathbb{Z}
\]
where $ P $ is a projector of $ E $. The construction of such homomorphism
is carried out as follows: as first step, in section \S 2.1, we define the 
\textsl{Calkin algebra}, $ \mathcal{C}(E) $, as the quotient of 
the algebra of bounded operators $ \mathcal{L}(E) $ with the closed ideal of 
compact operators $ \mathcal{L}_c (E) $. Then we prove that 
$ e\mathcal{H}(E) $ is homotopically equivalent to 
$ \mathcal{P}(\mathcal{C}(E)) $, the space of 
\textsl{idempotent} elements of the Calkin algebra. In section \S 2.4 we
prove that the map
\[
\prc\colon\mathcal{P}(E)\rightarrow\mathcal{P}(\mathcal{C}),
\ \ \prc(P) = P + \mathcal{L}_c (E)
\]
is surjective and induces a locally trivial fiber bundle. Using 
the \textsl{Leray-Schauder degree} we prove in section \S 2.6 that the typical
fiber of such bundle has infinite numerable connected components. Hence,
for every projector $ P $, we can complete the exact homotopy sequence of the 
fiber bundle as follows
\[
\xymatrix{
\pi_1 (\mathcal{P}(E),P) \ar[r]^{\prc_*} &
\pi_1 (\mathcal{P}(\mathcal{C}),\prc(P)) \ar[r]^-{\vfi_P} &
\mb{Z};}
\]
we call $ \vfi_P $ \textsl{index} of fiber bundle 
$ (\mathcal{P}(E),\mathcal{P}(\mathcal{C}),\prc) $ with respect to $ P $
or, simply \textsl{index} when no ambiguity occurs. Thus, $ \vfi_P\circ\Psi_* $
is well-defined on $ e\mathcal{H}(E) $, where $ \Psi $ is a homotopy
equivalence with $ \mathcal{P}(\mathcal{C}) $.
All these facts are proved without making assumptions on the Banach space 
$ E $. Given a projector $ P $ the two conditions
\begin{itemize}
\item[h1)] $ P $ is connected to a projector $ Q $ such that 
$ Q - P\in\mathcal{L}_c (E) $ and $ \dim(Q,P) = m $,
\item[h2)] the connected component of $ P $ in $ \mathcal{P}(E) $ is 
simply-connected,
\end{itemize}
are sufficient to ensure that $ m\in\text{Im}(\vfi_P) $ and $ \vfi_P $ is
injective. When $ m = 1 $, we have an isomorphism.
These hypotheses are verified by every projection of a Hilbert space with 
infinite dimensional range and kernel. In the most common Banach spaces such 
as $ L^p $ spaces and spaces of sequences, we can find such projectors.

In the last section, we give exhibit examples where the homomorphism 
$ \varphi $ is an isomorphism. This happens, for instance, when
$ E $ is an infinite-dimensional Hilbert space or $ L\sp p $ for $ p\geq 1 $
or $ L\sp\infty $ and spaces of sequences $ \ell\sp p,m,c_0 $.
\section{The space of essentially hyperbolic operators}
Given a Banach algebra $ \al{B} $ with unit $ 1 $, we denote by 
$ G(\mathcal{B}) $ the set of invertible elements. 
If $ x\in\mathcal{B} $ the \textsl{spectrum} of $ x $ is defined as the set 
$ \set{\lambda\in\mathbb{C}}{x - \lambda\cdot 1\not\in G(\mathcal{B})} $ and 
denoted it by $ \sigma_{\mathcal{B}} (x) $ or simply $ \sigma(x) $. Consider 
the following subsets endowed with the topology of the norm
\begin{gather*}
\al{P}(\al{B}) = \set{p\in\al{B}}{p^2 = p},\ \ 
\al{Q}(\al{B}) = \set{q\in\al{B}}{q^2 = 1},\\
\al{H}(\al{B}) = \set{x\in\al{B}}{\sigma(x)\cap \img\mathbb{R} = \emptyset};
\end{gather*}
\glsadd{labPB}
\glsadd{labQB}
\glsadd{labHB}
\glsadd{labsx}
We call the elements of these spaces \textsl{projectors} 
(or \textsl{idempotents}), \textsl{square roots of the unit} and 
\textsl{hyperbolic} respectively. In literature, hyperbolic operators are 
sometimes defined as those whose spectrum does not intersect the unit circle; 
in this case \textsl{infinitesimally hyperbolic} would be more appropriate for
the elements of $ \mathcal{H}(\mathcal{B}) $.
The spaces $ \mathcal{P}(\mathcal{B}) $ and $ \mathcal{Q} (\mathcal{B}) $ are 
analytic, closed, embedded sub-manifolds of 
$ \al{B} $, see \cite{AM03a}, \textsc{Lemma} 1.4 for a proof;
$ \mathcal{H}(\mathcal{B}) $ is an open subset of $ \mathcal{B} $. An 
analytical diffeomorphism between $ \mathcal{P}(\mathcal{B}) $ and 
$ \mathcal{Q}(\mathcal{B}) $ exists, given by 
\[
\al{P}(\al{B})\ni p\mapsto 2p - 1\in\al{Q}(\al{B}).
\]
We prove that these three spaces have the same homotopy type.
Since $ \mathcal{P} $ and $ \mathcal{Q} $ are diffeomorphic they have the 
same homotopy type; in the next proposition we define a homotopy equivalence 
between $ \mathcal{Q} $ and $ \mathcal{H} $. In order to do so, 
we need some preliminary notations and facts.
Let $ x $ be an element of the algebra $ \mathcal{B} $ and $ \{A_i\} $ a
finite open cover of the spectrum of pairwise disjoint sets.
There are projectors $ p_i $, called \textsl{spectral projectors}, such that 
\[
p_1 + \dots + p_n = 1,\quad p_i p_j = \delta_{ij} p_j,\quad
\sigma_{\mathcal{B}_i} (p_i x p_i) = A_i 
\]
where $ \mathcal{B}_i \subset\mathcal{B} $ is the sub-algebra of the elements 
$ p_i x p_i $ with $ x\in\mathcal{B} $. We denote $ p_i $ also by $ p(x;A_i) $.
These projectors can be obtained as integrals
\glsadd{labpAx}
\begin{align*}
p(x;A_i) = \frac{1}{2\pi i}\int_{\gamma_i} (\lambda - x)^{-1} d\lambda
\end{align*}
where $ \gamma_i $ are closed paths such that each $ \gamma_i $ 
\textsl{surrounds} $ A_i \cap\sigma(x) $ in 
$ \mathbb{C}\setminus\cup_{j\neq i} A_j $
in the sense of Definition \ref{defn:surround} of Appendix C.
\begin{proposition}
\label{retratto}
The space of roots of the unit is a deformation retract of the space of
hyperbolic elements.
\end{proposition}
\begin{proof}
If $ q $ is a square root of the unit its spectrum is contained in 
$ \{-1,+1\} $, hence $ q $ is hyperbolic. Call $ i $ the inclusion of
the space of idempotents in the space of hyperbolic elements. We define
a retraction map as follows: let $ x $ be a hyperbolic element of the
algebra; since,
\[
\sigma(x) = (\sigma(x)\cap\{\re z > 0\})\cup(\sigma(x)\cap\{\re z < 0\})
\]
the spectrum has an open cover of disjoint subsets. Denote by 
$ p^+ (x) $ and 
$ p^- (x) $ the spectral projectors $ p(x;(\sigma(x)\cap\{\re z > 0\})) $ and 
$ p(x;(\sigma(x)\cap\{\re z < 0\})) $ respectively. We define the map
\[
r\colon\mathcal{H}\ra\mathcal{B},\quad r(x) = p^+ (x) - p^- (x);
\]
$ r $ is continuous by Theorem \ref{decomposition_of_spectrum} and $ r(x) $
is a square root of unit. We prove that $ r $ is a left inverse of
the inclusion $ i $. Let $ q $ be a square root and 
$ z\in\mathbb{C}\setminus\sigma(q) $, then
\[
(z - q)^{-1} = \frac{z}{z^2 - 1} + \frac{q}{z^2 - 1} =
\frac{1}{2}\left(\frac{1}{z + 1} + \frac{1}{z - 1}\right) + 
\frac{1}{2}\left( \frac{1}{z - 1} - \frac{1}{z + 1}\right) q;
\]
let $ \gamma_+ $ and $ \gamma_{-} $ be paths that surrounds $ 1 $ and $ -1 $
in  $ \mathbb{C}\setminus\{1\} $ and $ \mathbb{C}\setminus\{-1\} $
respectively. By integrating both sides of the above equality 
around $ \gamma_+ $ and $ \gamma_{-} $ and dividing it by $ 2\pi i $,
we obtain
\[
p^+ (q) = (1 + q)/2,\quad p^- (q) = (1 - q)/2,\quad
r(q) = p^+ (q) - p^{-} (q) = q;
\]
this proves that $ \mathcal{Q} $ is a retraction of $ \mathcal{H} $.
Now, define the continuous map
\[
F\colon [0,1]\times\mathcal{H}\ra\mathcal{B},
\quad
(t,x) \mapsto (1 - t) p^+ x p^+ + tp^+ + (1 - t) p^- x p^- - tp^-.
\]
By Property ii) and iii) of Appendix C, $ F(t,x) $ is hyperbolic 
for every $ (t,x) $. We also have $ F(0,x) = x $, 
$ F(x,1) = i\ci r(x) $.  Thus $ i\ci r $ is homotopically
equivalent to $ id_{\al{H}} $.
\end{proof}
\begin{definition}
Given an operator $ T\in\mathcal{L}(E) $ we call \emph{essential spectrum},
and denote it by $ \sigma_e (T) $, the set 
$ \{\lambda\in\C:T - \lambda I:\mbox{\rm is not Fredholm}\} $.
\glsadd{labseT}
\end{definition}
\begin{definition}
A bounded operator $ T $ is called \emph{essentially hyperbolic} if and only 
if $ \sigma_e (T)\cap \img\mathbb{R} = \emptyset $. We denote by 
$ e\mathcal{H}(E) $ the set of essentially hyperbolic operators endowed with 
the norm topology.
\glsadd{labeHE}
\end{definition}
The set of compact operators on a Banach space $ E $ is a closed ideal of 
the algebra of bounded operators. Thus the quotient has a structure of
Banach algebra that makes the projection
\[
p\colon\mathcal{L}(E)\rightarrow \mathcal{L}(E)/\mathcal{L}_c (E),\quad
T\mapsto T + \mathcal{L}_c (E)
\]
an algebra homomorphism. The quotient space is called \textsl{Calkin algebra}
and we denote it by $ \mathcal{C}(E) $ or just $ \mathcal{C} $.
\glsadd{labCE}
We characterize the essential spectrum in terms of the Calkin algebra: given
$ T\in\mathcal{L}(E) $ there holds
\begin{gather}
\label{characterization_essential}
\sigma_e (T) = \sigma(p(T)).
\end{gather}
To prove the equality suppose $ \lambda\not\in \sigma_e (T) $, hence 
$ T - \lambda $ is Fredholm. By Proposition \ref{essential_inverse} there
exists an essential inverse $ S $ such that
\[
(T - \lambda)S - I,\quad S(T - \lambda) - I
\]
are compact operators. Hence $ p(T - \lambda) $ is invertible in the Calkin
algebra, $ p(S) $ being is its inverse, thus 
$ \lambda\not\in\sigma(p(T)) $. The prove of the other inclusion also follows
from Proposition \ref{essential_inverse}.
\begin{theorem}
\label{topology}
The space $ e\al{H}(E) $ has the homotopy type of 
$ \mathcal{P}(\mathcal{C}) $.
\end{theorem}
\begin{proof}
First we prove that $ e\al{H}(E) $ is homotopically equivalent 
to $ \al{H}(C) $. By classical results of continuous selections there
exists a continuous right inverse of $ p $, call it $ s $. It is a 
consequence of Theorem \ref{bartle_graves} when the topological space $ T $ 
consists of a point. Using the characterization 
(\ref{characterization_essential}) it is easy to check that
$ e\mathcal{H}(E) = p^{-1}(\mathcal{H}(\mathcal{C})) $. Moreover,
the two continuous maps
\begin{align*}
&\mathcal{H}(\mathcal{C})\times\ker p\rightarrow e\mathcal{H}(E), 
\quad (x,K)\mapsto s(x) + K\\
&e\mathcal{H}(E)\ra\mathcal{H}(\mathcal{C})\times\ker p,
\quad A\mapsto \big(p(A),A - s(p(A))\big)
\end{align*}
are the inverses of each other, hence 
$ \mathcal{H}(\mathcal{C})\times\ker p $ is homemorphic to 
$ e\mathcal{H}(E) $.
Since $ \ker p = \mathcal{L}_c (E) $ is a linear subspace of 
$ \mathcal{L}(E) $, is contractible, thus the two maps are homotopically
equivalent to the maps
\begin{align*}
s\colon &\mathcal{H}(\mathcal{C})\ra e\mathcal{H}(E)\\
p\colon &e\mathcal{H}(E)\ra \mathcal{H}(\mathcal{C}).
\end{align*}
Now, by Proposition \ref{retratto}, $ \al{H}(\al{C}) $ has the 
same homotopy type of $ \al{Q}(\al{C}) $ which is homeomorphic to 
$ \al{P}(\al{C}) $. Taking the composition of all the functions we
referred to, we can write explicitly an homotopy equivalence between 
$ e\mathcal{H}(E) $ and $ \mathcal{P}(\mathcal{C}) $ and its homotopic inverse:
\begin{align*}
\Psi&\colon e\mathcal{H}(E)\rightarrow\mathcal{P}(\mathcal{C}),\quad 
A\mapsto p\sp + (p(A))\\
\Phi&\colon \mathcal{P}(\mathcal{C})\ra e\mathcal{H}(E),\quad
p\mapsto s(2p - 1).
\end{align*}
\glsadd{labPs}
\glsadd{labPh}
\end{proof}
\section{The fiber bundle $ (G(\mathcal{B}),\mathcal{P}(\mathcal{B})) $}
In this section we define the fiber bundle with total space $ G(\mathcal{B}) $
and base space $ \mathcal{P}(\mathcal{B}) $. The exact homotopy sequence
associated to the fiber space provides us with some relations between
the homotopy groups of the base space and the total space. 
\begin{definition}
We say that two projectors $ p,q $ are \textsl{conjugated} if there exists an 
invertible element $ g\in G(\mathcal{B}) $ such that $ gp = qg $.  
\end{definition}
The projector $ \orl{p} = 1 - p $ is naturally associated
to $ p $.
\glsadd{labpp}
\begin{proposition}{\rm (cf. \cite{PR87}, \textsc{Proposition} 4.2).}
\label{pairs_of_projectors}
In the space of idempotents the following facts hold:
\begin{enumerate}
\item if $ \no{p - q} < 1 $, there exists an invertible element 
$ g\in G_0 (\mathcal{B}) $ such that $ gp = qg $; thus, the space of
idempotents is locally arcwise connected;
\item two idempotents are connected by a continuous path, if and only
if there exists $ g\in G_0 (\mathcal{B}) $ such that $ gp = qg $.
\end{enumerate}
\glsadd{labGB}
\end{proposition}
\begin{proof}
i). Given $ p,q $ we define $ L(p,q) = pq + (1 - p)(1 - q) $. As $ t $
varies in $ [0,1] $, we have
\begin{equation}
\label{eq:lpq}
\begin{array}{c}
(1 - t + tL(p,q))(1 - t + tL(q,p)) = 1 - t(2 - t)(p - q)^2\\
L(p,q)L(q,p) = 1 - (p - q)\sp 2.
\end{array}
\end{equation}
The right term is an invertible operator because $ \no{p - q} < 1 $
and $ t(2 - t)\leq 1 $. From the second equality it follows that $ L(p,q) $ 
and $ L(q,p) $ commute, hence they are invertible too. Moreover, each
of them is joint to the unit by the path. From Example 
\ref{ex:square-root} there exists $ R $ such that
\[
R(p,q)\in G_0 (\mathcal{B}),\ 
R(p,q)\sp 2 = \left(1 - (p - q)\sp 2\right)\sp{-1}
\]
Thus $ L(p,q) R $ and $ L(q,p) R $ are the inverse of each other. By 
multiplying the second of (\ref{eq:lpq}) by $ R $ on both sides, we obtain
\begin{equation}
\label{conjugation}
L(q,p) Rp = q L(p,q) R.
\end{equation}
We define $ g(p,q) = L(p,q) R $.\vskip .2em
\glsadd{labgpq}
ii). Let $ \alpha $ be a continuous path such that $ \alpha(0) = p $
and $ \alpha(1) = q $. Let $ \{t_i:0\leq i\leq n\} $ be a 
partition of the unit interval such that 
$ \no{\alpha(t_i) - \alpha(t_{i + 1})} < 1 $ for every $ i $. Let
\[
g = \prod_{i = 0} \sp{n - 1} g(\alpha(t_{n - i}),\alpha(t_{n - i - 1}))
\]
since $ g $ is a product of elements of $ G_0 (\mathcal{B}) $, it also 
belongs to $ G_0 (\mathcal{B}) $. 
By applying (\ref{conjugation}) $ n $ times, we 
obtain $ gp = qg $. Conversely, if there exists an element 
$ g\in G_0 (\mathcal{B}) $ such that $ gp = qg $, then the path
$ g(t)pg(t)\sp{-1} $ joins $ q $ to $ p $.
\end{proof}
\noindent Given a projector $ p $, we denote by 
$ \mathcal{P}_p (\mathcal{B}) $ the connected component of $ p $ and
define the following subgroups of $ G(\mathcal{B}) $:
\glsadd{labPBp}
\begin{gather*}
G_p (\mathcal{B}) = 
\{g\in G(\mathcal{B}):gpg\sp {-1}\text{ is connected to }p\},\\
F_p = \set{g\in G(\mathcal{B})}{gp = pg}.
\glsadd{labGpB}
\end{gather*}
Clearly $ F_p \subset G_p $. We define the map
\[
\pi_p\colon G_p\ra\mathcal{P}_p,\quad g\mapsto gpg\sp{-1}.
\glsadd{laborlp}
\]
By ii) of Proposition \ref{pairs_of_projectors}, $ \pi_p $ is surjective.
\begin{theorem}{\rm (cf. \cite{PR87}, \S 7).}
\label{th:gl-p}
The triple $ (G_p,\pi_p,\mathcal{P}_p) $ induces a principal bundle with 
group $ F_p $ acting on itself by multiplication on the left.
\end{theorem}
\begin{proof}
We prove that there exists an open cover of coordinate neighbourhoods.
Fix $ q\in\mathcal{P}_p $ and let $ g $ as in ii) of Proposition 
\ref{pairs_of_projectors}. On the ball $ B(q,1) $ we define a section of the 
projection map $ \pi_p $
\begin{equation}
\label{loc_sect_alg}
s_q \colon B(q,1)\rightarrow G_p,\quad r\mapsto g(r,q) g.
\end{equation}
Clearly $ \pi_p (s_q (r)) = r $. We define 
coordinate neighbourhoods
\[
\phi\colon B(q,1)\times F_p\ra\pi_p \sp{-1} (B(q,1)),\quad
(x,y) \mapsto s_q (x) \cdot y.
\]
It is an homeomorphism and $ \pi_p (\phi(x,y)) = x $. If two coordinate 
neighbourhoods, $ B(q_1,1) $ and $ B(q_2,1) $ intersect, the transition
maps are 
\[
\phi_{2,x}^{-1}\phi_{1,x}\colon F_p\rightarrow F_p, \ y\mapsto 
s_2 (x)^{-1} s_1 (x) y;
\]
where $ s_i $ are the sections defined on $ B(q_1,1) $ and $ B(q_2,1) $,
respectively. Since $ s_2 (x)^{-1} s_1 (x)\in F_p $ we have 
defined a \textsl{principal bundle} according to \cite{Ste51}, \S 8. 
\end{proof}
For principal bundles we can write the \textsl{exact homotopy sequence},
see \cite{Ste51}, \S 17. The sequence
\begin{gather}
\label{exact_homotopy_sequence:1}
\xymatrix{
\pi_k (F_p, 1) \ar[r]^{i_*} & \pi_k (G_p,1) \ar[r]^{\pi_{p,*}} 
& \pi_k (\mathcal{P}_p,p) \ar[r]^{\partial} & \pi_{k - 1} (F_p,1)
}
\end{gather}
is exact for every $ k\geq 1 $.
\section{The Grassmannian algebra}
Given $ p,q $ idempotents of an algebra $ \mathcal{B} $ we define
the following equivalence relation
\begin{gather}
\label{range_relation}
p\sim q\iff pq = q, \ qp = p.
\end{gather}
Symmetry is obvious. If $ (p,q) $ and $ (q,r) $ are
equivalent pairs then $ pr = p(qr) = (pq)r = qr = r $, similarly $ rp = p $,
then reflectivity follows.
\begin{definition}
We denote by $ Gr(\mathcal{B}) $ the set of equivalence classes endowed with
the quotient topology.
\glsadd{labGrB}
\end{definition}
H. Porta and L. Recht proved in \cite{PR87} that the 
Grassmannian algebra is a metric space, the quotient projection 
$ \pi\colon\mathcal{P}(\mathcal{B})\rightarrow Gr(\mathcal{B}) $ 
is an open map and there exists a global continuous section
of $ \pi $ on $ Gr(\mathcal{B}) $. In fact any global continuous section is 
a homotopy inverse of $ \pi $ (see \cite{PR87}, \S 3). 
\vskip .2em
When $ \mathcal{B} $ is the algebra of the bounded operators on a Banach space
$ E $ two projectors are equivalent if and only if they have the same images.
In fact the identity $ P Q = Q $ means that $ \ran Q\subseteq\ran P $.
Then we have a well defined bijection
\[
Gr(\mathcal{L}(E))\rightarrow G_s (E), \ \ \pi(P)\mapsto\ran P.
\]
\begin{lemma}[refer \cite{Geb68}]
\label{projectors_Grassmannian}
There exists a continuous section of the map that associates a projector
with its range. Every section is in fact a homotopy equivalence.
\end{lemma}
\begin{proof}
Call $ r $ the map $ \al{P}(E)\ni P\mapsto\ran P $. This is continuous
with the opening metric defined in \S 1.
In fact, given $ P,Q\in\al{P}(E) $, it can be easily
checked that
\glsadd{labrP}
\[
\delta_S (r(P),r(Q))\leq 2\no{P - Q}.
\]
We can build now a continuous section of $ r $ using the construction of
\cite{Geb68} whose idea is the following: fix $ X $ a splitting subspace
and choose $ Y $ a topological complement. By Proposition \ref{GE_is_open}, 
for every $ X'\in B(X,\hat{\gamma}(X,Y)) $, we have $ X'\oplus Y = E $. We 
define
\[
s\colon B(X,\hat{\gamma}(X,Y))\rightarrow \al{P}(E), \ X'\mapsto P(X',Y);
\]
by Proposition \ref{GE_is_open} this is a continuous local section of the map 
$ r $. Since $ G_s (E) $ is metric, thus paracompact, there exists 
a locally finite refinement of the open covering 
$ \{B(X,\hat{\gamma}(X,Y))\} $, say $ \mathcal{U} = \set{U_i}{i\in I} $. 
Let $ \{\vfi_i\} $ be a partition of unit subordinate to $ \mathcal{U} $. 
Thus for every $ X $ in $ G_s (E) $ define
\[
s(X) = \sum_{i\in I} \vfi_i (X) s_i (X), \ s\in C(G_s (E),\mathcal{L}(E)).
\]
To prove that $ s(X) $ is a projector observe that if $ X\in U_i\cap U_j $,
then
\[
\ran P(X,Y_i) = \ran P(X,Y_j) = X.
\]
This is equivalent to
\[
s_i (X) s_j (X) = s_j (X), \ s_j (X) s_i (X) = s_i (X);
\]
keeping in mind these relations it is easy to prove that 
$ s(X) $ is a projector with range $ X $. In fact
\begin{equation*}
\begin{split}
s(X)^2 &= \sum_i \vfi_i s_i (X) \left(\sum_j \vfi_j s_j (X)\right) = 
\sum_i \vfi_i \left(\sum_j \vfi_j (X) s_i (X) s_j (X)\right) \\ 
&= \sum_i \vfi_i s(X) = s(X).
\end{split}
\end{equation*}
This also proves that $ r^{-1} (\{X\}) $ is a convex, actually affine, subspace
of $ \mathcal{P}(E) $. By construction $ r\ci s = id $. For every 
projector $ P $ we have  
\[
r (s \ci r (P)) = r(P), \ \ tP + (1 - t) s\ci r (P)\in\al{P}(E)
\]
for every $ t\in [0,1] $. This defines a homotopy between $ s\circ r $ and 
the identity map.
\end{proof}
As application of the preceding Lemma we state a result of stability of
the relative dimension defined on Chapter I. 
\begin{theorem}
\label{stability_dimension}
Let $ X $ and $ Y $ be continuous functions defined on a
topological space $ M $ such that $ X(t) $ and $ Y(t) $ are closed and 
splitting subspaces and $ X(t) $ is compact perturbation of $ Y(t) $ for 
every $ t $ in $ M $. Hence $ \dim(X(t),Y(t)) $ is locally constant.
\end{theorem}
\begin{proof}
Let $ s $ be a continuous section on $ G_s (E) $ of the map $ r $ defined in
the Lemma \ref{projectors_Grassmannian}. Then it is defined a continuous map
\[
\nu\colon G_s (E)\rightarrow G_s (E), \ \ X\mapsto \ker s(X).
\]
By the identity (\ref{alternative}) the relative dimension of the pair
$ (X(t),Y(t)) $ is the Fredholm index of the pair $ (X(t),\nu(Y(t))) $. 
Fix $ t_0 $ in $ M $; by Theorem \ref{perturbation} there exists a open 
neighbourhood of $ t_0 $, say $ U $, such that 
\[
\ind(X(t),\nu(Y(t))) = \ind(X(t_0),\nu(Y(t_0)))
\]
for every $ t\in U $. Therefore we conclude with (\ref{alternative}).
\end{proof}
\begin{theorem}
If $ \al{B} $ is the algebra of bounded operators on $ E $ then 
$ Gr (\al{B}) $ with the quotient topology is homeomorphic to $ G_s (E) $
with the topology induced by the metric $ \delta_S $.
\end{theorem}
\begin{proof}
Let $ s $ and $ \gamma $ be sections on $ G_s (E) $ and $ Gr(\mathcal{B}) $ 
respectively. We prove that the maps $ \pi\ci s $ and 
$ r\ci \gamma $ are inverse one of each other. Let $ X $ be a closed 
splitting subspace. Then
\[
\gamma((\pi\ci s)(X))\sim s(X)
\]
then $ r(s(X)) = X $. Thus $ (r\ci \gamma)\circ(\pi\ci s) = id $. 
Similarly, we have $ (\pi\ci s)\circ (r\circ \gamma) = id $.
\end{proof}
\section{Fibrations of spaces of idempotents}
Set $ \al{B} = \mathcal{L}(E) $; we recall that the Calkin algebra is defined as the
quotient algebra $ \mathcal{C} = \mathcal{L}(E)/\mathcal{L}_c (E) $ where
$ \mathcal{L}_c (E) $ is the ideal of compact operators on $ E $. 
It is a Banach algebra with unit. The projection to the quotient
$ p\colon\mathcal{B}\rightarrow\mathcal{C} $ is a surjective homomorphism.
Consider the restrictions
\begin{align*}
\prc\colon &\mathcal{P}(E)\rightarrow\mathcal{P}(\mathcal{C}) \\
p_r\colon &\mathcal{Q}(E)\rightarrow\mathcal{Q}(\mathcal{C}).
\end{align*} 
The purpose of this section is to prove that these maps induce locally 
trivial bundle, with non-constant fiber. First we need the following
\begin{proposition}{\rm{(cf. also \cite{AM03a}, \textsc{Proposition 6.1})}}.
\label{loc_sect_Calkin}
The maps $ \prc $ and $ p_r $ are surjective.
\end{proposition}
\begin{proof}
It is enough to prove it for $ p_r $, because the homeomorphism 
between $ \mathcal{P} $ and $ \mathcal{Q} $ commutes with the quotient
projections. Let $ q $ be a square root of identity in the Calkin algebra and 
let $ Q $ be an operator such that $ p(Q) = q $. There exists a compact 
operator $ K $ such that $ Q\sp 2 = I + K $. The spectrum of $ I + K $ is a 
countable subset of $ \C $ with at most $ 1 $ as limit point. Let $ U $ be a 
neighbourhood of $ 1 $ such that
\[
\p U\cap \sigma(I + K) = \emptyset,\quad 
U\cap\sigma(I + K)\subset B(1,1).
\]
Let $ P $ be the spectral projector relative to $ U $. Clearly $ I - P $
has finite dimensional range. Let $ Q_1 $ and $ K_1 $ be the restrictions
of $ Q $ to the range of $ P $. We have
\[
Q_1 \sp 2 = I + K_1
\]
where $ K_1 $ is compact and $ \sigma(K_1)\subset B(0,1) $. Thus, $ Q_1 $
is invertible. We seek $ H $ compact such that 
\[
(Q_1 - HQ_1)\sp 2 = I,\ [Q_1,H] = 0.
\] 
The first becomes $ (I + K_1)(I - H)\sp 2 = I $. A solution of this
equation is given by 
\[
I - \widehat{f}(K_1),\quad f(z) = \frac{1}{\sqrt{1 + z}}
\glsadd{labwf}
\]
where $ \widehat{f}(K_1) $ is defined according to Theorem 
\ref{thm:functional-calculus}.
Since the first coefficient of $ f $ in the power series expansion,
in a neighbourhood of the origin, is $ 1 $,
the operator above is compact. Thus, 
$ (I - P)\oplus Q_1 (I - \widehat{f}(K_1)) $ is a square root of unit and a 
compact perturbation of $ Q $.
\end{proof}
\begin{theorem}
\label{Calkin_bundle}
The map $ \prc\colon\mathcal{P}(E)\rightarrow
\mathcal{P}(\mathcal{C}) $ induces a locally trivial fiber bundle.
\end{theorem}
\begin{proof}
Let $ x_0\in \mathcal{P}(\mathcal{C}) $ and $ D $ be its connected component.
The tuple $ (p\sp{-1} (D),D,\prc) $ is a 
locally trivial bundle with fiber homeomorphic to $ \prc\sp{-1} (\{x_0\}) $. 
By Theorem \ref{th:gl-p} and Appendix D, there exists a map on a
neighbourhood $ U_{x_0} $ of $ x_0 $
\[
T\colon U_{x_0}\ra GL(E),\quad (\pi_{x_0} \circ p)\circ T(x) = x.
\]
Thus, we can define a coordinate neighbourhood on $ U_{x_0} $, with its
inverse, as follows
\[
\phi\colon U_{x_0}\times p\sp{-1}(x_0)\ra p\sp{-1} (U_{x_0}),
\quad (x,y)\mapsto T(x)y T(x)\sp{-1}.
\]
There holds $ p\circ \phi (x,y) = x $ and is invertible. Given a point 
$ z\in D $, let $ g\in G(\mathcal{C}) $ and 
$ G\in GL(E) $ such that
\[
p(G) = g,\quad g x_0 g^{-1} = z.
\]
Such $ g $ is provided by ii) of Proposition \ref{pairs_of_projectors}.
The existence of $ G $ follows from the surjectivity of 
$ p\colon GL(E)\ra G(\mathcal{C}) $ (refer Appendix D). We define
a trivialization of the neighbourhood $ U_z = g\sp{-1} U_{x_0} g $ as
\[
\phi\colon U_z \times p\sp{-1} (x_0)\ra p\sp{-1} (U_z),\quad
(x,y)\mapsto G T(g^{-1} x g) y T(g^{-1} x g)^{-1} G^{-1}.
\]
The left composition with $ p $ is the projection onto the first factor of
the product $ U_z \times p\sp{-1} (x_0) $. In fact,
\[
\begin{split}
\prc \circ\phi(x,y) &= 
p(G) p\big(T(g^{-1} x g) y T(g ^{-1} x g)^{-1}\big) p(G)^{-1} \\
&= g p\big(T(g^{-1} x g)\big) x_0 p\big(T(g^{-1} x g)^{-1}\big) g^{-1}\\
&= g g^{-1} x g g\sp{-1} = x.
\end{split}
\]
\end{proof}
\section{The essential Grassmannian}
In $ \mathcal{P}(E) $ and $ G_s (E) $ we consider the relation of compact 
perturbation. We write $ X\sim_c Y $ if and only if $ X $ is compact
perturbation of $ Y $ in the sense of Definition \ref{compact_perturbation}
and $ P\sim_c Q $ if and only if they have compact difference.
Given $ X\in G_s (E) $ and $ P\in\mathcal{P}(E) $ we define
\begin{align*}
\mathcal{P}_c (P;E) &= \set{Q\in\mathcal{P}(E)}{P\sim_c Q} \\
G_c (X;E) &= \set{Y\in G_s (E)}{X\sim_c Y}
\glsadd{labPcE}
\glsadd{labGcE}
\end{align*}
endowed with the topology of subspace. We denote by $ \mathcal{P}_e (E) $ and
$ G_e (E) $ the quotient spaces, endowed with the quotient topology.% 
\glsadd{labPeE}
\glsadd{labGeE}
In literature the latter is called \textsl{essential Grassmannian}, 
check, for instance, \cite{AM03a}, \S 6.
Let $ \Pi_e $ and $ \pi_e $ denote be the projections onto the quotient 
spaces of $ \mathcal{P}(E) $ onto $ \mathcal{P}_e (E) $ and
$ G_s (E) $ onto $ G_e (E) $, respectively. By Theorem \ref{Calkin_bundle}, 
the map
\[
\prc\colon\mathcal{P}(E)\rightarrow\mathcal{P}(\mathcal{C})
\]
has local sections, hence is open. Moreover, two projectors belong to the same 
class of compact perturbation if and only if their difference is compact, 
hence the map induced to the quotient
\[
\prc_e\colon\mathcal{P}_e (E)\rightarrow\mathcal{P}(\mathcal{C})
\]
\glsadd{labprc}
is a homeomorphism. If $ \Pi_e (P) = \Pi_e (Q) $ the operator $ P - Q $
is compact. Thus, $ \ran P\sim_c \ran Q $ and we have a well defined
map
\[
r_e\colon\mathcal{P}_e (E)\ra G_e (E),\quad \Pi_e (P)\mapsto \pi_e (\ran P).
\glsadd{labre}
\]
It is quotient map, because obtained as composition of
quotient maps.
\begin{proposition}
There is a homeomorphism between $ G_e (E) $ and $ Gr(\mathcal{C}) $ such
that the diagram
\[
\xymatrix{
\mathcal{P}_e (E) \ar[d]^{r_e} \ar[r]^-{\prc_e} & 
\mathcal{P}(\mathcal{C}) \ar[d]^{\pi} \\
G_e (E) \ar[r]^{} & Gr(\mathcal{C})}
\]
commutes.
\end{proposition}
\begin{proof}
Let $ P $ and $ Q $ be projectors such that 
$ \pi_e (\Pi_e (P)) = \pi_e (\Pi_e (P)) $. Hence 
$ \Pi_e (P)\sim_c \Pi_e (Q) $, that is, $ PQ - Q $ and $ QP - P $
are compact operators, thus 
\[
p(P)p(Q) = p(Q),\quad p(Q)p(P) = p(P),
\]
hence, $ \pi(p(P)) = \pi(p(Q)) $. By following each of the steps above
in the opposite order, it is easy to check that, if $ \pi(p(P)) = \pi(p(Q)) $,
then $ \pi_e (\Pi_e (P)) = \pi_e (\Pi_e (P)) $. Thus, given $ X\in G_s (E) $
and $ P $ such that $ \ran P = X $, we have a well defined and injective
map
\[
g_e \colon G_e (E)\ra Gr(\mathcal{C}),\quad g_e (\pi_e (X)) = \pi(p(P)).
\]
Since $ \pi $ and $ p_e $ are surjective, $ g_e $ is also surjective.
By definition, $ \pi\circ p_e = g_e \circ r_e $. We prove that
$ g_e $ is continuous. Given $ U\subset Gr(\mathcal{C}) $, then
\[
g_e \sp{-1} (U)\text{ is open iff. } r_e \sp{-1} (g_e \sp{-1} (U))
\]
is open, because the quotient topology is the finest making $ r_e $.
The latter is $ (\pi\circ p_e)\sp{-1} (U) $, which is open. 
Since $ \pi $ is also a quotient map, the continuity of the inverse 
follows.
\end{proof}
Since, by \cite{PR87}, \S 3, $ \pi $ is a homotopy equivalence, 
$ r_e $ is also a homotopy equivalence. A homotopy inverse of $ r_e $ is 
$ \prc_e ^{-1} s g_e $ where $ s $ is a right inverse
of $ \pi $. We conclude this section by showing that the spaces $ G_c $
and $ \mathcal{P}_c $ have the same homotopy type. 
\begin{proposition}
\label{relative_section}
Let $ X\in G_s (E) $ be a closed complemented subspace and $ P $ a 
projector with range $ X $. The restriction of $ r $ to 
$ \mathcal{P}_c (P;E) $ takes values in $ G_c (X;E) $ and is a homotopy 
equivalence.
\end{proposition}
\begin{proof}
Let $ r_c $ be the restriction of $ r $. %
\glsadd{labrc}
To achieve this result we follow the same steps of Lemma 
\ref{projectors_Grassmannian}. Fix $ X_0 $ compact 
perturbation of $ X $. By Theorem \ref{suitable_projector} there exists 
a projector $ P_0 $ with range $ X_0 $ such that $ P_0 - P $ is compact. Call
$ Y_0 $ its kernel and define the local section
\[
s_0\colon B(X_0,\hat{\gamma}(X_0,Y_0))\rightarrow P(E), \ 
X'\mapsto P(X',Y_0).
\]
This is continuous by Proposition \ref{GE_is_open}. Since 
$ r(s_0 (X')) = X' $, by Proposition \ref{equivalence_of_definitions} the
operators $ (I - s_0 (X')) P_0 $ and $ (I - P_0) s_0 (X') $ are compact. The 
relation $ \ker s_0 (X') = \ker P_0 $ implies $ s(X') (I - P_0) = 0 $, 
therefore
\[
P_0 - s(X') = (I - s(X')) P_0 + (P_0 - s(X'))(I - P_0) = 
(I - s(X')) P_0
\]
which is compact. Then $ s(X') - P $ is compact. Let 
$ \mathcal{U} = \set{U_i}{i\in I} $ be a locally finite refinement of 
$ \set{B(X_0,\hat{\gamma}(X_0,Y_0))}{X_0\in G_c (X;E)} $ and
$ \set{\vfi_i}{i\in I} $ a partition of unit subordinate to $ \mathcal{U} $. 
Then, for any $ Y\in G_c (X;E) $
\[
s(Y) - P = \sum_{i\in I} \vfi_i (Y) (s_i (Y) - P)
\]
is a finite sum of compact operators. The convex combination of $ s\circ r_c $ 
and $ id $ is a homotopy map.
\end{proof}
\section{The Fredholm group}
\label{fredholm_group}
We call \textsl{Fredholm group} the set of invertible operators on a Banach
space that can be written as sum of the identity and a compact operator.
It is a normal subgroup of $ GL(E) $. The Fredholm group is endowed with the
norm topology; we denote it by $ GL_c (E) $. 
\glsadd{labGLcE}
\begin{theorem}
If $ E $ is an infinite dimensional Banach space over a field 
$ \mathbb{F} $, that is $ \mathbb{R} $ or $ \mathbb{C} $,  the Fredholm group 
has the homotopy type of $ \Lim GL(n,\mathbb{F}) $.
\end{theorem}
For the proof see, for instance, \cite{Geb68}. The homotopy groups of the 
Fredholm group are, in the real and complex case, respectively
\begin{align}
\label{fredholm_rgroup}
\pi_i (GL(\infty,\mathbb{R}))&\cong\left\{
\begin{array}{ll}
\mathbb{Z}_2 & i\equiv 0,1 \ \mod 8 \\
0 & i\equiv 2,4,5,6 \mod 8\\
\mathbb{Z} & i\equiv 3,7 \mod 8
\end{array}
\right.\\
\label{fredholm_cgroup}
\pi_i (GL(\infty,\mathbb{C}))&\cong\left\{
\begin{array}{ll}
0 & i\equiv 0 \mod 2 \\
\mathbb{Z} & i\equiv 1 \mod 2
\end{array}
\right.
\end{align}
see \textsc{Theorem} II of \cite{Bot59}. The spectrum of $ T\in GL_c (E) $ is 
countable, and $ \sigma(T)\setminus\{1\} $ is made of eigenvalues of finite 
multiplicity. When $ E $ is a real Banach space it is defined the 
\textsl{Leray-Schauder} degree as
\[
\deg(T) = (-1)^{\beta (T)}
\glsadd{labdT}
\]
where $ \beta(T) $ is the sum of the algebraic multiplicities of the 
eigenvalues of $ T $ such that $ \re z > 1 $
eigenvalues. It is well defined on the connected components of $ GL_c (E) $
and defines a group isomorphism
\[
\deg\colon\pi_0 (GL_c (E))\rightarrow\{-1,+1\}\cong\mathbb{Z}_2.
\]
See \cite{Llo78} for details. The L.S. degree will help us to determine the 
connected components of $ G_c (X;E) $ when $ E $ is a real or complex Banach
space. We will prove that $ G_c (X;E) $ consists of infinitely numerable
components; these are
\begin{gather}
\label{conn_comp}
G_k (X;E) = \set{Y\in G_c (X;E)}{\dim(X,Y) = k}, \ k\in\mathbb{Z}.
\glsadd{labGkX}
\end{gather}
\begin{lemma}
\label{trans_act}
The Fredholm group acts transitively on each $ G_k (X;E) $ by the left
multiplication. Moreover, there are local sections of the action.
\end{lemma}
The carrying out of the proof follows the same steps of the Hilbert case
outlined in \cite{AM03a}, \S 5. 
\begin{proof}
Let $ Y\in G_k $ and $ T\in GL_c (E) $. Let 
$ t $ be the restriction of $ T $ to $ Y $ and $ i\colon Y\hookrightarrow E $
the inclusion. Both $ t,i\in\mathcal{L}(Y,E) $ are injective and $ t - i $ 
is compact. Hence, by Proposition \ref{cp:2} $ \ran t $ and $ \ran i $ are 
compact perturbation of each other and
\[
\dim(Y,T Y) = \dim(\ran i,\ran t) = \dim(\ker t,\ker i) = 0.
\]
Hence $ TY\in G_k $. Let $ Y,Z\in G_k (X;E) $, hence $ \dim(Y,Z) = 0 $. Let 
$ s $ be a continuous right inverse of $ r_c $ as in Proposition 
\ref{relative_section}. The operator $ s(Z) - s(Y) $ is compact, call it
$ K $. Observe that the restriction of $ s(Z) $ to $ Y $, considered as
operator with values in $ Z $, is Fredholm. Similarly we can consider the 
restriction of $ I - s(Z) $ to $ Y' := \ker s(Y) $ with values in 
$ \ker s(Z) $. For every $ y $ in $ Y $ and $ y'\in Y' $ we can write
\begin{align*}
&s(Z)y = s(Y)y + Ky = (I + K) y \\
&(I - s(Z)) y' = (I - s(Y)) y' - K y' = (I - K)y'.
\end{align*}
The Fredholm index of these operators is $ 0 $ by definition of relative
dimension. Fredholm applications of index $ 0 $ have a nice property: they are
perturbation of an isomorphism by a finite rank operator.
Then we can choose $ R_1 $ in $ \mathcal{L}(Y,Z) $ and $ R_2 $ in 
$ \mathcal{L}(Y',\ker s(Z)) $ suitable finite rank operators. 
Call $ T $ the operator obtained as direct sum of the two isomorphisms 
$ \res{s(Z)}{Y} + R_1 $ and $ \res{(I - s(Z))}{Y'} + R_2 $. It is 
invertible, maps $ Y $ onto $ Z $ and can be written as
\[
I + (K + R_1)s(Y) - (K - R_2)(I - s(Y)) 
\]
hence belongs to the Fredholm group. This proves that the action is transitive.
\vskip .2em
Given $ Y\in G_k $ we build a local section around $ Y $ as follows: let $ s $
be a continuous section as in Proposition \ref{relative_section}. There exists 
$ \var > 0 $ such that, for any $ Z\in B(Y,\var) $ the operator 
$ g(s(Z),s(Y)) $ is invertible. By (\ref{eq:lpq}) and (\ref{conjugation}),
\[
g(s(Z),s(Y))\in I + \mathcal{L}_c (E),
\]
thus, $ g(s(Z),s(Y))\in GL_c (E) $ and $ g(s(Z),s(Y)) Y = Z $. Then a local 
section of the action is defined as
\begin{gather}
\label{trans_act:1}
B(Y,\var)\rightarrow GL_c (E)\times G_k, \ Z\mapsto (g(s(Z),s(Y)),Y).
\end{gather}
\end{proof}
\begin{theorem}
\label{components}
The connected components of $ G_c (X;E) $ are $ G_k (X;E) $ with $ k $ in 
$ \mathbb{Z} $. 
\end{theorem}
\begin{proof}
Let $ Y,Z\in G_c (X;E) $, connected by an arc, $ k = \dim(X,Y) $. By
Proposition \ref{relative_section},  there exists a path 
$ \alpha $ in $ \mathcal{P}_c (P;E) $ that connects $ s(Y) $ to $ s(Z) $. 
Let $ g $ be as in ii) of Proposition \ref{pairs_of_projectors}, thus
$ g\in GL_c (E) $. By Lemma \ref{trans_act}, $ g(Y)\in G_k $.
Conversely, consider $ Y,Z\in G_k $. Hence $ \dim (Y,Z) = 0 $ and, by Lemma 
\ref{trans_act}, there exists $ T\in GL_c (E) $ such that $ T Y = Z $. 
If $ E $ is a complex Banach space, the Fredholm group is arcwise connected. 
Given a path $ \alpha $ that 
connects $ I $ to $ T $ the path $ \alpha(t) Y $ connects 
$ Y $ to $ Z $. If $ E $ is a real Banach space, let 
$ S\in GL(Y)\times GL(Y') $, where $ Y\oplus Y' = E $, and 
$ \deg(S) = -\deg(T) $. Then $ TS $ maps $ Y $ onto $ Z $ and is connected to 
the identity operator and we conclude as in the complex case.
\end{proof}
\section{The Stiefel space}
In this section we introduce the \textsl{Stiefel spaces} and for some
$ X\in G_s (E) $, we compute its homotopy type. That will help us to
determine the homotopy groups of $ G_c (X;E) $.
\begin{definition}
Let $ X\in G_s (E) $. We define the \emph{Stiefel space}, and denote it by 
$ St(X;E) $, the set of the bounded operators $ f\in\mathcal{L}(X,E) $
such that 
\begin{enumerate}
\item $ f(X) $ is complemented in $ E $;
\item $ f $ is injective;
\item $ f - i $ is compact,
\end{enumerate}
where $ i\colon X\hookrightarrow E $ is the inclusion. On it we consider
the topology of subspace.
\glsadd{labStX}
\end{definition}
The Stiefel space is an analytical manifold because is an open subset of
the affine space $ I + \mathcal{L}_c (X,E) $. We recall some results on the
homotopy type of $ St(X;E) $.
\begin{theorem}{\rm (refer \cite{DD63})}
If $ X $ is a finite-dimensional subspace of $ E $ $ St(X;E) $ is contractible.
\end{theorem}
Using the techniques of \cite{Geb68} it is possible to prove that
when $ X $ has infinite dimension and infinite co-dimension $ St(X;E) $ is 
contractible. Then, if $ X $ has infinite co-dimension $ St(X;E) $ is 
always contractible. The next result is known for Hilbert spaces, see for
example \cite{AM03a} \S 5. The generalization to Banach spaces requires, as
Lemma \ref{trans_act} does, the Proposition \ref{relative_section}.
\begin{theorem}
\label{stiefel}
Let $ r_{St}\colon St(X;E)\rightarrow G_0 (X;E) $ be the continuous map
defined as $ r_{St} f = f(X) $. Then $ (St(X;E),r_{St},G_0 (X;E),GL_c (X)) $ 
is a principal fiber bundle. The action of $ GL_c (X) $ onto itself is the left
multiplication.
\glsadd{labrSt}
\end{theorem}
\begin{proof}
As first step we build a local section around $ X $. Consider a continuous
map as in Proposition \ref{relative_section}. Let $ U $ be an open
neighbourhood of $ X $ where (\ref{trans_act:1}) is defined. Define
\[
\gamma_0 \colon U\rightarrow St(X;E), \ Y\mapsto g(s(Y),s(X))_{|X}.
\]
This suffices to build an open cover of coordinate neighbourhoods of $ G_0 $.
Given $ Y\in G_0 $, by Lemma \ref{trans_act} there exists $ T\in GL_c (E) $
such that $ T X = Y $. Then a trivialization of $ T (U) $ and its inverse
are given by
\begin{gather*}
\phi\colon T(U)\times GL_c (X)\rightarrow r_{St} ^{-1} (T (U)),\\
(Y',g)\mapsto T \gamma_0 (T ^{-1} Y') g;\\
\phi\sp{-1} \colon r_{St} ^{-1} (T (U))\rightarrow T(U)\times GL_c (X),\\
f\mapsto\big(T f(X),g(s(X),s(f(X)))\circ T\sp{-1} f\big).
\end{gather*}
We have to check that whenever two coordinate neighbourhoods $ U_i, U_j $
intersect, for every $ Z\in U_i \cap U_j $ the transitions maps are
left translations of $ GL_c (X) $ onto itself. In fact, given 
$ T_i $, $ T_j $ such that $ T_i X = T_j X = Z $ the transition map is
\[
\phi_{j,Z} ^{-1} \phi_{i,Z} (g) = 
g(s(X),s(T_j Z)) T_j \sp {-1} T_i \gamma_0 (T_i \sp{-1} Z) \cdot g
\]
is the left multiplication by an element of $ GL_c (X) $. Then
we have $ GL_c (X) $ compatibility.
\end{proof}
When $ X \subset E $ has infinite co-dimension and infinite dimension the 
exact sequence of the principal bundle 
$ (St(X;E),r_{St},G_0 (X;E),GL_c (X)) $ gives isomorphisms
\begin{equation}
\label{eq:stiefel}
\pi_i (G_0 (X;E),X)\cong \pi_{i - 1} (GL_c (X))\cong\pi_{i - 1} 
(GL(\mathbb{F},\infty)), \ i\geq 1 
\end{equation}
where $ \mathbb{F} $ is the real or complex field.
\section{The index of the exact sequence}
Using exact sequence of the fiber bundle 
$ (\mathcal{P}(E),\mathcal{P}(\mathcal{C}),\prc) $ we show how to
associate an integer to a closed loop in the space of idempotents of 
$ \mathcal{C}(E) $. In fact we define a group homomorphism on 
$ \pi_1 (\mathcal{P}(\mathcal{C})) $ denoted by $ \varphi $.
Since $ \mathcal{P}(\mathcal{C}) $ is homotopically equivalent to the 
space of essentially hyperbolic operators on $ E $, we definitely have
a group homomorphism on $ \pi_1 (e\mathcal{H}(E)) $ obtained as the
composition of $ \varphi $ with $ \Psi $, defined in \S 2.1.\par
Let $ P $ be any projector. By Theorem \ref{Calkin_bundle} the triple 
$ (\al{P}(E),\prc,\al{P}(\al{C})) $ is a locally 
trivial bundle. The typical fiber of $ p(P) $ is 
$ \al{P}_c (P;E) $. Then we have an exact sequence 
\[
\label{sequence_3}
\xymatrix{
\pi_1 (\al{P}_c (P;E),P) \ar[r]^-{i_*} &
\pi_1 (\al{P}(E),P) \ar[r]^-{\prc_*} &
\pi_1 (\al{P}(\al{C}),p(P)) 
}
\]
\begin{theorem}
\label{index_homomorphism}
There exists a group homomorphism 
$ \vfi_P\colon\pi_1 (\al{P}(\al{C}),p(P))\rightarrow\mathbb{Z} $ such that
the sequence of homomorphisms 
\[
\xymatrix{
\pi_1 (\al{P}(E),P) \ar[r]^-{\prc_*} &
\pi_1 (\al{P}(\al{C}),p(P)) \ar[r]^-{\vfi_P} &
\mb{Z}
}
\]
is exact. 
\end{theorem}
\begin{proof}
The homomorphism is defined as follows: given a loop 
$ a\in\mathcal{P}(\mathcal{C}) $, there exists a path 
$ \beta\in\mathcal{P}(E) $ such that $ p\circ\beta = a $. Thus
\[
p(\beta(0)) = p(\beta(1)),\ \beta(0) - \beta(1)\text{ is cpt. }
\]
We define $ \varphi(a) = \dim(\beta(1),\beta(0)) $. First, we observe that
the definition does not depend on the choice of the lifting path. In fact,
given $ \beta' $ as above, $ \dim(\beta(t),\beta'(t)) $ is constant. 
This follows from the Theorem
\ref{stability_dimension}. Hence,
\begin{equation}
\label{eq:well-defined}
\dim(\beta'(1),\beta'(0)) = \dim(\beta(1),\beta(0))
\end{equation}
We prove that $ \vfi_P $ is a group homomorphism. Let
$ a,b $ be two closed paths at the base point 
$ p(P) $. There are two lifting paths $ \alpha,\beta $ such that
\begin{align*}
\alpha(0) &= P, \ \prc\circ\alpha = a, \\ 
\beta(0) &= P, \ \prc\circ\beta = b.
\end{align*}
There also exists $ \beta' $ such that $ \beta'(0) = \alpha (1) $ and
$ p\circ\beta' = b $. Define
\[
\gamma = \alpha * \beta', \ \gamma(0) = \alpha(1)
\]
which is a lifting path for $ a * b $. Since $ \beta $ and $ \beta' $ are 
lifts of the same path $ b $, equality (\ref{eq:well-defined}) holds.
We have
\[
\begin{split}
\vfi_P(a * b) &= \dim(\beta' (1), \alpha(0)) =
\dim(\beta'(1),\beta'(0)) + \dim(\alpha(1),\alpha(0)) \\
&= \dim(\beta(1),\beta(0)) + \dim(\alpha(1),\alpha(0)) = \vfi_P(a) + \vfi_P(b).
\end{split}
\]
Finally, we prove that the sequence above is exact. 

\noindent $ \ker\vfi_P\subseteq\text{Im}(p_*) $. Suppose $ \vfi_P(a) = 0 $.
Hence, $ \dim(\beta(1),P) = 0 $. By Theorem \ref{components}, there exists a 
path $ \gamma\in\mathcal{P}_c (E;P) $ such that
\[
\gamma(0) = \beta(1),\ \gamma(1) = P,\ p_* (\beta *\gamma) = a.
\]
\noindent $ \text{Im}(p_*)\subseteq\ker\vfi_P $. Given a loop in 
$ \alpha\in\mathcal{P}_c (E;P) $, we have 
$ \vfi_P(p_* (a)) = \dim(\alpha(1),P) = 0 $.
\end{proof}
\begin{corollary}
Given $ P\in\mathcal{P}(E) $, we have the following properties of
the kernel and image of $ \vfi_P $:
\begin{itemize}
\item[h1)] $ m\in\text{Im}(\vfi_P) $ if and only
if there exists a projector $ Q $ in the same connected component of
$ P $ such that $ Q - P $ is compact and $ \dim(Q,P) = m $;
\item[h2)] the connected component of $ P $ in $ \mathcal{P}(E) $ is 
simply-connected.
\end{itemize}
\end{corollary}
Property h2) follows straightforwardly from the exactness of the sequence.
Let $ m\in\text{Im}(\vfi_P) $. Hence, there exists a path $ \beta $ of
projectors such that $ \dim(\beta(1),P) = m $. Thus, we choose 
$ Q = \beta(1) $. Conversely, let $ Q $ be a projector in the same connected
component of $ P $ such that $ Q - P $ is compact. If $ \beta $ joins
$ P $ to $ Q $, $ p\circ\beta $ is a loop and $ \vfi_P (p\circ\beta) = m $.
\vskip .4em
\subsubsection*{When the index is trivial}
When $ P $ is a projector whose image has finite dimension or
finite co-dimension its component in $ \mathcal{P}(\mathcal{C}) $ consists
of a single point, hence $ \vfi_P $ is the null homomorphism.
There are infinite-dimensional spaces, 
called \textsl{undecomposable}, where the only complemented subspaces have 
finite dimension or finite-co-dimension; an example of such space was 
described by W.~T.~Gowers and B.~Maurey in \cite{GM93}. In that case 
$ \mathcal{P}(\mathcal{C}) $ consists of two points. We observe that in h2),
$ \ran Q \cong\ran P $, by ii) of Proposition \ref{pairs_of_projectors}.
In \cite{GM93} W.~T.~Gowers and B.~Maurey showed that there are 
infinite-dimensional Banach spaces which are not isomorphic to any of their
proper subspaces. In this, case $ \vfi_P $ is the null homomorphism. 
\vskip .4em
The next lemma is needed in order to exhibit a wide class of examples where 
h1) conditions holds. By sake of completeness we exhibit a proof of it.
Such proof follows also from \cite{PR87}.
\begin{lemma}
\label{lem:X+Y}
Let $ E $ be a Banach space, and $ X,Y\subset E $ closed subspaces such
that $ X\cong Y $ and $ X\oplus Y = E $. Two projectors
$ P_X,P_Y $ with ranges $ X $ and $ Y $, are connected by a continuous
path on $ \mathcal{P}(E) $.
\end{lemma}
\begin{proof}
It is enough to prove it when $ P_X $ is the projector onto $ X $ along $ Y $ 
and $ P_Y $ is $ I - P_X $, because the set of projectors having a fixed
range is a convex subset of the space of projectors. Check, for instance
Lemma \ref{projectors_Grassmannian}.

Let $ \sigma $ be an isomorphism of $ Y $ with $ X $. We define the 
path
\[
G_{\sigma,\theta} (x + y) = (\cos\theta\, x + \sin\theta\,\sigma y) + 
(-\sin\theta\,\sigma^{-1} x + \cos\theta\, y)
\glsadd{labGsg}
\]
of invertible operators of $ E $. Direct computations show that
$ G_{\sigma} (-\theta) $ is its inverse. Moreover $ G_{\sigma} (0) $ is
the identity and $ G_{\sigma} (\pi/2) $ conjugates the projector $ P $ to 
$ P_Y $. Then the path
\[
P_{\theta} = G_{\sigma,\theta} P G_{\sigma,-\theta},
\ \ 0\leq\theta\leq\pi/2
\]
has the required properties.
\end{proof}
Here is a concrete example where the first condition hold.
\begin{proposition}
\label{surjective}
Let $ E = X\oplus Y $ be a Banach space and $ X $ a closed complemented 
subspace isomorphic to its closed subspaces of co-dimension $ m $.
Let $ P $ be the projector onto $ X $ along $ Y $. Then $ P $ satisfies the 
condition {\rm h1)} with the integer $ m $.
\end{proposition}
\begin{proof}
Since $ X $ is isomorphic to its hyperplanes, we can choose subspaces 
$ X\sp m,R_m\subset X $, where $ R $ has dimension $ m $, $ X\sp m $ is closed 
and $ X\sp m\cong X $. We have the decomposition and isomorphism
\[
E = R_m \oplus X\sp m\oplus Y,\ X\sp m\cong Y.
\]
By applying Lemma \ref{lem:X+Y} with $ E = X\sp m \oplus Y $, we obtain that 
$ P_{X\sp m} $ is connected to $ P_Y $. By applying it a second time to $ E $ 
and subspaces $ X = R_m\oplus X\sp m $ and $ Y $, we obtain that $ P_X $ is 
connected to $ P_Y $. Hence, $ P_X $ is connected to 
$ P_{X\sp m} $.
\end{proof}
The argument used to connect the two projectors $ P $ and $ P_X $ is a
modification of the one used for Hilbert spaces by J.~Phillips in 
\textsc{Proposition} 6 of \cite{Phi96} when $ m = 1 $: 
given the decomposition
\[
E = X\sp 1\oplus R_1\oplus Y 
\]
a shift operator $ s $ maps $ X\sp 1 $ and $ R_1\oplus Y $ isomorphically onto 
$ X\sp 1\oplus R_1 $ and $ Y $ respectively. Since the general linear group of
a Hilbert space is contractible the projectors are connected. The isomorphism 
$ G_{\sigma} $ used in the proof is connected to the identity 
regardless of whether $ GL(E) $ is connected or not.
\begin{example}
Given a Banach space $ X $ which is isomorphic to its subspaces of 
co-dimension two, but not to the subspaces of co-dimension one, let $ P $
be the projector onto the first factor of $ E = X\oplus X $. Then, 
by Proposition \ref{surjective}, $ 2\in\text{Im}(\vfi_P) $. However,
$ 1\not\in\text{Im}(\vfi_P) $, because condition h2) with $ m = 1 $
implies the existence of an isomorphism of $ X $ with a hyperplane.
Thus,
\[
\text{Im}(\vfi_P) = 2\mb{Z}\subset\mb{Z}.
\]
An example of such space $ X $ was showed by W.~T.~Gowers and B.~Maurey
in \cite{GM97}. Thus, there are projectors such that the homomorphism
is not surjective, but not trivial.
\end{example}
\vskip .4em
\subsubsection*{When the index is an isomorphism}
When $ E $ is one of the following Banach spaces, an infinite-dimensional
Hilbert space, spaces $ L\sp p (\Omega,\mu) $ for $ p\geq 1 $ and 
$ L\sp\infty (\Omega,\mu) $, or spaces of sequences $ \ell\sp p,m,c_0 $, 
the following three conditions hold:
\begin{enumerate}
\item $ E\cong E\times E $;
\item $ E $ is isomorphic to its hyperplanes;
\item $ GL(E) $ is contractible to a point.
\end{enumerate}
From (i), we can write $ E = X\oplus Y $ where $ X\cong Y $. By (ii),
$ X $ is isomorphic to a hyperplane, thus $ 1\in\text{Im}(\vfi_{P(X,Y)}) $
by Proposition \ref{surjective}. By (\ref{exact_homotopy_sequence:1}) 
with $ k = 1 $, we obtain
the condition h1), because $ F_P $ and $ G_P $ are contractible, by
(iii). Hence, $ \vfi_{P(X,Y)} $ is a group isomorphism.
\subsection*{The Douady space}
%\section{A space where $ \varphi $ is not injective}
\label{non-injective}
We exhibit an example of Banach space $ E $ with a projector
$ P $ of infinite dimensional range and kernel and a loop $ a $ in 
$ \mathcal{P}(\mathcal{C}) $ with base point $ p(P) $ such that
$ \varphi(a) = 0 $ but not homotopically equivalent to the constant path.
%Moreover, $ \vfi_P $ is the null homomorphism.
\begin{proposition}
Let $ X\subset E $ be a complemented subspace isomorphic to its 
complement and $ P $ a projector such that $ P(E) = X $.
If $ GL(X) $ is not connected, the component of $ P $ in $ \mathcal{P}(E) $ 
is not simply connected.
\end{proposition}
\begin{proof}
Choose a topological complement $ Y $ and let $ T\in GL(X) $ be such that 
there exists no path joining $ T $ to the identity. Let $ \sigma $ be 
an isomorphism of $ Y $ with $ X $. Hence the invertible operator
\[
T_1 = 
\begin{pmatrix}
T & 0 \\
0 & \sigma T^{-1} \sigma^{-1}
\end{pmatrix}
\]
lies in the connected component of $ GL(E) $ of the identity. A path can be 
defined as $ G_{\sigma,\theta} T_1 G_{\sigma,-\theta} $ where
$ G_{\sigma,\theta} $ is the operator defined Lemma \ref{lem:X+Y}.
Call $ S $ such path and define $ \alpha = S P S^{-1} $. Since $ T_1 $ 
commutes with $ P $ the path $ \alpha $ is a loop with base point $ P $. 
The group homomorphism
\[
\Delta\colon\pi_1 (\mathcal{P}(E),P) \rightarrow 
\pi_0 (GL(X))\times\pi_0 (GL(Y))
\]
induced by the fiber bundle $ (GL(E),\pi_P,\mathcal{P}(E)) $ 
maps $ \alpha $ to $ T_1 $. Thus $ \Delta\alpha\neq 0 $, hence 
$ \alpha\neq 0 $. 
\end{proof}
In order to find non-contractible loops with vanishing index we need some 
projector $ P $ such that the inclusion
\[
j_* \colon\pi_1 (\mathcal{P}_c (P;E))\rightarrow\pi_1 (\mathcal{P}(E),P)
\]
is not surjective. We will prove that for some spaces the second group
contains infinitely many distinct elements, while the first is a finite
group, according to (\ref{eq:stiefel}).
Let $ F $ and $ G $ be such that
\begin{enumerate}
\item every bounded map $ G\rightarrow F $ is compact,
\item both $ F $ and $ G $ are isomorphic to their hyperplanes;
\end{enumerate}
let $ X = F\oplus G $. Let $ T\in GL(X) $. We can write it block-wise using 
the projectors on $ F $ and $ G $
\[
T =
\begin{pmatrix}
A & B\\
C & D
\end{pmatrix}
\]
Since $ C $ is compact, it is possible to prove that $ A $ and $ B $
are Fredholm operators such that $ \ind(A) + \ind(D) = 0 $ 
(refer \cite{Mit70}). We define $ f(T) = \ind(A) $. 
We have the following
\begin{lemma}{\rm (refer \cite{Dou65}).}
The map $ f\colon GL(X)\rightarrow\mathbb{Z} $ is continuous and surjective.
\end{lemma}
We define $ E = X\oplus X $. 
By the lemma, we have a surjective homomorphism, obtained by composition
\[
\xymatrix{
(f\times 0)\circ\Delta\colon\pi_1 (\mathcal{P}(E),P)\rightarrow\mathbb{Z}.
}
\]
Hence, given a loop $ \alpha\not\in j_* (\pi_1 (\mathcal{P}_c (P;E))) $, 
the element $ a = \prc_* (\alpha) $ is non trivial and, since the sequence in 
Theorem (\ref{index_homomorphism}) is exact, $ \vfi_P (a) = 0 $.
\vskip .2em
A pair of spaces with the properties i) and ii) is given by 
$ (\ell\sp p,\ell\sp q) $
with $ p > q > 1 $; refer \textsc{Theorem} 4.23 of \cite{Rya02} for
property i). Isomorphisms with hyperplanes can be defined using the
operators $ (sx)_1 = 0 $, $ (sx)_i = x_{i - 1},i\geq 2 $ for every 
$ x\in\ell\sp p $. Thus, in the space
\[
E = (\ell\sp p\oplus\ell\sp 2)\oplus(\ell\sp p \oplus\ell\sp 2),
\]
if we call $ P $ the projector onto the first factor, then $ \vfi_P $ is
not injective. Finally, we observe that the image of $ P $ is isomorphic
to a hyperplane. Thus, by Proposition \ref{surjective}, 
$ \text{Im}(\vfi_P) = \mb{Z} $.
%%% TEXEXPAND: END FILE ./chapter2-v2.tex
%%% TEXEXPAND: INCLUDED FILE MARKER ./chapter3-v2.tex
\chapter{Linear equations in Banach spaces}
\label{chap3}
We state and prove some general results about differential equations on
a Banach algebra with unit, usually denoted by $ 1 $. We are mainly concerned 
of the Cauchy problem
\begin{gather}
\label{init}
u'(t) = A(t) u(t), \ u(0) = 1
\end{gather}
where $ A $ is a continuous path in a Banach algebra $ \mathcal{B} $. Local 
existence and uniqueness hold. In fact these solutions admit a prolongation 
to the whole real line $ \mathbb{R} $. Denote by $ X_A $ the solution of 
(\ref{init}).
\glsadd{labXA}
Using local uniqueness we prove some properties of the solution $ X_A $. When 
$ \mathcal{B} $ is the algebra of bounded operators on a Banach space $ E $ 
two linear subspaces, the \textsl{stable} and \textsl{unstable space}, are 
defined
\begin{align*}
W_A ^s &= \set{x\in E}{\lim_{t\ra +\infty} X_A (t) x = 0}\\
W_A ^u &= \set{x\in E}{\lim_{t\ra -\infty} X_A (t) x = 0}.
\glsadd{labWAs}
\glsadd{labWAu}
\end{align*}
If $ A $ is a bounded and asymptotically hyperbolic these are closed linear 
subspaces, admit a topological complement, and have the asymptotic behaviour
\begin{align*}
&\lim_{t\rightarrow +\infty} X_A (t) W_A ^s = E^- (A_0 (+\infty)), \\
&\lim_{t\rightarrow +\infty} X_A (t) Y = E^+ (A_0 (+\infty))
\end{align*}
where $ W_A ^s \oplus Y = E $. The limits are taken in the topology of 
$ G(E) $. In the last section we look at the effects of
perturbation of an asymptotically hyperbolic path on its stable space.
Precisely the stable space varies continuously in the topology of $ G_s (E) $.
If $ A - B $ is a path of compact operators then $ W^s _A $ and $ W^s _B $
are compact perturbation one of each other.
\section{The Cauchy problem}
Let $ \mathcal{B} $ be a Banach algebra and $ A $ a continuous 
path defined on the real line. Given $ u,v\in\mathcal{B} $ we can always 
consider two 
\textsl{Cauchy problems}
\begin{align}
{X_{A,u}}^{\prime} (t) &= A(t) X_{A,u} (t), \ X_{A,u} (0) = u \label{cauchyl}\\
{X^{A,v}}^{\prime} (t) &= X^{A,v} (t) A(t), \ X^{A,v} (0) = v \label{cauchyr}.
\end{align}
By Theorem \ref{Cau} unique local solutions always exist and the maximal
solutions can extended, by Proposition \ref{whole}, to $ \mathbb{R} $.
\begin{proposition}
\label{inv}
Let $ u, v\in\mathcal{B} $. We have
\begin{align*}
X^{-A,v} (t) \cdot X_{A,u} (t) &= vu, \\ 
X_{A,u} (t) \cdot X^{- A,v} (t) &= X_{A,1} (t) \cdot uv \cdot X^{-A,1} (t)
\end{align*}
for every $ t\in\R $. Moreover $ X_{A,1} $ is invertible and its inverse is
$ X^{-A,1} $.
\end{proposition}
\begin{proof}
To prove the first equality consider the $ C^1 $ path 
$ X^{-A,v} \cdot X_{A,u} $. By hypothesis the path is $ vu $ at $ t = 0 $ and
its derivative is
\begin{equation*} 
 \begin{split}
     {X^{-A,v}}' X_{A,u} + X^{-A,v} X_{A,u} ' = 
   - X^{-A,v} A X_{A,u} + X^{-A,v} A X_{A,u} = 0;
 \end{split}
 \end{equation*}
then $ X^{-A,v} X_{A,u} (t) = vu $ for every $ t $. To prove the second
we argue similarly. The path $ X_{A,v} (t) \cdot X^{-A,u} (t) $ has derivative 
\begin{equation*}
{ X_{A,v} }' X^{-A, u} + X_{A,v} { X^{-A,u} }' = [A,X_{A,v}\cdot X^{- A,u}]
\end{equation*}
and is therefore solution of the Cauchy problem
$ X' = [A,X] $ with starting point at $ uv $. By direct computation 
$ X_{A,1} \cdot uv\cdot X^{-A,1} $ solves the same equation. By uniqueness 
the second equality holds. The first equality applied to $ u = v = 1 $ gives
$ X^{-A,1} \cdot X_{A,1} = 1 $. Since $ X_{A,1}\cdot X^{-A,1} $ and the 
constant path $ 1 $ solve the same equation, $ X_{A,1} $ is invertible.
\end{proof}
\begin{definition}
An element $ u\in\mathcal{B} $ is a \emph{left inverse} if there exists 
$ v $, called \emph{right inverse} for $ u $, such that $ uv = 1 $. We denote 
the subsets of left and right inverses by $ \mathcal{B}_l $ and 
$ \mathcal{B}_r $ respectively.
\glsadd{labBl}
\glsadd{labBr}
\end{definition}
\begin{proposition}
\label{lr}
If $ u\in\al{B}_r $ (resp. $ \mathcal{B}_l $)  then 
$ X_{A,u} \subset\mathcal{B}_r $ (resp. $ \mathcal{B}_l $). If $ u $ is
invertible then $ X_{A,u} (t) ^{-1} = X^{-A,u^{-1}} (t) $.
\end{proposition}
\begin{proof}
Let $ u\in\mathcal{B}_r $ and $ v $ be such that $ vu = 1 $. By the first 
equality of Proposition \ref{inv} $ X_{A,u}\subset\mathcal{B}_r $. If 
$ u\in G(\mathcal{B}) $ let $ v $ be its inverse. The first and the second of 
\ref{inv} give $ X^{-A,v}\cdot X_{A,u} = X_{A,u}\cdot X^{-A,v} = 1 $.
\end{proof}
We will abbreviate the notation for the rest of this section: for curves
with starting point $ 1 $ we write $ X_A $ instead of $ X_{A,1} $.
\begin{proposition}
\label{op}
$ \al{B}_r $ and $ \al{B}_l $ are open subsets of $ \al{B} $.
\end{proposition}
\begin{proof}
We will prove that $ \al{B}_r $ is open. Let $ u\in \al{B}_r $ and $ v $ 
be such that $ v\cdot u = 1 $. Let $ r_0 = 1/\no{v} $
and $ h\in\al{B} $. Then
\[
v (u + h) = v u + v h = 1 + vh.
\]
If $ h\in B(u,r_0) $, by the Von Neumann series, $ 1 + vh $ is 
invertible. Then $ (1 + vh)^{-1} v $ is a left inverse of $ u + h $. 
Actually, in a neighbourhood of $ u $, we have defined a smooth function
\begin{gather}
\label{local_left_inverse}
B(u,r_0)\rightarrow\mathcal{B}_l, \ u'\mapsto L_{u,r_0} (u') = 
[1 + v(u' - u)]^{-1} v\in \al{B}_l.
\end{gather}
such that $ L_{u,r_0} (u^{\prime}) \cdot u = 1 $. The same conclusions hold 
for $ \al{B}_l $.
\end{proof}
\begin{proposition}
\label{incon}
Let $ X\in C^1 (\R, \al{B}_r) $. There exists $ A\in C(\R,\al{B}) $ such that
$ X_{A,X(0)} = X $.
\end{proposition}
\begin{proof}
As first step we prove that there exists a path $ Y $ with values in 
$ \mathcal{B}_l $ such that $ YX \equiv 1 $. Let $ t_0\in\R $. Since 
$ X(t_0)\in \al{B}_r $ there exists $ Y(t_0) $ such that 
$ Y(t_0) X(t_0) = 1 $ and the (\ref{local_left_inverse}) provides us with a 
differentiable map defined in a neighbourhood $ B(t_0,\var(t_0)) $, namely 
$ L_{X(t_0),\var(t_0)} $. By paracompactness of $\R $ we can extract a 
locally finite sub-covering of $ \set{B(t,\var(t))}{t\in\R} $, say 
$ \al{U} = \set{U_i}{i\in I} $. Let $ \sigma\colon I\rightarrow\R $ be a 
choice function and $ \set{\vfi_i}{\supp\vfi_i \subseteq U_i} $ a partition 
of unity subordinate to $ \al{U} $. Then set
\begin{gather*}
  Y = \sum_i \vfi_i Y_{\sigma(i)}.
\end{gather*}
Actually $ Y $ is infinitely differentiable. Its image lies in $ \al{B}_l $, 
in fact
\begin{gather*}
  Y(t)X (t) = \sum_i \vfi_i Y_{t_i} (t) X(t) = \sum_i \vfi_i (t) 1 = 1.
\end{gather*}
Now, in the chain of equalities $ X'= X'\cdot 1 = X'\cdot YX = (X'Y) X $
set $ A = X'Y $ and obtain $ X' = AX $. By uniqueness, $ X = X_{A,X(0)} $
 q.e.d.
\end{proof} 
This proposition gives us a characterization of the solutions of 
$ X' = AX $ when the starting point lies in $ \al{B}_r $ (resp. $ \al{B}_l $). 
They are just $ C^1 $ curves on $ \al{B}_r $ (resp. $ \al{B}_l $). 
\begin{proposition}
\label{top}
$  G(\mathcal{B}) $ is union of connected components of $ \al{B}_r $.
\end{proposition}
\begin{proof}
Let $ \al{B}_r ' $ be a connected component of $ \al{B}_r $ such that 
$  G(\mathcal{B}) \cap \al{B}^{\prime} _r \neq\emptyset $. 
Let $ x\in \al{B}_r ' $. Since $ \al{B}_r ' $ is an
open set we may  choose a path $ \Gamma\in C^1 ([0,1],\al{B}_r ') $ such that 
$ \Gamma (0) = g\in G(\mathcal{B}) $ and $ \Gamma (1) = x $. Then, 
by Proposition \ref{incon}, $ \Gamma = X_{A,\Gamma(0)} $
for some $ A\in C([0,1],\al{B}) $. Since $ \Gamma (0) $ is an
invertible element of $ \al{B} $ Proposition \ref{inv} states that
$ \Gamma (t)\in G(\mathcal{B}) $ for any $ t\in [0,1] $. In particular 
$ \Gamma(1)\in G(\mathcal{B}) $ thus 
$ \mathcal{B}' _r \subset G(\mathcal{B}) $.
\end{proof}
The proofs of the following equalities are consequence of the uniqueness 
of the solutions of Cauchy problems. Given a path 
$ A\in C(\mathbb{R},\mathcal{B}) $, $ \tau\in\mathbb{R} $ we denote by 
$ A_{\tau} $ the path $ A(\cdot+\tau) = A(t + \tau) $.
\glsadd{labAtau}
\begin{proposition}
\label{ca:ba}
  Let $ A $ and $ B $ be two continuous paths. Then
\begin{align*}
&X_{A + B} = X_A\cdot X_{{X_A}^{-1} B X_A} \\
&X_{A(\cdot +s)} (t) X_A (s) = X_A (t + s) \\
\end{align*}
for any $ t,s\in\R $.
\end{proposition}
\begin{proof}
Let $ X = X_A X_{{X_A}^{-1} B X_A} $. Differentiating
\begin{equation*} 
  \begin{split} 
   X' &= X_A ' \cdot X_{{X_A}^{-1} B X_A} +
   X_A \cdot X_{{X_A}^{-1} B X_A} ' \\
   &= (A + B) X_A\cdot X_{{X_A}^{-1} B X_A} = (A + B) X
  \end{split} 
\end{equation*}
hence $ X = X_{A + B} $. To prove the second equality let 
$ Y = X_{A(\cdot +s)} (t) X_A (s) $. Differentiating we find that 
$ Y'(t) = A(t + s) Y(t), \ Y(0) = X_A (s) $. Since the same holds for 
$ Z(t) = X_A (t + s) $ the second equality is proved. 
\end{proof}
\vskip .2em
\begin{proposition}
\label{b-a}
 Let $ A,B\in C(\R,\al{B}) $
\begin{gather}
\label{int}
X_B (t) = X_A (t) + \int_0 ^t X_A (t) X_A (\tau) ^{-1} (B - A) X_B (\tau) 
d\tau
\end{gather}
\end{proposition}
\begin{proof}
Call $ X $ and $ Y $ respectively the left and right members of (\ref{int}). 
We have $ X(0) = Y(0) = 1 $ at $ t = 0 $. We prove that both solve the Cauchy
problem $ u' = Au + (B - A) X_B $ with starting point $ 1 $. In fact
\begin{equation*}
\begin{split}
X' &= B X_B = A X_B + (B - A) X_B = AX + (B - A) X_B \\
Y' &= A X_A  + A \int_0 ^t X_A (t) X_A (\tau) ^{-1} (B - A) X_B (\tau) +
(B - A) X_B \\ 
&= A Y + (B - A) X_B.
\end{split}
\end{equation*}
\end{proof}
When $ \mathcal{B} $ is the algebra of bounded operators on a Banach space
$ E $, given a path $ A $ in $ \mathcal{L}(E) $ we can always consider the
adjoint $ A^*\in C(\mathbb{R},\mathcal{L}(E^*)) $. The relation
\begin{gather}
\label{dual}
{(X_A^{-1})} ^*  = X_{-A^*}
\end{gather}
holds. In fact the derivative of the left member is 
\begin{gather*}
(-(X_A)^{-1} A X_A X_A ^{-1})^* = - A^* (X_A ^{-1}) ^* = X_{-A^*} '.
\end{gather*}
\section{Exponential estimate of $ X_A $}
In this section we denote by $ C_b (\mathbb{R},\mathcal{B}) $ the space of
bounded functions in $ \mathcal{B} $. This space is endowed with the norm
$ \no{A}_{\infty} = \sup_{t\in\mathbb{R}} \no{A(t)} $ that makes it a Banach
space.
\begin{proposition}
\label{blambda}
If $ A $ is bounded $ X_A (t) $ satisfies the exponential 
estimate
\begin{gather}
\label{blambda:1}
\| X_A (t) X_A (s) ^{-1}\|\leq c e^{\lambda |t - s|}
\end{gather}
for some $ c > 0 $, $ \lambda\in\R $ and any $ t,s\in\R $.
\end{proposition}
\begin{proof}
Let $ r = t - s $. By the Proposition \ref{ca:ba} it is enough to prove that
\[ 
\no{X_{A(\cdot+s)} (r)}\leq c e^{\lambda |r|}
\]
for every $ r\in\mathbb{R} $. To achieve this inequality we apply the 
Gronwall's lemma to the function $ \alpha(r) = \no{X_{A(\cdot+s)}(r)} $.
In fact since
\[
\alpha(r) \leq 1 + \int_0 ^r \no{A_{(\cdot+s)} (\tau)} \alpha(\tau) d\tau
\]
by the Gronwall's lemma (see Lemma \ref{gron})
\begin{gather*}
\alpha(r) \leq 1 + \int_0 ^r e^{\no{A}_{\infty} (r - \tau)} d\tau.
\end{gather*}
Easy computations show that $ c = 2\max\{1, 1 - 1/\no{A}_{\infty}\} $ and
$ \lambda = \no{A}_{\infty} $ fit our request. Repeating the same argument
for $ t < 0 $ we complete the proof.
\end{proof}
\begin{proposition}
\label{op:1}
Let $ A, H\in C_b (\R,\al{B}) $.
If $ \no{X_A (t) X_A (s)^{-1}}\leq c e^{\lambda (t - s)} $ for any 
$ t\geq s\geq 0 $ we have 
$ \no{X_{A + H} (t) X_{A + H} (s) ^{-1}}\leq c e^{\mu (t - s)} $ where 
$ \mu = \lambda + c \no{H}_{\infty} $.
\end{proposition}
\begin{proof}
Applying the first equality of Proposition \ref{ca:ba} to $ A $ and $ \mu $ it 
easy to check that $ X_A $ satisfies the exponential estimate for any 
$ t\geq s\geq 0 $ with constants $ (c,\lambda) $ 
if and only if $ A + \mu $ does the same with $ (c,\lambda - \mu ) $. 
In fact
\begin{equation*}
\begin{split}
X_{A + \mu I} (t) = X_A \cdot X_{X_A \cdot \mu I \cdot X_A ^{-1}} (t) 
= X_A \cdot X_{\mu I} (t) = e^{\mu t} X_A (t).
\end{split}
\end{equation*}
Set $ B = A + \mu I $. Hence we just have to prove that if 
$ (c,\lambda - \mu ) $ works with $ X_B $ then $ (c,0) $ works with 
$ X_{B + H} $. Now fix $ s\geq 0 $. By the second of Proposition \ref{ca:ba}
$ X_B (t) X_B (s) ^{-1} = X_B (s + t - s) X_B (s) ^{-1} = 
X_{B(\cdot + s)} (t - s) $ and the statement reduces to prove that
\begin{equation}\label{redu}
\begin{split}
X_{B_s} (r)\leq c e^{(\lambda - \mu) r} \rba
X_{B_s + H_s} (r)\leq c, \ r > 0,
\end{split}
\end{equation}
where $ B_s = B_{(\cdot + s)} $, $ H = H_{(\cdot +s)} $. To prove 
(\ref{redu}) fix $ t\in\R^{+} $ and consider the following map of 
$ C_b ([0,t],\mathcal{B}) $ into itself
\begin{equation*}
\begin{split}
X\mapsto (fX)(r) = X_{B_s} (r) \left[ 1 + \int_0 ^{r} 
X_{B_s} (\tau) ^{-1} H_s (\tau) Y(\tau) d\tau \right].
\end{split}
\end{equation*}
By (\ref{int}) $ X_{B_s + H_s} $ is a fixed 
point of $ f $. We will prove that $ f $ is a contraction and 
that $ \orl{B} (0,c) $ is invariant for $ f $. Since every nonempty
closed invariant subset for a contraction contains its fixed point this
will conclude the proof. It is enough to prove that the linear application
$ L = f - X_{B_s} $ is bounded and $ \no{L} < 1 $. This will suffice to prove
that $ L $ is a contraction, hence the affine map $ L + X_B $ is also a 
contraction. Let $ X $ in $ C([0,t],\al{B}) $
\begin{equation*}
\begin{split}
\no{LX}_{\infty}\leq\frac{c\no{H}_{\infty}}{ \mu - \lambda }
\Big( 1 - e^{ - (\mu - \lambda) t } \Big) \no{X}_{\infty},
\end{split}
\end{equation*}
hence $ f $ is a contraction. To prove that $ \orl{B} (0,c) $ is invariant
for $ f $ let $ X\in \orl{B} (0,c) $ thus
\begin{equation*}
\begin{split}
\no{(fX)(t)} = &\left\| X_{B_s} (t) 
\left[ 1 + \int_0 ^t X_{B_s} (\tau) ^{-1} H_s (\tau) Y(\tau) d\tau \right ]
\right\| \\
&\leq c e^{ (\lambda - \mu) t } + c^2 \no{H}_{\infty} \int_0 ^t 
\| X_B (t) X_B (\tau) ^{-1} \| d\tau \\
&\leq c e^{(\lambda - \mu) t} 
\left ( 1 - \frac{c \no{H}_{\infty}}{ \mu - \lambda } \right) + 
\frac{c^2 \no{H}_{\infty}}{ \mu - \lambda } = c.
\end{split}
\end{equation*}
Then $ \no{fX}_{\infty} \leq c $ and the proof is complete.
\end{proof}
\section{Asymptotically hyperbolic paths}
For the remainder of this chapter we restrict our attention to the algebra
of bounded operators on a Banach space $ E $. Given a continuous path $ A $ 
in the space of bounded operators, defined on $ \mathbb{R}^+ $ we define
the \textsl{stable space} as
\begin{gather*}
W_A ^s = \left\{\hskip .1em x\in X \hskip .1em\Big|\hskip .1em
\lim_{t\ra +\infty} X_A (t) x = 0\hskip .1em\right\}.
\end{gather*}
Similarly, if $ A $ is a path defined on $ \mathbb{R}^- $ we define the
unstable space 
\begin{gather*}
W_A ^u = \left\{\hskip .1em x\in X\hskip .1em\Big|\hskip .1em 
\lim_{t\ra -\infty} X_A (t) x = 0\hskip .1em\right\}.
\end{gather*}
Using the equalities of Proposition \ref{ca:ba}, for every $ t\geq 0 $ and
$ s\leq 0 $ we have
\begin{gather}
\label{stab}
X_A (t) W_A ^s = W_{A(\cdot + t)} ^s, \ 
X_A (s) W_A ^u = W_{A(\cdot + s)} ^u .
\end{gather}
We denote by $ \mathbb{H}^{+} $ and $ \mathbb{H}^- $ the semi-planes of %
\glsadd{labH+}
\glsadd{labH-}
$ \mathbb{C} $ with positive and negative real part, respectively. Let 
$ A_0 $ be a hyperbolic operator, that is 
$ \sigma (A_0) \cap \img\mathbb{R} = \emptyset $. Thus we have a decomposition
of the spectrum
\begin{gather*}
\sigma(A_0) = \sigma^+ (A_0) \cup\sigma^- (A_0) 
\end{gather*}
where $ \sigma^{\pm} (A_0) = \sigma(A_0)\cap\mathbb{H}^{\pm} $. Let $ P^+ $, 
$ P^- $ be the spectral projectors of the decomposition, $ E^+ $ and $ E^- $ 
their range respectively. It is clear that the stable and unstable spaces of
the constant path $ A_0 $ are $ E^- $ and $ E^+ $. In the following
theorem we prove that if $ A = A_0 + H $ is a small perturbation of $ A_0 $
the stable and unstable spaces of $ A $ are closed and admit a topological
complement.
\begin{proposition}\rm{(cf. \cite{AM03b}, \textsc{Proposition} 1.2).} 
\label{tow}
Let $ A_0 $ be a hyperbolic operator, with
$ \sigma^{-} (A_0) $ and $ \sigma^+ (A_0) $ nonempty, and a pair
$ (c,\lambda) $, $ \lambda > 0 $ such that, for any $ t \geq 0 $
\begin{equation}
\label{hyp}
\begin{split}
\| e^{tA_0} |_{E^{-}} \|\leq c e^{-\lambda t},\ 
\| e^{-tA_0} |_{E^{+}} \|\leq c e^{-\lambda t}.
\end{split}
\end{equation}
Let $ M := \max\{\no{P^+},\no{P^-}\} $. 
There are positive constants $ h,\nu,b $ depending only on $ c $ and 
$ \lambda $ such that if 
\begin{gather*}
 \no{H}_{\infty}\leq\frac{\lambda}{Mc(1 + \sqrt{c})}
\end{gather*}
the following facts hold:
\begin{enumerate}
  \item for every $ t\geq 0 $, $ X_A (t) W_A ^s $ is the graph of a
bounded operator $ S(t)\in \mathcal{L}(E^- , E^+) $,
  \item $ \| S(t) \| \leq c^2 \displaystyle
\int_t ^\infty e^{-\nu (\tau - t) } \|H (\tau)\| d\tau $,
  \item the function $ S $ has much differentiability as $ X_A $,
  \item for every $ u_0\in W_A ^s $ and every $ t\geq s\geq 0 $ 
        there holds
    \begin{equation*}
      \begin{split}
        | X_A (t) u_0 |\leq be^{ -\nu (t - s) } | X_A (s) u_0 |.
      \end{split}
    \end{equation*}
\end{enumerate}
\end{proposition}
\begin{proof}
First we check what kind of differential equation satisfies 
$ u = X_A \cdot u_0 $, for any $ u_0 \in E^{-} \oplus E^{+} $, in terms of
the projectors $ P^{\pm} $. Let $ u = x + y $. 
Differentiating both sides we find that
\begin{gather}
\label{diff}
\left\{
\begin{array}{l}
x' = A_{-} x + A_{\mp} y  \\
y' = A_{\pm} x + A_{+} y
\end{array}
\right.
\end{gather}
where $ A_{\pm} = P^{+} A P^{-} $, $ A_{-} = P^{-} A P^{-} $ and so 
on. For every $ r\geq t\geq s $ the system above can be rewritten as 
\begin{equation}
\label{diff2}
  \begin{split}
    x(t) &= X_{A_{-}} (t) X_{A_{-}} (s) ^{-1} x(s) +
    \int_s ^t X_{A_{-}} (t) X_{A_{-}} (\tau) ^{-1} A_{\mp} (\tau) y(\tau) 
    d\tau\\
    y(t) &= X_{A_{+}} (t) X_{A_{+}} (r)^{-1} y(r) 
    - \int_t ^r X_{A_{+}} (t) X_{A_{+}} (\tau) ^{-1} A_{\pm} (\tau) x(\tau) 
    d\tau .
  \end{split}
\end{equation}
By hypothesis $ A_{0,-} $ fulfills the exponential 
estimate (\ref{blambda:1}) with constants $ (c, -\lambda) $. 
Thus $ A_{-} $, by Proposition \ref{op:1}, also does it with constants $ c $ 
and $ -\mu_{-} = -\lambda + c\no{H_{-}} $. Similarly, by (\ref{hyp}) 
$ - A^* _+ $ fulfills the estimate (\ref{blambda:1}) with constants $ c $ and 
$ -\mu_+ = -\lambda + c\no{H^* _+} = -\lambda + c\no{H_+} $. By the
equality (\ref{dual}) we have
\begin{equation}
\begin{split}
&\| X_{A_+} (t) X_{A_+} (r) ^{-1} \| = 
\| (X_{A_+} (t) X_{A_+} (r) ^{-1})^* \| \\
=& \| {X_{A_+} (r)^{-1}}^* X_{A_+} (t)^* \| = 
\| X_{-A_+ ^*}(r) X_{-A_+ ^*} (t) ^{-1} \|
\end{split}
\end{equation}
for $ r\geq t\geq 0 $. The first equation of (\ref{diff2}) gives
inequalities
\begin{equation}
\begin{split}
\label{1}
    \left|\int_s ^t X_{A_{-}} (t) X_{A_{-}} (\tau) ^{-1} 
    H_{\mp} (\tau) y(\tau) d\tau \right| &\leq c\int_s ^t 
    e^{ -\mu_{-} (t - s)} \no{H_{\mp} (\tau)} | y(\tau) | d\tau \\
    &\leq \frac{c\no{H_{\mp}}}{\mu_{-}}
 \left(1 - e^{ -\mu_{-} (t - s) }\right) \no{y}_{\infty, [s,t]}\\
\end{split}
\end{equation}
and the second gives
\begin{equation}
\begin{split}
  \label{2}
    \left|\int_t ^r X_{A_{+}} (t) X_{A_{+}} (\tau) ^{-1} 
    H_{\pm} (\tau) x(\tau)d\tau\right| &
\leq c\int_t ^r e^{-\mu_{+} (\tau - t)} \no{H_{\pm} (\tau)} dt
   \hskip .1em \no{x}_{\infty,[t,r]}\\
 &\leq\frac{c\no{H_{\pm}}}{\mu_{+}} 
\left(1 - e^{-\mu_{+} (r - t)}\right) \no{x}_{\infty,[t,r]}.
\end{split}
\end{equation}
Since $ \mu_+ $ and $ \mu_- $ are positive,
in the second of (\ref{diff2}) we can take the limit as 
$ r\rightarrow +\infty $. Set $ s = 0 $ in the first of (\ref{diff2}).
Therefore the equations (\ref{1}) and (\ref{2}) permit to define
a continuous map on the Banach space $ C_b (\mathbb{R}^+,E^- \oplus E^+) $ 
\begin{gather}
  \label{contr}
\vfi_{A,x_0}\cdot 
  \begin{pmatrix}
    x \\
    y
  \end{pmatrix}
  = L_A 
\begin{pmatrix}
    x \\
    y
  \end{pmatrix}
+ 
\begin{pmatrix}
    X_{A_{-}} (\cdot) x_0 \\
    0
  \end{pmatrix}
\end{gather}
where
\[
\label{linear_part}
L_A
\begin{pmatrix}
x \\
y
\end{pmatrix}
(t) =
\begin{pmatrix}
\displaystyle
\int_0 ^t X_{A_{-}} (t) X_{A_{-}} (\tau) ^{-1} 
A_{\mp} (\tau) y(\tau) d\tau\\
- \displaystyle\int_t ^\infty X_{A_{+}} (t) X_{A_{+}} (\tau) ^{-1} A_{\pm} 
(\tau) x(\tau) d\tau
\end{pmatrix}
\]
\glsadd{labphAx0}
\glsadd{labLA}
By (\ref{1}) and (\ref{2}), the operator $ L_A $ is bounded. A bounded
solution $ u $ of (\ref{diff}), with $ P^- u(0) = x_0 $ is a fixed point of 
$ \vfi_{A,x_0} $. The estimate of $ \no{H}_{\infty} $ in the hypothesis gives
\begin{gather}
\label{choice2}
(2c^3)^{1/2} \no{H_{\mp}} < \mu_{-}, \ (2c^3)^{1/2}\no{H_{\pm}} < \mu_+
\end{gather}
hence $ \varphi_{A,x_0} $ is a contraction. Clearly if $ u_0\in W_A ^s $ 
the curve $ X_A (t) u_0 $ is a fixed point of $ \vfi_{A,x_0} $. Using 
(\ref{1}) and (\ref{2}) we prove that if $ u $ is fixed point then 
$ u(0)\in W_A ^s $, hence $ u $ is not just bounded, but infinitesimal also. 
If $ u(0) = 0 $ it is clear. Suppose $ u(0)\neq 0 $. For any $ t\geq s $
\begin{equation}
\label{x}
  \begin{split}
    | x(t) | &\leq c e^{-\mu_- (t - s)} | x(s) | + 
\frac{c\no{H_{\mp}}}{\mu_{-}}
 \Big(1 - e^{ -\mu_- (t - s)}\Big) \| y \|_{\infty, [s,t)} \leq \\
        &\leq \max \{ c|x(s)|, \frac{c\no{H_{\mp}}}{\mu_{-}} 
\no{y}_{\infty,[s,\infty)} \},
  \end{split}
\end{equation}
the supremum on the real axis is allowed since we know that both $ x $ and 
$ y $ are bounded. From (\ref{2})
\begin{gather}
\label{3}
| y(s) |\leq \frac{c\no{H_{\pm}}}{\mu_{+}} \no{x}_{\infty, [s	,\infty)} 
\end{gather}
and, taking the sup on $ [s,\infty) $
\begin{gather}
\no{y}_{\infty, [s,\infty)} \leq \frac{c\no{H_{\pm}}}{\mu_{+}}
 \no{x}_{\infty, [s,\infty)}
\end{gather}
and we get
\begin{equation}
\label{andweget}
  \begin{split}
    \| x \|_{\infty, [s,\infty)} \leq 
    \max \{ c |x(s)| ,\hskip .1em \frac{c^2 \no{H_{\pm}}\no{H_{\mp}}}
{\mu_{-}\mu_{+}}
\| x \|_{\infty, [s,\infty)} \};
  \end{split}
\end{equation}
the estimate of $ \no{H} $ also implies that 
$ c^2 \no{H_{\pm}}\no{H_{\mp}} < \mu_- \mu_+ $, therefore (\ref{andweget})
allows to write
\begin{gather}
\label{crucial}
\no{x}_{\infty, [s,\infty) } \leq c | x (s) |,
\end{gather}
and, by (\ref{3}) we get the final estimate
\begin{gather}
\label{fin}
| y(s) | \leq \frac{c^2\no{H_{\pm}}}{\mu_{+}} | x(s) |. 
\end{gather}
It is easy to check that $ x $ does not vanish at any point of 
$ \mathbb{R}^+ $ for, if such $ t\in\R^+ $ exists (\ref{fin}) implies 
$ y(t) = 0 $, thus $ 0 = x(t) + y(t) = u(t)=  X_A (t) u_0 $. 
Since $ X_A (t) $ is invertible we had $ u_0 = 0 $ in contradiction with the
hypothesis. If $ E $ is a Hilbert space it is easy to build a continuous
path $ U (t) $ of operators in $ \mathcal{L}(E^-,E^+) $ that maps $ x(t) $
to $ y(t) $ and $ \no{U(t)} = |y(t)|/|x(t)| $. Just define
\begin{gather*}
U(t) z = \frac{(x(t),z)}{|x(t)|^2} y(t)
\end{gather*}
where $ (\cdot,\cdot) $ denotes the scalar product of the Hilbert space.
For Banach spaces we need some results of continuous selection such as
\textsc{Theorem} 4 of \cite{BK73}. By Corollary \ref{cont_sel_of_op} of
Appendix D there exists a path $ U_{\var} $ continuous and bounded in
$ \mathcal{L}(E^-,E^+)) $ such that 
\begin{gather*}
U_{\var} (t) x(t) = y(t), \ \no{U_{\var}(t)}\leq (1 + \var)
\frac{c^2\no{H_{\pm}}}{\mu_{+}} + \var
\end{gather*}
for every $ \var > 0 $. Then we can write the first of (\ref{diff}) as 
\begin{equation*}
  \begin{split}
    x' = [ A_{-} (t) + A_{\mp} (t) U_{\var} (t) ] x
  \end{split}
\end{equation*}
Since $ A_{\mp} (t) U_{\var} (t) $ is a bounded operator in $ \mathcal{L}(E^{-}) $ 
we can apply the Proposition \ref{op:1}: in fact $ A_{-} $ satisfies an 
exponential estimate with constants $ (c, - \mu_{-}) $, then the path 
$ A_{-} (t) + A_{\mp} (t) U_{\var} (t) $ does it with constants 
$ (c, -\nu_{\var}) $ 
where
\begin{equation*}
-\nu_{\var} = -\mu_{-} + c \no{H_{\mp} U_{\var}}\leq - \mu_{-} + c
\no{H_{\mp}}\cdot(1 + \var)\frac{c^2 \no{H_{\pm}}}{\mu_+} + c\var\no{H_{\mp}}.
\end{equation*}
Let $ \nu = \nu_0 $. We have $ -\mu_+ \nu = -\mu_{-}\mu_{+} + 
c^3\no{H_{\pm}}\no{H_{\mp}} $. By (\ref{choice2}) $ -\mu_+ \nu < 0 $, hence 
$ -\nu < 0 $. Then, if we choose $ \var $ small enough $ -\nu_{\var} < 0 $ 
and
\begin{gather*}
|x(t)|\leq e^{-\nu_{\var} (t - s)} |x(s)|, \ \ t\geq s\geq 0.
\end{gather*}
Taking the limit as $ \var\rightarrow 0 $ we obtain
\begin{equation}
 \begin{split}
   \label{van}
   | x(t) | &\leq c e^{-\nu (t - s)} | x(s) | \\ 
   | y(t) | &\leq\frac{c^3 \no{H_{\pm}}}{\mu_+} e^{-\nu (t - s)} | x(s) | 
 \end{split} 
 \end{equation}
and $ x $ and $ y $ vanish at infinity. Thus the fixed point $ u $ of 
$ \varphi_{A,x_0} $ can be characterized as a curve that solves (\ref{diff}) 
such that
\begin{gather}
\label{fixed_point_phi}
u(+\infty) = 0, \ \ P^- u(0) = x_0.
\end{gather}
An application $ S $ in $ \mathcal{L}(E^-,E^+) $ whose graph is $ W_A ^s $ is 
defined as follows: given $ x_0 $ in $ E^- $ there exists 
a unique fixed point of $ \varphi_{A,x_0} $, call it $ u $. Thus 
$ u(0)\in W_A ^s $. We define $ Sx_0 = P^+ u(0) $ and we have
\begin{gather*}
u(0) = P^- u(0) + P^+ u(0) = x_0 + Sx_0
\end{gather*}
hence $ \graph(S)\subseteq W_A ^s $. Conversely, given $ u_0\in W_A ^s $ the
curve $ v(t) = X_A (t)\cdot u_0 $, by the characterization in 
(\ref{fixed_point_phi}), is the fixed 
point of $ \vfi_{A,P^- u_0} $, hence
$ P^+ u_0 = S P^- u_0 $. Then $ \graph(S) = W_A ^s $. We can write explicitly 
$ S $ 
\begin{gather}
\label{graph_stab}
S = P^{+} \ci ev_0 \ci (I - L_A) ^{-1} \cdot 
  \begin{pmatrix}
    X_{A_{-}} (\cdot) x_0 \\
    0
  \end{pmatrix},
\end{gather}
where $ ev_0 $ is defined on $ C_b $ as the evaluation at $ t = 0 $.
\glsadd{labev0}%
Then $ S $ is bounded and $ W_A ^s $ is closed and $ E = E^+ \oplus W_A ^s $.
Since $ A_t = A_0 + H(\cdot + t) $ the same constants work to show 
that $ W_{A_t} ^s = X_A (t) W_A ^s $ is graph of an unique bounded 
operator, say $ S(t) $ and i) is proved. By direct computation 
\begin{gather}
\label{asmu}
S(t) = P^+ X_A (t) (I_{E_{-}} + S) \cdot [ P^- X_A (t)(I_{E_{-}} + S) ]^{-1}.
\end{gather}
Hence $ S\in C(\mathbb{R},\mathcal{L}(E^-,E^+)) $ inherits the regularity of 
$ X_A $ and ii) follows. 
\vskip .2em
Taking the limit as $ r\ra +\infty $ and $ t = 0 $ in (\ref{2})
\begin{equation*} 
 \begin{split}
 | S x_0 | = |y_0| &\leq 
c\left( \int_0 ^{\infty} e^{-\mu_+ \tau} \| H_{\pm} (\tau) \| 
d \tau \right) 
\| x \|_{\infty} \\
&\leq c^2 \left( \int_0 ^{\infty}
e^{-\nu \tau } \| H (\tau) \| d \tau \right) |x_0|
 \end{split} 
 \end{equation*}
since $ \nu < \mu_+ $. For the general case consider the shifted path 
$ A (\cdot + t) $. Then 
\begin{equation*}
\begin{split}
| S(t)x_0 |&\leq c^2 \left(\int_0 ^{\infty} e^{-\mu_+\tau'} 
\| {H_{(\cdot+t)}}_{\pm} (\tau')\| d\tau'\right) |x_0| \\
&= c^2 \left(\int_t ^{\infty} e^{-\mu(\tau - t)} \| H_{\pm} (\tau) \|
d\tau\right) |x_0|
\end{split}
\end{equation*}
where $ \tau = t + \tau' $. This proves iii). Finally let $ u_0\in W_A ^s $.
By (\ref{van}) and (\ref{choice2}) we can write
\begin{equation*}
 \begin{split}
   | X_A (t) u_0 | &= | x(t) + y(t) | \leq | x(t) | + | y(t) | 
   \leq c e^{-\nu (t - s)} 
   \left(1 + \frac{c^2 \no{H_{\pm}}}{\mu_+}\right)| x(s) | \\
   & \leq (c + c^2)\no{P^-} e^{-\nu (t - s)} |u(s)| \leq
   b e^{-\nu (t - s)} | X_A (s) u_0 |.
 \end{split} 
 \end{equation*}
where $ b = c(1 + c) \no{P^{-}}\no{P^+} $. The proof is complete.
\end{proof}
\begin{proposition}\rm{(cf. \cite{AM03b}, \textsc{Proposition} 1.2)}.
\label{from}
  With the same hypotheses of the preceding statement we have
  \begin{enumerate}
  \item for every $ t\geq 0 $, $ X_A (t) E^+ $ is the graph of an operator 
    $ T(t)\in \mathcal{L}(E^{+}, E^{-}) $,
  \item $ \| T (t) \| \leq c^2 \displaystyle
    \int_0 ^t e^{ -\nu (t - \tau) }
    \| H (\tau) \| \hskip .1em  d\tau $,
  \item $ T $ is as much differentiable as $ X_A $,
    \item for every $ y_0 \in E^{+} $, $ t\geq s\geq 0 $ the inequality
      \begin{equation*} 
        \begin{split} 
          | X_A (t) y_0 |\geq b^{-1} e^{ \nu (t - s) } | X_A (s) y_0 |
        \end{split} 
      \end{equation*}
        holds.
    \end{enumerate}
\end{proposition}
\begin{proof}
Let $ \orl{t}\in\R^+ $ and $ \orl{y} \in E^{+} $. In (\ref{diff2}) let 
$ r = \orl{t} $ and $ s = 0 $. Then we have a continuous map, 
on $ C([0,\orl{t}],E^{-}\oplus E^+) $ into itself
\begin{gather}
\label{contr2}
\psi_{A,\orl{y}}\cdot 
\begin{pmatrix}
x \\
y
\end{pmatrix}
= R_A\cdot 
\begin{pmatrix}
x \\
y
\end{pmatrix}
+ 
\begin{pmatrix}
0 \\
X_{A^{+}} (\cdot) X_{A^{+}} (\orl{t}) ^{-1} \orl{y}.
\end{pmatrix}
\end{gather}
where $ R_A $ is a bounded operator defined as
\begin{gather*}
R_A\cdot
\begin{pmatrix}
x \\
y
\end{pmatrix} = 
\begin{pmatrix}
\displaystyle
\int_0 ^t X_{A_{-}} (t) X_{A_{-}} (\tau)^{-1} A_{\mp} (\tau) y(\tau) d\tau\\
\displaystyle   
-\int_t ^{\orl{t}} X_{A_{+}} (t) X_{A_{+}} (\tau)^{-1} A_{\pm}(\tau) x(\tau)
d\tau
\end{pmatrix}
\end{gather*}
\glsadd{labpsAy}
\glsadd{labRA}
The map is continuous because $ R_A $ is bounded. If $ \no{H}_{\infty} $ is
estimated by the same constant of the preceding proposition $ \no{R_A} < 1 $,
hence $ \psi_{A,\overline{y}} $ is a contraction. The fixed point $ v $ is
a solution of (\ref{diff}) characterized by the property 
\begin{gather}
\label{fixed_point_psi}
P^- v (0) = 0, \ \ P^+ v(\overline{t}) = \overline{y}
\end{gather}
Let $ \overline{y}\in E^+ $ and let $ u $ be the fixed  point of 
(\ref{contr2}). We define $ T(\orl{t})\cdot\orl{y} = P^- u(\orl{t}) $.
By (\ref{contr2}) $ u(0)\in E^{+} $ and $ P^+ u(\orl{t}) = \orl{y} $, hence
\begin{gather*}
 \orl{y} + T(\orl{t})\orl{y} = P^+ u(\orl{t}) + P^- u(\orl{t}) = 
u(\orl{t}) = X_A (\orl{t}) u(0) 
\end{gather*}
thus $ \graph(T(\orl{t}))\subset X_A (\orl{t}) E^{+} $. Conversely,
let $ z\in X_A (\orl{t}) E^+ $ and $ y \in E^+ $ be such that
$ z = X_A (\orl{t}) y $. The curve $ u = X_A (\cdot) y $ has the property 
(\ref{fixed_point_psi}), thus coincides with the fixed point of 
$ \psi_{A,P^+ z} $. Hence
\begin{gather*}
z = P^+ z + P^- z = P^+ z + P^- u(\orl{t}) = P^+ z + P^- v(\orl{t}) = 
P^+ z + T(\orl{t}) P^+ z.
\end{gather*}
Hence $ X_A (\orl{t}) E^+ = \graph(T(\orl{t})) $. The map can also be
written as
\begin{gather*}
T(\orl{t}) y = P^- \circ ev_0 (I - R_A)^{-1} \cdot 
\begin{pmatrix}
0 \\
X_{A^{+}} (\cdot) X_{A^{+}} (\orl{t}) ^{-1} \orl{y}.
\end{pmatrix}
\end{gather*}
and i) is proved. For every $ t\geq 0 $ 
\begin{gather}
T(t) = P_{-} ({P_{+}|} _{X_A (t) E^{+}}) ^{-1} = 
P_{-} X_A (t) [ P_{+} X_A (t) ]^{-1}
\end{gather}
and iii) follows. Now let $ 0\leq t\leq r\leq\orl{t} $. If $ (x, y) $ is the 
fixed point of $ \psi_A $ we find
\begin{gather}
\label{x21}
|x(t)|\leq \frac{c\no{H_{\mp}}}{\mu_{-}} \|y\|_{\infty,[0,t]}
\end{gather}
still from (\ref{1}) and (\ref{2}) we can write
\begin{equation}
\label{y21}
  \begin{split}
    | y(t) | &\leq c e^{-\mu_+ (r - t)} | y(r) | + 
    \frac{c\no{H_{\pm}}}{\mu_+}
 \Big(1 - e^{ -\mu_+ (r - t)}\Big) \no{x}_{\infty, [0,r)} \leq \\
        &\leq \max \{ c|y(r)|, \frac{c\no{H_{\pm}}}{\mu_+} 
   \no{x}_{\infty,[0,r]} \};
  \end{split}
\end{equation}
from (\ref{x21}) we write
\begin{gather}
\label{x22}
\no{x}_{\infty,[0,t]}\leq\frac{c\no{H_{\mp}}}{\mu_{-}} \no{y}_{\infty,[0,t]}
\end{gather}
for any $ 0\leq t\leq r $. By (\ref{y21}) and (\ref{x22})
\begin{gather}
\label{y22}
 \no{y}_{\infty,[0,r]}\leq \max \{ c |y(r)|, 
\frac{c^2 \no{H_{\pm}}\no{H_{\mp}}}{\mu_{-}\mu_{+}} 
\no{y}_{\infty,[0,r]} \}.
\end{gather}
Since $ c^2 \no{H_{\pm}}\no{H_{\mp}} < \mu_{-} \mu_+ $ we have
\begin{gather}
\label{y24}
\no{y}_{\infty,[0,r]} \leq c |y(r)|.
\end{gather}
Setting $ t = r $ in (\ref{x21}) from (\ref{y24}) it follows that
\begin{gather}
\label{last}
|x(r)|\leq\frac{c^2\no{H_{\mp}}}{\mu_{-}} |y(r)|
\end{gather}
As we have done for the preceding Proposition for every $ \var > 0 $ the 
Corollary \ref{cont_sel_of_op} provides us with 
$ V_{\var}\in C([0,\orl{t}],\bx{E^{+}}{E^{-}}) $ such that
\begin{gather*}
V_{\var} (r) y(r) = x(r), \ \no{V_{\var}}
\leq (1 + \var)\frac{c^2\no{H_{\mp}}}{\mu_{-}} + \var
\end{gather*}
hence $ y' = (A_{+} + H_{\pm} V_{\var}) y $. Applying the Proposition 
\ref{op:1} to $ -A_+ ^* $ for every $ \var > 0 $ and $ r\geq t\geq 0 $ there 
holds $ | y(r) |\geq c^{-1} e^{\nu_{\var} (r - t)} | y(t) | $. Taking the 
limit as $ \var\rightarrow 0 $ 
\begin{gather}
\label{y25}
| y(r) |\geq c^{-1} e^{\nu (r - t)} | y(t) |.
\end{gather}
for every $ r\geq t \geq 0 $. By (\ref{last})
\begin{equation}
\label{y26}
\begin{split}
| y(s) | &= \frac{1}{1 + c} \Big(c |y(s)| + |y(s)|\Big)
\geq\frac{1}{1 + c}\Big(|x(s)| + |y(s)|\Big) \geq\frac{1}{1 + c} |u(s)|.
\end{split}
\end{equation}
Given $ u_0\in E^{+} $, using (\ref{y26}) and the fact that the norm of
a projector is at least $ 1 $ we can write
\begin{equation}
\begin{split}
| X_A (r) u_0 | &\geq |y(r)| \no{P^+} ^{-1}
\geq (c\no{P^+}) ^{-1} e^{\nu (r - t)} |y(s)| \\
&\geq \frac{ e^{\nu (r - t)} |u(t)|}{c(1 + c)\no{P^+}} \geq
b^{-1} e^{\nu (r - t)} | X_A (t) u_0 |.
\end{split}
\end{equation}
and (iv) follows. Finally
\begin{equation*}
\begin{split}
|T(\orl{t}) \orl{y}| = |x(\orl{t})|&\leq c \left(\int_0 ^{\orl{t}} 
e^{-\mu_+ (\orl{t} - \tau)}\| H_{\mp}\|\right) \| y \|_{\infty,[0,\orl{t}]}\\
&\leq c^2 \left(\int_0 ^{\orl{t}} 
e^{-\nu(\orl{t} - \tau)}\| H_{\mp}\|\right) | \orl{y} |;
\end{split}
\end{equation*}
the last estimate follows from (\ref{y24}) with $ r = \orl{t} $ and (ii) 
is proved.
\end{proof}
\begin{remark}
\rm{The statements and the proofs of the two theorems regard only the stable 
space. To obtain the same conclusions for the unstable space defined on the 
negative real line just set $ \check{A}(t) = A(-t) $. %
Using argument of 
uniqueness of Cauchy problems we obtain
\begin{gather*}
X_A (-t) = X_{-\check{A}}(t), \ W_A ^u = W_{-\check{A}} ^s, \ 
-\check{A}(+\infty) = - A(-\infty).
\end{gather*}
Thus we can apply Propositions \ref{tow} and \ref{from} to $ -\check{A} $ 
on the positive real line.}
\end{remark}
\section{Properties of $ W_A ^s $ and $ W_A ^u $}
In the preceding section it has been proved that $ W_A ^s $ 
(as $ W_A ^u $ ) is a splitting space if $ A $ is close, in the uniform
topology, to a constant hyperbolic path $ A_0 $. We prove that
it is true for any asymptotically hyperbolic path. Conversely we provide, 
for any pair $ (X,Y) $ in $ G_s (E) $, a path $ A $ such that
$ (W_A ^s, W_A ^u) = (X,Y) $.
\begin{theorem}\rm{(cf. \cite{AM03b}, \textsc{Theorem} 2.1).}
\label{rec}
  Let A be an asymptotically hyperbolic path of operators defined on 
  $ \al{\R} ^+ $. Let $ A_0 = A(+\infty) $, $ E^+ \oplus E^- $ the 
  spectral decomposition. Then $ W_A ^s $ is a splits
  \begin{enumerate}
  \item $ W_A ^s $ is the only closed subspace $ W $ such that 
    $ X_A (t) W \ra E^{-} $,
  \item $ \| \res{X_A (t)}{W_A ^s} \| \leq 
    c e^{-\lambda (t - s)} \| \res{X_A (s)}{W_A ^s} \| $ for suitable 
    $ c, \lambda > 0 $ and every $ t\geq s\geq 0 $,
  \item for every $ V\in G_s (E) $ such that $ V\oplus W_A ^s = E $
  $ \rho(X_A (t) V, E^{+})\ra 0 $,
  \item $ \displaystyle\inf_{\substack{ v\in V \\ |v| = 1 }} | X_A (t) v | $ 
        grows at exponential rate,
  \item $ W_{-A ^*} ^s = (W_A ^s)^{\bot} $. 
\end{enumerate}
\end{theorem}
\begin{proof}
Let $ A(+\infty) = A_0 $. Since $ A_0 $ is a hyperbolic operator there 
exist $ c $ and $ \lambda $ such that the condition (\ref{hyp}) holds.
Let $ H = A_0 - A $. If $ \tau $ is large enough 
$ \no{H_{(\cdot+\tau)}} $ is smaller than the constant of 
Proposition \ref{tow} then 
\begin{gather*}
W_{A(\cdot+\tau)} ^s = X_A (\tau) W_A ^s
\end{gather*}
is a topological complement of $ E^+ $ and, since $ X_A (\tau) $ is invertible,
$ W_A ^s $ is closed too and
\begin{gather*}
X_A (\tau) W_A ^s \oplus E^{+} = 
E = W_A ^s \oplus X_A (\tau)^{-1} E^{+}.
\end{gather*}
Now for $ t \geq\tau $ the Proposition \ref{tow} says that 
$ X_{A(\cdot+\tau)} (t) W_{A(\cdot+\tau)} ^s $ is the graph of a
bounded linear map $ S(t)\colon E^- \rightarrow E^+ $ and
\begin{equation*} 
 \begin{split} 
   \| S (t) \| \leq 
   c^2 \int_t ^\infty e^{-\nu (t - \tau') } 
   \| H(\tau + \tau') \| d\tau' .
 \end{split} 
 \end{equation*}
This implies that $ S(t) $ converges to the null operator as 
$ t\rightarrow +\infty $. By Proposition \ref{continuity_of_images}, 
$ \graph(S(t)) $ converges to $ \graph(0) = E^- $, hence 
$ X_A (t + \tau) W_A ^s = 
X_{A(\cdot+\tau)} (t) W_{A(\cdot+\tau)} ^s \ra E^{-} $.
\vskip .2em
The ii) follows from iv) of Proposition \ref{tow} taking the 
supremum over the unit sphere of $ W_A ^s $ on both sides of the inequality.
\vskip .2em
Let $ V $ be a closed subspace of $ E $. Up to a time shift we can suppose
that $ V $ is graph of a bounded operator 
$ L\in\mathcal{L}(E^{+},W_{A} ^s ) $. First we prove that 
$ \rho (X_A (t)E^{+}, X_A (t)V) $ converges to $ 0 $. 
Let $ v\in X_A (t) V $ and $ y\in E^+ $ be such that 
$ v = X_A (t) \cdot (y + L y) $. Set $ u = X_A (t) y $. Then
\begin{equation*} 
 \begin{split} 
   | v - u | = | X_A (t) Ly | &\leq b e^{-\nu t} 
   \no{L} | y | 
   \leq b^2 e^{-2\nu t} \no{L} | X_A (t) y | \\
   &= b^2 e^{-2\nu t} \no{L} | u | 
   \leq b^2 e^{-2\nu t} \no{L} ( | v | + | v - u | )
 \end{split} 
 \end{equation*}
since $ \alpha (t) := b^2 e^{-2\nu t} \| L \| $ is an infinitesimal 
sequence, for $ t\geq \orl{t} $ we have $ \alpha (t) < 1 $ and the above 
inequality becomes
\begin{equation*} 
 \begin{split} 
   | v - u | \leq \alpha(t) ( | v | + | v - u |) \rba
   | v - u | \leq \frac{\alpha(t)}{1 - \alpha(t)} | v |
 \end{split} 
 \end{equation*}
and we conclude that $ \rho (X_A (t) Y,X_A (t) E^+) \ra 0 $ as 
$ t \ra +\infty $. On other hand
\begin{equation*} 
 \begin{split} 
   | u - v | = | X_A (t) Ly | &\leq b e^{-\nu t} \|L \| | y | 
   \leq b^2 e^{-2\nu t} \| L \| | X_A (t) y | \\
   &= b^2 e^{-2\nu t} \| L \| | u | = \alpha(t) | u |
 \end{split} 
\end{equation*} 
and $ \rho (X_A (t) E^+, X_A (t) V) \leq \alpha (t) $. 
The proof is complete using the fact that 
$ \rho (X_A (t) E^+, E^{+}) \ra 0 $ which follows from i) and ii) of 
Theorem \ref{from}.
\vskip .2em
To prove the converse of i) let $ W\subseteq E $ be a closed subspace 
such that $ X_A (t) W \ra E^{-} $. By iii) for every topological complement 
of $ W_A ^s $, say $ V $, we have $ V\cap W = \{0\} $, hence 
$ W\subset W_A ^s $. There exists $ t_0 > 0 $ such that, 
$ \rho (X_A (t_0) W, X_A (t_0) W_A ^s ) < 1 $ and, by Proposition
\ref{trk:2}, $ X_A (t_0) W = X_A (t_0) W_A ^s $ hence
$ W = W_A ^s $ and i) is proved.
\vskip .2em
In order to prove the iv) we can suppose, up to a time shift, that
$ V\oplus W_A ^s = E = W_A ^s \oplus E^+ $. Again $ V = \graph(L) $, 
$ L\in\mathcal{L}(E^+,W_A ^s) $. Then
\begin{equation*} 
 \begin{split} 
| X_A (t) v | &= | X_A (t) y + X_A (t) Ly | 
\geq | X_A (t) y | - | X_A (t) L y | \\
&\geq b^{-1} e^{\nu t} | y | - b e^{ -\nu t } \| L \| | y | 
= ( b^{-1} e^{\nu t} - b e^{ -\nu t } \| L \| ) | y | \\
&\geq 1/(1 + \| L \|)\hskip 1pt ( b^{-1} e^{\nu t} - b e^{ -\nu t } 
\| L \| ) | v |
 \end{split} 
 \end{equation*}
and iv) follows by taking the infimum over $ S(V) $.
By (\ref{dual}) we have the chain of equalities
\begin{gather}
\label{w1}
X_{-A^*} (t) (W_A ^s)^{\bot} = (X_A (t) ^{-1})^* (W_A ^s)^{\bot} = 
(X_A (t) W_A ^s )^{\bot}.
\end{gather}
Since $ X_A (t) W_A ^s $ converges to $ E^{-} $ and $ E^{-} $ splits 
the Proposition \ref{orthogonal_map} allows us to take the limit in 
(\ref{w1}) which is $ {(E^{-})}^{\bot} $. Since 
$ (E^-) ^{\bot} = E^- (- A^*) $, by i) 
\begin{gather*}
X_{-A^*} (t) (W_A ^s )^{\bot}\rightarrow E^- (- A^*)
\end{gather*}
implies $ (W_A ^s )^{\bot} = W_{-A ^*} ^s $.
\end{proof} 
Analogous statements hold for the unstable space $ W_A ^u $ by
considering the path $ -\check{A} $. 
\begin{lemma}
\label{invariance}
  Let $ A $ be an asymptotically hyperbolic path of on   $ \R ^+ $. 
Then $ X_A (t) W_A ^s = E^{-} $ for every
  $ t\geq 0 $ if and only if $ A(t) E^{-} \subseteq E^{-} $
\end{lemma}
\begin{proof}
  For any $ W\subseteq E $ such that $ X_A (t) W = E^- (A(+\infty)) $ we can 
  set $ t = 0 $ to get $ W = E^- (A(+\infty)) $, hence
  \begin{gather}
    \label{xav}
    X_A (t) E^{-} = E^{-}
  \end{gather}
  for any $ t\geq 0 $. Now, fix $ \orl{t}\in\R^{+} $ and let $ x\in E^{-} $, 
  $ \orl{x} = X_A (\orl{t}) ^{-1} x $. By the (\ref{xav}) the curve
  $ u(t) = X_A (t) \orl{x} $ is $ C^1 $ and takes values in $ E^{-} $, 
  therefore $ u'(t)\in E^- $ for any $ t\in\R^{+} $. Hence
  \begin{gather*}
    E^-\ni u'(\orl{t}) = A(\orl{t}) X_A (\orl{t}) \orl{x} = 
    A(\orl{t}) X_A (\orl{t}) X_A (\orl{t}) ^{-1} x = A(\orl{t}) x.
  \end{gather*}
Conversely, assume that the second condition is true for any $ t\in\R^{+} $.
First we prove that $ X_A (t) E^{-} \subseteq E^{-} $. Let $ x\in E^{-} $
and let $ u(t) = X_A (t) x $. In the second of (\ref{diff}) 
we have $ A_{\pm} x = 0 $ by hypothesis, thus $ y' = A_{+} y $. Hence
\begin{gather*}
P^+ u(t) = X_{A_{+}} (t) P^+ u(0);
\end{gather*}
since $ P^+ u(0) = 0 $ we have $ P^+ u = 0 $ and from the first of 
(\ref{diff}) we obtain $ u(t)\in E^- $. Now, $ X_A $ sets a continuous path of
semi-Fredholm operators on $ E^- $. By Proposition \ref{index-is-constant} 
these operators have the same index for any $ t\in\R^{+} $. Since 
$ X_A (0) = Id $ the index of these operators is zero. Since every
$ X_A (t) $ is restriction of an invertible operator they are injective,
thus surjective, that is $ X_A (t) E^{-} = E^{-} $. In particular
$ X_A (t) E^{-} $ converges to $ E^{-} $. By i) of Theorem \ref{rec}
$ E^{-} = W_A ^s $.
\end{proof}
\begin{proposition}
\label{the_patching_argument}
Given a pair of splitting subspaces $ (X,Y) $ in $ E $ there exists a path
$ A $, continuous and asymptotically hyperbolic on $ \mathbb{R} $, such that 
$ W_A ^s = X $, $ W_A ^u = Y $.
\end{proposition}
\begin{proof}
Let $ P $, $ Q $ be two projectors on $ X $ and $ Y $ respectively. We build 
first a path $ A^s $ on $ \R^+ $ such that $ W_{A^s} ^s = X $. 
Let $ A^s $ be the constant path $ I - 2P $ which is hyperbolic because
$ (I - 2P)^2 = I $. The spectral projector on the negative and 
positive eigenprojectors are, respectively, $ P $ and $ I - P $. A solution
$ x + y $ of (\ref{diff}) satisfies
\begin{align*}
x' &= A^s _- x + A^s _{\mp} y = - x \\
y' &= A^s _+ y + A^s _{\pm} x = y.
\end{align*}
Thus $ X_{A^s} (t) = e^{-t} P + e^t (I - P) $ and the stable space is $ X $.
Similarly we can define $ A^u (t) = 2Q - I $ for $ t < 0 $. The joint path 
$ A^u\# A^s $ is piecewise continuous. In order to find a smooth path 
consider a smooth function $ \vfi $ such that $ \vfi([-1/2,1/2]) = 1 $ and
$ \vfi (^c (-1,1)) = -1 $. Thus the path
\begin{gather*}
A = 
\left\{
\begin{array}{ll}
\vfi(t) P + (I - P) & t \geq 0 \\
\vfi(t) (I - Q) + Q & t < 0
\end{array}
\right.
\end{gather*}
is smooth. The solution of (\ref{diff}) with starting point $ x(0) + y(0) $
is
\begin{gather*}
\left\{
\begin{array}{ll}
x(t) &= e^{\Phi(t)} x(0) \\
y(t) &= e^t y(0)
\end{array}
\right.
\end{gather*}
where $ \Phi $ is the smooth function such that $ \Phi(0) = 0 $ and 
$ \Phi'(t) = \varphi(t) $. Since $ \Phi $ diverges to $ -\infty $ as 
$ t\rightarrow +\infty $ the stable space is $ X $. Since $ \Phi $ diverges
to $ +\infty $ as $ t\rightarrow -\infty $, hence the unstable space is $ Y $.
\end{proof}
\section{Perturbation of the stable space}
In the previous sections we have defined the stable (and unstable)
space and proved that is an element of $ G_s (E) $, the Grassmannian of
splitting subspaces. Thus, in the set
\begin{gather*}
  C_h (\mathbb{R}^+,\mathcal{L}(E)) = 
\set{A\in C(\orl{\mathbb{R}}^+,\mathcal{L}(E))}%
{\sigma(A(+\infty))\cap \img\mathbb{R} = \emptyset}
\end{gather*}
endowed with the uniform topology it is defined an application that maps 
$ A $ to $ W_A ^s $. In the next two theorems we prove that it is continuous 
and that if two paths differ by a path of compact operators then
the stable spaces are compact perturbation one of each other.
\begin{theorem}\rm{(cf. \cite{AM03b}, \textsc{Theorem} 3.1).}
  \label{small_perturbations}
  The map $ A\mapsto W_A ^s $ is continuous.
\end{theorem}
\begin{proof}
Since $ C_h (\R^+,\mathcal{L}(E)) $ is a metric space it is enough to prove that
the map is sequentially continuous. Let
$ \{\hskip .1em A_n \hskip .1em| \hskip .1em n\in\mb{N}\hskip .1em\} $ be a 
sequence in $ C_h (\R^+,\mathcal{L}(E)) $ converging to an asymptotically 
hyperbolic path $ A $. Let $ A(+\infty) = A_0 $. Call $ P^{\pm} $ the
spectral projectors on $ E^- (A_0) $ and $ E^+ (A_0) $ respectively. Since
$ A_0 $ is hyperbolic, there exist a pair $ (c,\lambda) $ such that
  \begin{gather*}
    \| e^{tA_0} P^{-} \|\leq c e^{-\lambda t },\quad
    \| e^{- t A_0} P^{+} \| \leq c e^{-\lambda t },\quad t\geq 0.
  \end{gather*}
The sequence $ \{A_n\} $ converges to $ A_0 $ uniformly as 
$ n\rightarrow\infty $. Moreover $ A(t) $ converges to $ A_0 $, as 
$ t\rightarrow +\infty $. Using triangular inequalities we can find
$ \tau\in\mathbb{R}^+ $ and $ N\in\mathbb{N} $ such that, for
every $ t\geq\tau $ and $ n\geq N $
\begin{gather}
\label{small_perturbations:1}
\no{A_n (t) - A_0}\leq\frac{\lambda}{Mc(1 + \sqrt{c})}.
\end{gather}
where $ M = \max\{\no{P^+},\no{P^-}\} $. Therefore for every $ n\geq N $ the 
paths $ A_{n,\tau} $, together with $ A_{\tau} $, fulfill the conditions of 
Proposition \ref{tow}. In 
particular there are $ S_n,S \in\mathcal{L}(E^-,E^+) $ such that
\begin{gather*}
X_{A_n} (\tau) W_{A_n} ^s = W_{A_{n,\tau}} ^s = \graph(S_n), \ \
X_A (\tau) W_A ^s = \graph(S). 
\end{gather*}
It is enough to prove that $ S_n $ converges to $ S $. In fact, by Proposition
\ref{continuity_of_graph},  this implies that $ X_{A_n} (\tau) W_{A_n} ^s $
converges to $ X_A (\tau) W_A ^s $ and the conclusion follows because 
$ X_{A_n} $ converges to $ X_A $ point-wise. For the remainder of the proof
we omit the subscript $ \tau $ from the paths. We recall that, by 
(\ref{graph_stab}), given $ x\in E^- $
\begin{gather}
S_n x = P^{+} ev_0 (I - L_{A_n})^{-1}
(X_{A_{n-}}(\cdot) x) = P^{+}
\sum_{k = 0} ^{\infty} ev_0 [L_{A_n} ^k
(X_{A_{n-}}(\cdot) x)].
\end{gather}
Since the estimate (\ref{small_perturbations:1}) holds for every $ n\geq N $
we can apply the Proposition \ref{op:1} to $ A_{n-} $ and $ A_{n+} $ 
in order to obtain uniform exponential estimates
\begin{align*}
&\no{X_{A_{n-}} (t) X_{A_{n-}} (s) ^{-1} x}\leq c e^{-\mu_- (t - s)} |x| \\
&\no{X_{A_{n+}} (t) X_{A_{n+}} (r) ^{-1} x}\leq c e^{\mu_+ (t - r)} |x|
\end{align*}
where $ \mu_{-} $ and $ \mu_+ $ are the same constants defined in 
Proposition \ref{tow}. By (\ref{1}) and (\ref{2}) there exists
$ 0 < \alpha < 1 $ such that $ \no{L_{A_n}}\leq\alpha $ for every 
$ n\geq N $. Then
\begin{gather}
\label{conv_dom}
|[L^k _{A_n} X_{A_{n-}} (\cdot) x]|\leq c\alpha ^k |x|.
\end{gather} 
In order to prove that $ S_n $ converges to $ S $ we show, by induction on 
$ k\in\mathbb{N} $, that $ L^k _{A_n} X_{A_{n-}} (\cdot) x $ converges to
$ L^k _{A} X_{A_{-}} (\cdot) x $ point-wise. Therefore the series
\begin{gather*}
\sum_{k = 0} ^{\infty} ev_0 [L_{A_n} ^k
(X_{A_{n-}}(\cdot) x)]
\end{gather*}
converges point-wise and, by (\ref{conv_dom}), is dominated uniformly on 
$ \mathbb{N} $ by the series of the sequence $ \{\alpha^k\} $. This is enough
to obtain the convergence of series to the point-wise limit. We claim that for
every $ t\geq 0 $
\begin{gather*}
\lim_{n\rightarrow\infty} L_{A_n} ^k X_{A_{n-}} (t) x = 
L_{A_0} ^k X_{A_{0-}} (t) x, \\
L_{A_n} ^k X_{A_{n-}} (t) x \in E^-, \ \text{if} \ k\text{ is even},\\
 L_{A_n} ^k X_{A_{n-}} (t) x \in E^+, \ \text{if} \ k\text{ is odd}.
\end{gather*}
If $ k = 0 $ the thesis follows since $ x\in E^- $ by hypothesis. 
Suppose it is true for $ k\in\mathbb{N} $. If $ k $ is odd, by 
(\ref{linear_part})
\begin{gather}
\label{kodd}
L_{A_n} ^{k + 1} X_{A_{n-}} (t) x = 
\int_0 ^t X_{A_{n-}} (t) X_{A_{n-}} (\tau)^{-1} A_{n\mp} (\tau) 
L_{A_n} ^k X_{A_{n-}} (\tau) x d\tau
\end{gather}
which belongs to $ E^- $. The last term converges to 
$ L_{A_n} ^k X_{A_{n-}} (t) x $ by inductive hypothesis. The other converges
by Proposition \ref{b-a} and the fact that $ A_n $ converges to $ A $. 
The integrand of (\ref{kodd}) is bounded in $ [0,t] $ by
\begin{gather*}
c^2 e^{-\mu_- (t - \tau)}\sup_n \no{A_n}_{\infty} \alpha^k |x|.
\end{gather*}
Then, by the dominate convergence theorem, the left member of (\ref{kodd}) 
converges point-wise. If $ k $ is even, by (\ref{linear_part})
\begin{gather}
\label{keven}
L_{A_n} ^{k + 1} X_{A_{n-}} (t) x = -\int_t ^{\infty} 
X_{A_{n+}} (t) X_{A_{n+}} (\tau)^{-1} A_{n\pm} (\tau)
L_{A_n} ^k X_{A_{n-}} (\tau) x d\tau.
\end{gather}
Similarly the integrand converges point-wise and is dominated by 
\begin{gather*}
c^2 \alpha^k e^{\mu_+ (t - \tau)} | x | 
\sup_n \no{A_n}_{\infty}\in L^1 (\mathbb{R}^+).
\end{gather*}
Again, by the dominate convergence theorem, we clinch the point-wise convergence
of (\ref{keven}) and the inductive step is concluded.
Thus
\begin{align*}
\lim_{n\rightarrow\infty} & ev_0 [L^k _{A_n} X_{A_{n-}} (\cdot) x] = 
ev_0 [L^k _{A_0} X_{A_{0-}} (\cdot) x],\\
|&ev_0 [L^k _{A_n} X_{A_{n-}} (\cdot) x]|\leq c\alpha^k |x|
\end{align*}
for every $ k\in\mathbb{N} $ we have convergence of the series. 
\end{proof}
We state without proof a couple of facts on compactness useful for the
next theorem.
\begin{lemma}\label{res1}
Let $ J $ be an interval of the real line, 
$ K\in L^1 (J,\mathcal{L}(E)) $ such that $ K(t)\in\mathcal{L}_c (E) $ almost 
everywhere. Then the map
\begin{gather*}
C_b (J,E)\ni u\longmapsto\int_J K(\tau)u(\tau)d\tau\in E
\end{gather*}
is a compact operator.
\end{lemma}
\begin{proof}
When $ K $ is constant the map is obtained by composition on the left with
a compact operator. If $ K $ is a characteristic function on $ J $ it is sum
of compact operators. We conclude with the density of characteristic functions
in $ L^1 (J,\mathcal{L}(E)) $ and closeness of compact operators.
\end{proof}
\begin{theorem}[Ascoli-Arzel\`a]
\label{ascoli-arzela}
Let $ X $ be a compact metric space, $ E $ a Banach space. A bounded subset
$ \mathcal{W}\subset C(X,E) $ is relatively compact if and only is 
equicontinuous and, for every $ x\in X $, the set 
$ \mathcal{W}(x) = \set{f(x)}{f\in\mathcal{W}} $ is relatively compact in 
$ E $. 
\end{theorem}
For a proof see \cite{Die87}, pp. 142--143.
\begin{theorem}\rm{(cf. \cite{AM03b}, \textsc{Theorem} 3.6).}
\label{compact}
Let $ A,B\in C_h (\R^+,\mathcal{L}(E)) $ be such that $ K = B - A $ is a compact 
operator for every $ t $. Then $ W_A ^s $ is a compact perturbation of
$ W_B ^s $ and
\begin{gather*}
\dim (W_A ^s,W_B ^s) = \dim (E^- (A(+\infty)),E^- (B(+\infty))).
\end{gather*}
\end{theorem}
\begin{proof}
Up to a time shift we can assume that $ A $ and $ B $ satisfy the conditions 
of the Proposition \ref{tow}. Then $ W_A ^s $ and $ W_B ^s $ are graph of operators
\begin{align*}
S_A &\in\mathcal{L}(E^- (A(+\infty)),E^+ (A(+\infty))), \\ 
S_B &\in\mathcal{L}(E^- (B(+\infty)),E^+ (B(+\infty))).
\end{align*}
Let $ P^- (A) $ and $ P^- (B) $ be the spectral projectors of the negative
eigenspaces. Observe that
\begin{equation*}
W_A ^s = \ker(P^+ (A) - S_A P^- (A)), \ W_B ^s = \ker(P^+ (B) - S_B P^- (B)).
\end{equation*}
The differences $ P^\pm (A) - P^\pm (B) $ are compact operators; we wish
to prove that $ S_A P^- (A) - S_B P^- (B) $ is also compact. Therefore 
$ W_A ^s $ is a compact perturbation of $ W_B ^s $ and, by Proposition 
\ref{cp:2},
\begin{equation*}
\begin{split}
\dim(W_A ^s, W_B ^s) &= 
\dim(\ker(P^+ (A) - S_A P^- (A)),\ker(P^+ (B) - S_B P^- (B))) \\
&= \dim(\ran(P^+ (B) - S_B P^- (B)),\ran((P^+ (A) - S_A P^- (A))) \\
&= \dim(E^+ (B(+\infty)),E^+ (A(+\infty))) \\
&= \dim(E^- (A(+\infty)),E^- (B(+\infty))),
\end{split}
\end{equation*}
which is the thesis when $ W_A ^s $ and $ W_B ^s $ are graphs. In the general 
case there exists a real $ \tau $ such that $ A(\cdot +\tau) $ and 
$ B(\cdot +\tau) $ satisfy the conditions of Proposition \ref{tow}. Then
\begin{equation*}
\begin{split}
\dim (W^s _{A(\cdot +\tau)},W^s _{B(\cdot +\tau)}) &= 
\dim (X_A (\tau) W_A ^s, X_B (\tau) W_B ^s) \\
&=\dim (W_A ^s, X_A (\tau)^{-1} X_B(\tau) W_B ^s) \\
&=\dim(W_A ^s, W_B^s) + \dim(W_B ^s, X_A (\tau)^{-1} X_B(\tau) W_B ^s)
\end{split}
\end{equation*}
The last term of the equality is $ 0 $ because $ X_A (\tau)^{-1} X_B (\tau) $ 
can be written as 
$ I + (X_A (\tau)^{-1} - X_B (\tau) ^{-1}) X_B(\tau) $ which is an invertible
operator of the Fredholm group. Then the
conclusion follows from Proposition \ref{cp:2}. Now we write, by 
(\ref{graph_stab})
\begin{gather*}
S_A P^{-} (A) = P^+ (A) ev_0 [(I - L_A)^{-1} X_{A_{-}}(\cdot)P^{-}(A)], \\
S_B P^{-} (B) = P^+ (B) ev_0 [(I - L_B)^{-1} X_{B_{-}}(\cdot)P^{-}(B)].
\end{gather*}
Using the Theorem of Ascoli--Arzel\`a we prove first that $ L_A - L_B $ is a 
compact operator on $ C_b (\R^+,E) $. In fact let $ \mathcal{W} $ be a bounded 
subset of $ C_b (\R^+,E) $. Given $ u\in\mathcal{W} $ for every 
$ t\in\mathbb{R}^+ $ we have
\begin{align*}
(L_A u)' (t) &= [P^+ (A) A(t) P^+ (A) + P^- (A) A(t) P^- (A)](L_A - I) u (t) + 
A(t) u (t)\\
(L_B u)' (t) &= [P^+ (B) B(t) P^+ (B) + P^- (B) B(t) P^- (B)](L_B - I) u (t) + 
B(t) u (t).
\end{align*}
Since $ A $ and $ B $ are bounded the set 
$ \set{(L_A - L_B)u(t)}{u\in\mathcal{W}} $ is bounded by a constant that 
depends on $ t $ at most. Then $ (L_A - L_B) \mathcal{W} $ is 
equicontinuous. Now we prove that the set
\begin{gather*}
\set{(L_A - L_B)u (t)}{u\in\mathcal{W}} 
\end{gather*}
is relatively compact. The prove is carried on interpolating $ L_A $ and
$ L_B $ and applying Lemma \ref{res1} to the differences as follows
\begin{equation*}
\begin{split}
(L_A - L_B)u(t) &= P^- (A)\int_0 ^t X_{A_-} (t) X_{A_-} (\tau)^{-1} 
P^- (A) A (\tau) P^+ (A) u(\tau) d\tau \\
&- P^- (B)\int_0 ^t X_{B_-} (t) 
X_{B_-} (\tau)^{-1} P^- (B) B (\tau) P^+ (B) u(\tau) 
d\tau \\
&- P^+ (A)\int_t ^\infty X_{A_+} (t) X_{A_+} (\tau)^{-1} 
P^+ (A) A (\tau) P^- (A) u(\tau) d\tau \\
&+ P^+ (B)\int_t ^\infty X_{B_+} (t) 
X_{B_+} (\tau)^{-1} P^+ (B) B (\tau) P^- (B) u(\tau) 
d\tau.
\end{split}
\end{equation*}
Since $ X_A (t) - X_B (t) $ and $ A(t) - B(t) $ are compact by interpolation
we obtain the sum of two integrals on $ [0,t] $ and $ [t,+\infty) $ 
with compact integrands. We conclude by applying Lemma \ref{res1} to 
the two integrands. By composition $ S_A P^- (A) - S_B P^- (B) $ is compact.
\end{proof}
%%% TEXEXPAND: END FILE ./chapter3-v2.tex
%%% TEXEXPAND: INCLUDED FILE MARKER ./chapter4-v2.tex
\chapter{Ordinary differential operators on Banach spaces}
\label{chap4}
Given a path $ A \in C(\mathbb{R},\mathcal{L}(E)) $ we study the properties of
the differential operator $ F_A u = u' - A u $. %
\glsadd{labFA}
When $ E $ is a Hilbert space the operator can be defined in 
$ H^1 (\mathbb{R},E) $ with values in $ L^2 (\mathbb{R},E) $. 
By \textsc{Theorem} 5.1 of \cite{AM03b} the operator $ F_A $ is Fredholm if 
and only if the pair $ (W_A ^s,W_A^u) $ is a Fredholm pair and
\begin{gather*}
\ind F_A = \ind(W_A ^s,W_A ^u).
\end{gather*}
In this chapter we prove the same result when $ E $ is a Banach space and
the operator $ F_A $ is defined on $ C^1 _0 (\mathbb{R},E) $ and takes
values in $ C_0 (\mathbb{R},E) $, where
\begin{align*}
C_0 (\mathbb{R},E) &= \set{u\in C(\mathbb{R},E)}%
{\lim_{t\rightarrow\pm\infty} u(t) = 0} \\
C^1 _0 (\mathbb{R},E) &= \set{u\in C^1 (\mathbb{R},E)}%
{\lim_{t\rightarrow\pm\infty} u(t) = 0, \ 
\lim_{t\rightarrow\pm\infty} u'(t) = 0}.
\glsadd{labC0}
\glsadd{labC10}
\end{align*} 
We remark that the result also holds when $ F_A $ is defined on the
Sobolev space $ W^{1,p} (\mathbb{R},E) $ with values in $ L^p (\mathbb{R},E) $
with $ p\geq 1 $. 
\section{The operators $ F^+ _A $ and $ F_A ^- $}
Consider the spaces
\begin{align*}
C _0 (\mathbb{R}^+,E) &= \set{u\in C^1 (\mathbb{R}^+,E)}%
{\lim_{t\rightarrow +\infty} u(t) = 0} \\
C^1 _0 (\mathbb{R}^+,E) &= \set{u\in C^1 (\mathbb{R}^+,E)}%
{\lim_{t\rightarrow +\infty} u(t) = 0,\ \lim_{t\rightarrow +\infty} u' (t) = 0}
;
\end{align*}
\glsadd{labC10+}
\glsadd{labC0+}
we define the operator
\begin{gather*}
F_A ^+\colon C^1 _0 (\mathbb{R}^+,E)\rightarrow C_0 (\mathbb{R}^+,E), \ 
u\mapsto u' - Au
\end{gather*}
and similarly $ F_A ^- $ on $ C^1 _0 (\mathbb{R}^-,E) $. We wish to prove
that when $ A $ is asymptotically hyperbolic $ F_A ^+ $ has a right inverse. 
First observe that in special case $ A \equiv A_0 $
the operator $ F_{A_0} $ is invertible and its inverse is given by
\begin{gather}
R_{A_0} h = G_{A_0} * h
\end{gather}
for any $ h\in C^1 _0 (\mathbb{R},E) $, where
\begin{gather}
G_{A_0} (t) = e^{tA_0}
\left[ P^- (A_0) 1_{\R^+} - P^+ (A_0) 1_{\R^{-}} \right]
\end{gather}
where $ P^- (A_0) $ and $ P^+ (A_0) $ are the spectral projectors of $ A_0 $ 
relative to decomposition $ \sigma(A_0) = \sigma^+ \cup \sigma^- $ and 
$ 1_{\mathbb{R}^+} $ and $ 1_{\mathbb{R}^-} $ are the characteristic
functions of the subsets $ \mathbb{R}^+ $ and $ \mathbb{R}^- $. Exponential
estimates of $ G_{A_0} $ makes $ G_{A_0} * h $ a continuously differentiable
function in $ C^1 _0 (\mathbb{R},E) $. Moreover
\glsadd{lab1S}
\begin{equation*}
\begin{split}
F_{A_0} (G_{A_0} * h) (t) &= (G_{A_0} * h)' - A_0 (G_{A_0} * h) \\
&= A_0 (G_{A_0} * h) + P^- h(t) + P^+ h(t) - A_0 (G_{A_0} * h) = h;
\end{split}
\end{equation*}
hence $ R_{A_0} $ is a right inverse of $ F_{A_0} $. Otherwise
\begin{equation*}
\begin{split}
G_{A_0} * F_{A_0} u &= 
\int_{-\infty} ^t e^{(t - \tau) A_0} P^- (u' - A_0 u) d\tau - 
\int_t ^{+\infty} e^{(t - \tau) A_0} P^+ (u' - A_0 u) d\tau \\
\end{split}
\end{equation*}
integration by parts lead to
\begin{equation*}
\begin{split}
&\int_{-\infty} ^t e^{(t - \tau) A_0} P^- (u' - A_0 u) d\tau = P^- u(t) \\
-&\int_t ^{+\infty} e^{(t - \tau) A_0} P^+ (u' - A_0 u) d\tau = P^+ u(t)
\end{split}
\end{equation*}
taking the sum we conclude.
If $ A $ is a asymptotically hyperbolic path we know that $ W_A ^s $ 
and $ W_A ^u $ are closed and have topological complements. Choose $ X_s $ 
and $ X_u $ such that $ X_s \oplus W_A ^s = E = X_u \oplus W_A ^u $ and let 
$ P_s = P(W_A ^s,X_s) $, $ P_u = P(W_A ^u,X_u) $. Define
\begin{gather}
G_{A,P_s}^+ (t,\tau) = X_A (t) 
\left[P_s 1_{\mathbb{R}^{+}} - (I - P_s) 1_{\mathbb{R}^{-}} \right] 
X_A (\tau) ^{-1}
\label{GAt:1} \\
G_{A,P_u}^- (t,\tau) = X_A (t) 
\left[(I - P_u) 1_{\mathbb{R}^{+}} - P_u 1_{\mathbb{R}^{-}} \right] 
X_A (\tau)^{-1}
\label{GAt:2}
\end{gather}
\begin{proposition}
If $ A $ is an asymptotically hyperbolic path there are positive 
constants $ (c,\lambda) $ such that
\begin{gather}
\label{conv}
\| G_{A,P_s} ^{+} (t,\tau) \| \leq c e^{-\lambda | t - \tau |}
\end{gather}
for every $ (t,\tau)\in\R^+ \times \R^+ $.
\end{proposition}
\begin{proof}
By the Theorem \ref{rec}, if $ P_s $ is a projector on $ W_A ^s $,
$ I - P_s ^* $ is a projector on $ (W_A ^s)^{\bot} = W_{-A^*} ^s $.
Hence $ (G_{A,P_s}^+ (t,\tau))^*  = G_{-A^*, I - P_s ^*} (\tau,t) $ and
it's enough to prove the statement for $ t\geq\tau\geq 0 $. We have
\begin{equation}
\begin{split}
\label{GAt}
\| G_{A,P_s}^+ (t,\tau) \| &\leq \| X_A (t) P_s X_A (t)^{-1} \| 
\cdot \| X_A (t) X_A (\tau)^{-1} \| \\
&\leq c' e^{-\lambda (t - \tau)} \| X_A (t) P_s X_A (t)^{-1} \|. 
\end{split}
\end{equation}
For every $ t\in\R^{+} $ $ P(t) = X_A (t) P_s X_A (t)^{-1} $ is a 
projector onto $ X_s (t) = X_A (t) W_A ^s $ and $ I - P(t) $ onto
$ X_u (t) = X_A (t) X_u $. By Theorem \ref{rec}, i) and iii), $ X_s (t) $ 
converges to $ E^- (A(+\infty)) $, and $ X_u (t) $ to $ E^+ (A(+\infty)) $.
Then by Proposition \ref{continuity_of_splits} the $ P(t) $ is bounded 
(in fact converges to a projector). Then the last term of (\ref{GAt}) is 
estimated by $ Mc' e^{-\lambda (t - \tau)} $.
\end{proof}
This allows us to prove the following
\begin{proposition}
\label{rinverse}
Let $ A $ be a bounded continuous path on $ \R^+ $. Then $ F_{A,+} $ is a 
bounded operator. Moreover if $ A $ is asymptotically hyperbolic $ F_A ^+ $ has
right inverse also and one is given by
\begin{gather}
\label{total}
R_{A,P_s} ^+ h (t) = \int_{\mathbb{R}} G_{A,P_s} ^+ (t,\tau) 
h(\tau) 1_{\mathbb{R}^+} (\tau) d\tau.
\glsadd{labRAPs}
\glsadd{labRAPu}
\end{gather}
where $ P_s $ is a projector onto the stable space. 
\end{proposition}
\begin{proof}
That $ F_A ^+ $ is bounded it's clear from the definition. We prove 
that $ R_{A,P_s} ^+ $ maps $ C_0 (\mathbb{R}^+,E) $ in 
$ C^1 _0 (\mathbb{R}^+, E) $. In fact if $ h\in C_0 $ then 
$ R_{A,P_s} ^+ h(t) $ is
\begin{gather*}
\int_0 ^t X_A (t) P_s X_A (\tau)^{-1} h(\tau) d\tau -
\int_t ^{+\infty} X_A (t) (I - P_s) X_A (\tau)^{-1} h(\tau) d\tau
\end{gather*}
hence is continuous and continuously differentiable. 
By the (\ref{conv}) we have
\begin{equation}
\label{bnd:1}
\begin{split}
\| R_{A,P_s} ^+ h(t) \| &\leq\int_{\R^+} 
c e^{-\lambda |t - \tau|} | h(\tau) | d\tau
\leq\no{h}_{\infty}\int_{\R^+} e^{-\lambda |t - \tau|} d\tau \leq 
\frac{\| h \|_{\infty}}{\lambda} e^{-\lambda t}
\end{split}
\end{equation}
hence $ R_{A,P_s} ^+ h \in C_0 (\mathbb{R}^+,E) $. Since its derivative is
\begin{gather}
\label{bnd:2}
(R_{A,P_s} ^+ h)' = A R_{A,P_s} ^+ h + h 
\end{gather}
and $ A $ is bounded, we have $ R_{A,P_s} ^+ h\in C^1 _0 (\mathbb{R}^+,E) $. 
Actually (\ref{bnd:1}) and (\ref{bnd:2}) say that $ R_{A,P_s} $ is a bounded
operator. Still from (\ref{bnd:2}) 
\begin{gather*}
F_A ^+ R_{A,P_s} ^+ h = (R_{A,P_s} ^+ h) ' - A R_{A,P_s} ^+ h = 
 A R_{A,P_s} ^+ h + h - A R_{A,P_s} ^+ h = h.
\end{gather*}
Then $ R_{A,P_s} ^+ $ is a right inverse of $ F_A ^+ $.
\end{proof}
Similarly we have
\begin{proposition}
\label{linverse}
If $ A $ is a bounded continuous path on $ \R^- $ the operator 
$ F_{A,-} $ is bounded and admits a right inverse if $ A $ is asymptotically 
hyperbolic. One is given by
\begin{gather*}
R_{A,P_u} ^- h (t) = \int_{\mathbb{R}} G_{A,P_u} ^- (t,\tau) 
h(\tau) 1_{\mathbb{R}^-} (\tau) d\tau.
\end{gather*}
where $ P_u $ is a projector onto the unstable space.
\end{proposition}
The proof is completely similar and we omit it. 
\begin{example}
Notice that if $ A_0 $ is 
invertible but not hyperbolic these operators  can be non surjective. For 
example let $ E $ be the Euclidean space $ \R^2 $ and define 
\begin{gather*}
A_0 = 
\begin{pmatrix}
0 & b \\
-b & 0
\end{pmatrix}
, \ 
e^{A_0} = 
\begin{pmatrix}
\cos\theta & \sin\theta \\
- \sin\theta & \cos\theta 
\end{pmatrix}
= R_{\theta}.
\end{gather*}
First observe that $ F_{A_0} ^+ $ is injective: 
given $ u $ in $ C^1 _0 (\mathbb{R}^+,E) $ such that $ F_{A_0} ^+ u = 0 $. We 
have $ u(t) = R_{t\theta} u(0) $ by uniqueness of the solutions of 
(\ref{cauchyl}). Since $ R_{\theta} $ is an isometry 
$ | u(t) | = | u(0) | $ for every $ t\geq 0 $. Taking the limit as 
$ t\rightarrow +\infty $ we obtain $ u(0) = 0 $, hence $ u $ is zero.
Now let $ h $ be a continuous function on $ \mathbb{R}^+ $ that vanishes
at $ +\infty $ and $ u $ in $ C^1 _0 (\mathbb{R}^+,E) $ 
such that $ F_{A_0} ^+ u = h $. Since $ F_{A_0} ^+ $ is injective
\begin{gather}
\label{fund}
u(t) = e^{t A_0} \left(\int_0 ^t e^{-s A_0} h(s) ds + u(0)\right)
\end{gather}
is the only solution of the problem. Fix $ v_0 $ in $ E\setminus\{0\} $ and
$ \alpha $ in $ C_0 (\R^+,\R^+) $ not integrable. Let 
$ h(s) = \alpha(s) R_{s\theta} v_0 $. Since $ R_{\theta} $ is an isometry, 
the norm of $ u(t) $ is equal to the one of
\begin{gather}
\label{norm}
\int_0 ^t R_{-s\theta} h(s) ds + u(0) = \int_0 ^t \alpha(s) R_{-s\theta}
(R_{s\theta} ) v_0 ds + u(0) = \int_0 ^t \alpha(s) ds \hskip 2pt v_0 + u(0).
\end{gather}
Since the last term of (\ref{norm}) does not converge to $ 0 $ as 
$ t\ra +\infty $ the function $ h $ is not in the image of $ F_{A_0} ^+ $.
\end{example}
Given a continuous function $ h $ in $ C_0 (\mathbb{R}^+,E) $ evaluating
$ R_{A,P_s} ^+ h $ at $ t = 0 $ we obtain a vector of $ \ker P_s $. Similarly
we can evaluate $ R_{A,P_u} ^- h $ and we have a continuous functions
\begin{align*}
&r_{A,P_s} ^+ \colon C_0 (\mathbb{R}^+,E)\rightarrow X_s, \ h\mapsto 
ev_0 R_{A,P^s} ^+ h \\ 
&r_{A,P_u} ^- \colon C_0 (\mathbb{R}^-,E)\rightarrow X_u, \ h\mapsto
ev_0 R_{A,P^u} ^- h.
\end{align*} 
\glsadd{labrAPs}
\glsadd{labrAPu}
When no ambiguity occurs on the choice of the path $ A $ and the projectors
we simply denote them by $ r^+ $ and $ r^- $ respectively.
We have the following 
\begin{proposition}\rm{(cf. \cite{AM03b}, \textsc{Lemma} 4.2).}
\label{sur}
The functions $ r^+ $ and $ r^- $ are linear and continuous
applications and map $ C^{\infty} _c ((0,+\infty),E) $ onto $ X_s $ and 
$ C^{\infty} _c ((-\infty,0),E) $ onto $ X_u $.
\end{proposition}
\begin{proof}
We prove the assertion for $ r^+ $. Since $ R_{A,P_s} ^+ $ is bounded, 
$ r^+ $ is bounded. Let $ v $ be a vector of $ E $ and 
$ \varphi\in C^{\infty} _c ((0,+\infty),\mathbb{R}) $ a smooth function such 
that
\begin{gather*}
\label{operator_U}
U = - \int_{\mathbb{R}} \varphi (\tau) X_A (\tau)^{-1} d\tau
\end{gather*}
is an invertible operator on $ E $. We choose 
$ h = \vfi\cdot U^{-1} v $ 
\begin{equation*}
\begin{split}
r ^+ h &= - (I - P_s) \int_0 ^{+\infty} X_A (\tau)^{-1} \vfi(\tau)
U^{-1} v d\tau \\ 
&=- (I - P_s) \int_0 ^{+\infty} X_A (\tau)^{-1} \vfi(\tau) d\tau U^{-1} v = 
 (I - P_s) v. 
\end{split}
\end{equation*}
\end{proof}
In the above proof one could remark that choosing a smooth compact
supported function $ \psi $ on $ \mathbb{R}^+ $ such that 
$ \int \psi = 1 $, for every $ v\in X_s $ the function 
$ h (t) = - \psi(t) X_A (t)\cdot v $ still works. However $ h $ is at
most as regular as $ X_A $.
\section{Fredholm properties of $ F_A $}
We show that, as in the Hilbert setting, that the Fredholmness of $ F_A $
depend on the Fredholmness of the pair of subspaces
$ (W_A ^s, W_A ^u) $.
\begin{lemma}\rm{(cf. \cite{AM03b}, \textsc{Proposition} 5.2).}
\label{charact_fred}
We have the following characterizations of $ \ker F_A $ and $ \ran F_A $:
\begin{align}
\ker F_A &= \set{u\in C^1 _0}{u(0)\in W_A ^s \cap W_A ^u}\label{kr:1} \\
\ran F_A &= \set{h\in C_0}
{r_{A,P_s} ^+ h - r_{A,P_u} ^- h\in W_A ^s + W_A ^u} \label{kr:2} \\
\orl{\ran F_A} &= \set{h\in C_0}
{r_{A,P_s} ^+ h - r_{A,P_u} ^- h\in\orl{W_A ^s + W_A ^u}}\label{kr:3}
\end{align}
\end{lemma}
\begin{proof}
We omit the proof of (\ref{kr:1}) that comes straightforwardly from
the definition of stable and unstable subspaces. Let $ h\in\ran F_A $ and
$ u\in C^1 _0 $ such that $ F_A u = h $. By Proposition \ref{rinverse} we 
have a decomposition 
$ C^1 _0 (\mathbb{R}^+,E) = \ker F_A ^+ \oplus\ran R_{A,P_s} ^+ $. Thus
\begin{equation}
\label{pm}
\begin{split}
u^+ &= X_A (t) u_0 + R_{A,P_s} ^+ h^+ \\
u^- &= X_A (t) v_0 + R_{A,P_u} ^- h^-
\end{split}
\end{equation}
where $ u^+ $ and $ u^- $ are the restrictions of $ u $ to the positive 
(respectively negative) real line. Evaluating in $ 0 $ and taking the 
difference of the two equations we
obtain 
\begin{gather*}
W_A ^s + W_A ^u \ni u_0 - v_0 = r^- _{A,P_u} h - r^+ _{A,P_s} h.
\end{gather*}
To prove the converse let $ h\in C_0 $ such that 
$ r ^+ h - r ^- h\in W_A ^s + W_A ^u $. By Propositions
\ref{rinverse} and \ref{linverse} we have $ u^+ $ and $ u^{-} $ such that
\begin{gather}
F_A ^+ u^+ = h^+,\quad F_A ^- u^- = h^-.
\end{gather}
In order to exhibit an element of $ C^1 _0 $ such that $ F_A u = h $ we want 
to find suitable $ u^+ $ and $ u^- $ such that $ u^- \# u^+ $ is a continuous
function and continuously differentiable. Hence it's enough to choose
$ u_0 $ and $ v_0 $ in (\ref{pm}) such that 
\begin{gather}
u^+ (0) = u^{-} (0) \label{cond:1}\\
{u^+} '(0) = {u^{-}} ' (0) \label{cond:2},
\end{gather}
as before evaluate (\ref{pm}) in $ 0 $ and set (\ref{cond:1}) in the
left sides. If we choose $ u_0 $ and $ v_0 $ such that 
$ u_0 - v_0 = r^+ h - r^- h = w $ the joint function
$ u_-\# u_+ $ is continuous. Differentiating the (\ref{pm})
\begin{equation*}
\begin{split}
{u^+}' (t) &= A(t) X_A (t) u_0 + A(t) R^+ _{A,P_s} h^+ (t) + h^+ (t) \\
{u^-}' (t) &= A(t) X_A (t) v_0 + A(t) R^- _{A,P_u} h^- (t) + h^- (t)
\end{split}
\end{equation*}
we get $ A(0) (u_0 - v_0 - w) = 0 $, hence any choice in 
$ W_A ^s \times W_A ^u $ that makes $ u^- \# u^+ $ continuous it also
makes it $ C^1 $.
\vskip .2em
The proof of the left inclusion of (\ref{kr:3}) is completely
similar to the above step. Conversely suppose that $ h $ belongs
to the right set of the (\ref{kr:3}). Let $ \var > 0 $ and
$ \delta = 1/(\| I - P_s \|\cdot \| U^{-1} \|) $ where $ U $ is the
operator defined in (\ref{operator_U}). Set 
$ w = r_{A,P_s} ^+ h - r_{A,P_u} ^- h $. There exists 
$ x\in W_A ^s + W_A ^u $ such that $ | w - x | < \delta $. By Proposition
\ref{sur} 
\begin{gather*}
r^+ h_{\delta} = (I - P_s) (w - x), \ \ 
h_{\delta} = - \varphi U^{-1} (w - x)
\end{gather*}
and $ \no{h_{\delta}} < \var $. Since $ h_{\delta} $ has compact support
in $ (0,+\infty) $ it can be extended on $ \mathbb{R}^- $ with the constant
value $ 0 $. Thus
\begin{gather*}
r^+ (h - h_{\delta}) - r^- (h - h_{\delta}) = w - r^+ h_{\delta} = 
x + P_s (w - x)
\end{gather*}
is an element of $ W_A ^s + W_A ^u $ hence, by (\ref{kr:2}), 
$ h - h_{\delta} $ is in the image of $ F_A $. 
\end{proof}
We conclude the chapter with the relationship between the Fredholm
properties of $ F_A $ and the Fredholm properties of the 
pair $ (W_A ^s, W_A ^u) $.
\begin{theorem}\rm{(cf. \cite{AM03b}, \textsc{Theorem} 5.1).}
\label{facts}
If $ A $ is an asymptotically hyperbolic path the following facts hold:
\begin{enumerate}
\item $ F_A $ has closed range if and only if $ W_A ^s + W_A ^u $ is
closed,
\item $ F_A $ is onto if and only if $ W_A ^s + W_A ^u = E $,
\item $ F_A $ is semi-Fredholm if and only $ (W_A ^s, W_A ^u ) $ is a
semi-Fredholm pair; in this case we also have 
$ \ind F_A = \ind(W_A ^s,W_A ^u) $.
\end{enumerate}
\end{theorem}
\begin{proof}
If $ W_A ^s + W_A ^u $ is closed the two sets on the right of
(\ref{kr:1}) and (\ref{kr:2}) are equal, hence $ \ran F_A $ coincides with
its closure. Conversely, suppose $ \ran F_A $ is
closed and let $ w $ be an element of $ W_A ^s + W_A ^u $. By Proposition
\ref{sur}, there exists $ h $ smooth with compact support such that
\begin{gather*}
w = P_s w + (I - P_s) w = P_s w + r^+ h - r^- h
\end{gather*}
hence $ r^+ h - r^- h $ is in the closure of $ W_A ^s + W_A ^u $. Then, by
hypothesis $ r^+ h - r^- h\in W_A ^s + W_A ^u $, hence 
$ w\in W_A ^s + W_A ^u $ and i) is proved. Suppose $ F_A $ is onto, that is
the range of $ F_A $ is closed. By i) $ W_A ^s + W_A ^u $ is also closed
and there is an isomorphism of Banach spaces
\begin{gather}
\label{isomorphism}
C_0 /\ran F_A \rightarrow E/W_A ^s + W_A ^u, \ \ 
h + \ran F_A \mapsto r^+ h - r^-h.
\end{gather}
It is injective by (\ref{kr:2}). Given $ x\in E $ the element 
$ h + \ran F_A $ such that $ r^+ h - r^- h = (I - P^s) x $ is in the 
counter-image of $ x + W_A ^s + W_A ^u $, therefore is surjective. The 
continuity follows straightforwardly from the definition of the norm for
a quotient space. In fact, for every $ u\in C^1 _0 $, we have
\begin{equation*}
\begin{split}
\dist(r^+ h - r^-h,W_A ^s + W_A ^u)&\leq
\dist(r^+ h - r^-h,r^+ F_A u - r^- F_A u) \\
&\leq(\no{r^+} + \no{r^-})|h - F_A u|.
\end{split}
\end{equation*}
Taking the infimum over $ C^1 _0 $ we prove that the application is bounded. 
We conclude with the open mapping theorem. If $ F_A $ is onto the quotient 
spaces $ C_0 /\ran F_A $ is the null space, then, by (\ref{isomorphism}) 
$ W_A ^s + W_A ^u = E $ and the converse is similar, hence ii) is proved. 
If $ F_A $ is semi-Fredholm $ \ran F_A $ is closed, hence $ W_A ^s + W_A ^u $ 
is also closed. By (\ref{kr:1}) and (\ref{isomorphism}) the index of $ F_A $ 
and the one of the pair $ (W_A ^s,W_A ^u) $ coincide, this proves iii).
\end{proof}
%%% TEXEXPAND: END FILE ./chapter4-v2.tex
%%% TEXEXPAND: INCLUDED FILE MARKER ./chapter5-v2.tex
\chapter{Spectral flow}
\label{chap5}
Given a continuous path of essentially hyperbolic operators, we can define
an integer called \textsl{spectral flow}. The definition we provide in this
chapter generalizes the one given by J.Phillips for paths
of Fredholm and self-adjoint operators and coincides with the one given
by C.~Zhu and Y.~Long in \cite{ZL99}. We wish to make our notation
coherent with the latter, thus we use the notation
$ [P - Q] $ for relative dimension $ \dim(P,Q) $ when $ P - Q $ is a compact
operator. We show that the definition of spectral flow depends only on the 
class of fixed-endpoints homotopy of a path. Moreover, the spectral flow of 
the catenation of two paths is the sum of the spectral flows of the paths, 
hence we have a group homomorphism
\begin{gather*}
\Sf_{A_0}\colon\pi_1 (e\mathcal{H}(E),A_0)\rightarrow\mathbb{Z}.
\end{gather*}%
\glsadd{labSfA}
\glsadd{labeHE}
In chapter \ref{chap2} we established a homotopy equivalence between 
the space of essentially hyperbolic operators $ e\mathcal{H}(E) $ and the
space of idempotents $ \mathcal{P}(\mathcal{C}) $ of the Calkin algebra, we 
denoted it by $ \Psi $ and defined it as
\begin{gather*}
\Psi(A) = P^+ (p(A))
\end{gather*}
where $ P^+ (p(A)) $ is the eigenprojector relative to the positive complex
half-plane. In Theorem \ref{f=-p} we prove that there is a strict relation 
between the spectral flow and the homomorphism $ \varphi $ defined through 
the exact sequence of the bundle 
$ (\mathcal{P}(E),\mathcal{P}(\mathcal{C}),\prc) $. Precisely,
\[
\Sf_{A_0} = - \varphi_{P\sp + (A_0)} \circ\Psi
\]
Thus the spectral flow inherits all the properties of the index $ \varphi $.
The equality holds for every Banach space and gives a characterization of
the paths whose spectral flow is zero and necessary and sufficient conditions
in order to have nontrivial spectral flow.\vskip .2em
In the last section we extend the definition of spectral flow to 
asymptotically hyperbolic and essentially hyperbolic paths. We prove that 
if $ A $ is also an \textsl{essentially splitting path} the differential 
operator $ F_A $ is Fredholm and
\begin{gather*}
\ind F_A = -\Sf(A) = \dim(E^- (A(+\infty)),E^- (A(-\infty))).
\end{gather*}
In general, none of the these equalities holds. Counterexamples are known
even in the Hilbert spaces.
\section{Essentially hyperbolic operators} 
\label{Essentially hyperbolic operators}
We recall that an operator $ A $ is said \emph{essentially hyperbolic} if 
$ A + \mathcal{L}_c $ is a hyperbolic element of the Calkin algebra
$ \mathcal{C} $. 
We denote by $ e\al{H}(E) $ the set of the essentially hyperbolic operators.
\begin{lemma}[Structure of the spectrum]
\label{structure}
Let $ A $ be a bounded operator, $ D $ the set of isolated points of 
$ \sigma(A) $. Then $ \p \sigma(A)\setminus D \subset\sigma_e(A) $.
\end{lemma}
\begin{proof}
We argue by contradiction: let $ \lambda_0 \in\sigma(A)\setminus D $. If 
$ \lambda_0 \not\in\sigma_e (A) $ $ A - \lambda_0 $ is Fredholm of index 
$ k $. There exists $ r > 0 $ such that for every 
$ \lambda\in B(\lambda_0,r)\setminus\{\lambda_0\} $ the operator 
$ A - \lambda $ is Fredholm of the same index and $ \dim\ker (A - \lambda) $ 
and $ \dim\coker (A - \lambda) $ have constant dimension, by 
Theorem \ref{perturbation_of_spaces}. Since $ \lambda_0 $ is a boundary 
point there are $ z,w\in B(\lambda_0,r)\setminus\{\lambda_0\} $ such 
that $ z\in\sigma(A) $ and $ w\in\rho(A) $. But $ A - w\in GL(E) $ implies
that $ B(\lambda_0,r)\setminus\{\lambda_0\}\subset\rho(A) $, hence 
$ z\in\rho(A) $ and we get a contradiction.
\end{proof}
\begin{theorem}
\label{finite}
An operator $ B $ is essentially hyperbolic
if and only if $ B = A + K $, $ K \in\mathcal{L}_c (E) $, $ A $ hyperbolic.
\end{theorem}
\begin{proof}
Let $ A $ be a hyperbolic operator. We want to prove that $ A + K $ is 
essentially hyperbolic, in fact, by Proposition \ref{t+k_is_fredholm}
we have $ \sigma_e (A + K) = \sigma_e (A) $. Since $ A $ is hyperbolic
its spectrum does not meet the imaginary axis. Suppose $ B $ is essentially
hyperbolic. We show that 
$ F = \sigma(B)\cap \img\mathbb{R} $ is an isolated set in 
$ \sigma(B) $ and therefore finite (since is compact). 
We argue by contradiction. Suppose $ \lambda $ is not isolated. 
By hypothesis $ B - \lambda $ is Fredholm. Let $ C $ be
the connected component of $ \lambda $ in $ \sigma(B)\cap i\mathbb{R} $.
It is a closed interval of the imaginary axis. Let
\begin{gather*}
J = -i (C\cap \img\mathbb{R}), \ a = \max J.
\end{gather*}
By Proposition \ref{index-is-constant} $ B - a $ is Fredholm with the same
index as $ B - \lambda $. By Theorem \ref{perturbation_of_spaces} there 
exists $ r > 0 $ such that, for every
$ w\in B(ia,r) $ the operator $ B - w $ is Fredholm and
\begin{gather*}
\dim\ker (B - w), \ \ \dim\coker (B - w)
\end{gather*}
are constants, for every $ w\in B(ia,r)\setminus\{ia\} $. Since a 
connected component is maximal respect to the inclusion $ ia $ is not an 
internal point of $ \sigma(B)\cap i\mathbb{R} $, hence there exists 
$ 0 < t < r $ such that
$ i(a + t) $ is not in the spectrum of $ B $, hence $ B - i(a + t) $ is
invertible and its kernel and co-kernel are the null space, hence
$ B - i(a - t) $ is also invertible, thus the connected component of 
$ \lambda $ consists of $ \{\lambda\} $. This proves that $ \lambda $ is
not an internal point of $ \sigma(B) $; it is not isolated neither, by
hypothesis. Therefore Lemma \ref{structure} allows us to conclude that
$ \lambda\in\sigma_e (B) $ which contradicts the hypothesis.\par
Now we can write the spectrum as 
$ \sigma(B) = \sigma^+ \cup \sigma^- \cup\{\lambda_1,\dots,\lambda_n\} $
and choose a family of paths that surrounds $ \sigma (B) $ in $ \mathbb{C} $, 
say $ \Gamma = \{\gamma^+,\gamma^-,\gamma_1,\dots,\gamma_n\} $. We have
projectors $ \{P^+,P^-,P_i\} $. Since all the points of 
$ \sigma(B)\cap \img\mathbb{R} $ are isolated \textsl{eigenvalues} of $ B $, 
each $ B - \lambda_i $ is a Fredholm operator of index $ 0 $. By 
\textsc{Theorem} 5.28, Ch. IV, \S 5.4 of \cite{Kat95}, each eigenprojector
$ P_i $ has finite rank. Thus
\begin{gather}
\label{structure:1}
B = \left( B (P^+ + P^-) + \sum_{i = 1}^n P_i \right) 
- (I - B)\sum_{i = 1} ^n P_i.
\end{gather}
\end{proof}
The space $ e\mathcal{H}(E) $ is an open subset of $ \mathcal{L}(E) $ hence
is locally arcwise connected. Theorem \ref{finite} and Proposition 
\ref{retratto} allow us to connect the operator $ B $ to the square root of 
unit
\begin{gather*}
P^+ (B) - P^- (B) + \sum_{i = 1}^n P_i.
\end{gather*}
Moreover, if there exists a path that connects $ 2P - I $ and $ 2Q - I $
in $ e\mathcal{H}(E) $, by Theorem \ref{Calkin_bundle}, there exists 
$ T $ invertible such that $ TP T^{-1} - Q $ is a compact operator. For 
instance, if $ P $ is a finite rank projector and $ E $ is an infinite
dimensional space we always have at least the components: the one that contains
$ 2P - I $ and the one of $ 2(I - P) - I $. We denote them by 
$ e\mathcal{H}_+ (E) $ and $ e\mathcal{H}_- (E) $ respectively. 
By Theorem \ref{finite} we have
\begin{align*}
e\mathcal{H}_+ (E) &= \set{A\in\mathcal{H}(E)}%
{\re z > 0\ \forall z\in\sigma_e (A)} \\
e\mathcal{H}_- (E) &= \set{A\in\mathcal{H}(E)}%
{\re z < 0\ \forall z\in\sigma_e (A)}.
\glsadd{labeHpE}
\glsadd{labeHmE}
\end{align*}
These are star-shaped to $ I $ and $ -I $ respectively, hence
contractible. There are infinite dimensional Banach spaces 
(see \textsc{Corollary} 19 of \cite{GM93}) where the only complemented 
subspaces are the finite dimensional and the closed infinite dimensional.
For such spaces $ e\mathcal{H}_+ (E) $ and $ e\mathcal{H}_- (E) $ are the only
connected components of $ e\mathcal{H}(E) $.
\section{The spectral flow in Banach spaces}
A definition of spectral flow for Banach spaces and essentially hyperbolic
operators has been given in \cite{ZL99}. For sake of completeness we
restate it and show that, using the homotopy lifting properties of
the locally trivial bundle $ (\mathcal{P}(E),\mathcal{P}(\mathcal{C}),p) $, 
the spectral flow can be computed more easily and some properties, like 
homotopy invariance can be proved without considering partitions of
the unit interval.

Let $ A $ be a continuous path on $ [0,1] $ of essentially hyperbolic 
operators. By composition, we have a continuous path
\[
a(t) = \prc(P\sp + (A(t)))\in\mathcal{P}(\mathcal{C}(E)).
\]
By Theorem \ref{Calkin_bundle}, there exists a continuous path of
projectors, $ P $ such that $ \prc(P) = a $. We have an integer associated
to it
\begin{equation}
\label{eq:sf}
\Sf(A;P) = \big[P(0) - P\sp + (A(0))\big] - \big[P(1) - P\sp + (A(1))\big].
\glsadd{labSfAP}
\end{equation}
Moreover, given $ Q $ such that $ \prc(Q) = a $, by Theorem 
\ref{stability_dimension}, we have
\[
\begin{split}
\Sf(A;Q) &= \big[Q(0) - P\sp + (A(0))\big] - \big[Q(1) - P\sp + (A(1))\big]
= [Q(0) - P(0)] + \big[P(0) - P\sp + (A(0))\big] \\
&- [Q(1) - P(1)] - \big[P(1) - P\sp + (A(1))\big] = \Sf(A;P).
\end{split}
\]
\begin{definition}
Given $ A $ as above, we define the \textsl{spectral flow} as
the integer $ \Sf(A;P) $ where $ P $ is any of the paths of projectors
such that $ \prc(P(t)) = P\sp + (p(A(t))) $. We denote it by $ \Sf(A) $.
\glsadd{labSf}
\end{definition}
\begin{proposition}
\label{prop:sf-properties}
The spectral flow satisfies the following properties:
\begin{enumerate}
\item It is well behaved with respect to the catenation of paths; thus,
on the fundamental group, the spectral flow induces a $ \mb{Z} $-valued
group homomorphism;
\item the spectral flow of a constant path or a path in $ \mathcal{H}(E) $
is zero;
\item it is invariant for free-endpoints homotopies in 
$ \mathcal{H}(E) $ and for fixed-endpoints homotopies in $ e\mathcal{H}(E) $.
\end{enumerate}
\end{proposition}
\begin{proof}
i). Let $ A,B $ be two paths such that $ A(1) = B(0) $. We can choose
paths of projectors $ P $ and $ Q $ such that $ \prc(P) = \Psi(A) $ and
$ \prc(Q) = \Psi(B) $, with $ Q(0) = P(1) $. Denote by $ C $ and $ R $ the
catenation of $ A,B $ and $ P,Q $ respectively. Then,
\[
\begin{split}
\Sf(A*B) &= [R_0 - P\sp + (C_0)]
- [R_1 - P\sp + (C_1)]\\
&= [P_0 - P\sp + (A_0)] -
[Q_1 - P\sp + (B_1)] = [P_0 - P\sp + (A_1)] \\
&- [P_1 - P\sp + (A_1)]
+ [Q_0 - P\sp + (B_0)] - [Q_1 - P\sp + (B_1)] \\
&= \Sf(A) + \Sf(B).
\end{split}
\]\par
\noindent ii). If $ A $ is constant, the path $ P $ can be chosen to
be constant. Hence the spectral flow is zero. If $ A $ is hyperbolic,
$ P\sp + (A(t)) $ is continuous and can be chosen as lifting path of
$ \Psi(A) $. Hence, 
\[
\Sf(A) = [P\sp + (A(0)) - P\sp + (A(0))] - [P\sp + (A(1)) - P\sp + (A(1))] = 0.
\]\par
\noindent iii). Let $ H\colon I\times I \ra e\mathcal{H}(E) $ be a continuous
map. There exists $ P\colon I\times I\ra \mathcal{P}(E) $ such that
\[
P(t,s) - P\sp + (H(t,s))\in\mathcal{L}_c (E),\text{ for every } t,s.
\]
Let $ H(\cdot,0) = A $ and $ H(\cdot,1) = B $. We have
\[
\Sf(A) = [P(0,0) - P\sp + (H(0,0))] - [P(1,0) - P\sp + (H(1,0))].
\]
For $ i=0,1 $ and every $ s $, the operator $ P(i,s) - P\sp + (H(i,s)) $ 
is compact. The right summand is constant or continuous, 
whether the homotopy has fixed endpoints in $ e\mathcal{H}(E) $ or laying in 
$ \mathcal{H}(E) $. In both cases
\[
[P(i,s) - P\sp + (H(i,s))] = k_i\text{ for all } s,i=0,1.
\]
Thus, $ \Sf(A) = k_0 - k_1 = \Sf(B) $.
\end{proof}
Given a projector $ P $ of $ E $, we consider the connected component of
$ e\mathcal{H}(E) $ of the hyperbolic element $ 2P - I $. On its
fundamental group, we have defined the spectral flow. We have the following
\begin{theorem}
\label{f=-p}
For every projector $ P $, $ \Sf_{2P - I} = -\vfi_P\circ\Psi_* $.
\end{theorem}
\begin{proof}
Given a loop $ A $ in $ e\mathcal{H}(E) $, there exists a path of projectors
$ P $ such that $ P - P\sp + (A_t) $ is compact. By definition of $ \vfi_P $,
\[
\vfi_P (\Psi_* (A)) = [P_1 - P_0] = 
[P_1 - P\sp + (A_1)] - [P_0 - P\sp + (A_1)] = -\Sf_{2P - I} (A).
\]
\end{proof}
The theorem says, in particular, that the homomorphisms have the same kernel.
Hence we have a characterization of the kernel of the spectral flow.
\begin{proposition}
\label{sf=0}
A path loop $ A $ has spectral flow equal to zero if and only if there exists
a continuous loop $ \beta $ in $ \mathcal{P}(E) $ such that
\begin{gather*}
\beta(t) - P(A(t);\mathbb{H}^+)
\end{gather*}
is compact for every $ t\in [0,1] $.
\end{proposition}
The theorem states also that they have the same images. Thus we have a
characterization of the image of the spectral flow also. 
\begin{proposition}
\label{sf=k}
Given a Banach space $ E $ and a projector $ P $ there exists a loop of
essentially hyperbolic operators based on $ 2P - I $ with spectral flow $ k $ 
if and only if the projector $ P $ is connected to a projector $ Q $ such that
$ P - Q $ is compact and $ [P - Q] = k $.
\end{proposition}
In general all the facts proved for the index $ \varphi $ are true for the
spectral flow: if $ P\in\mathcal{P}(E) $ and the hypotheses h1), with
$ m = 1 $ and h2) hold,
the spectral flow is an isomorphism on $ \pi_1 (e\mathcal{H}(E),2P - I) $ with
$ \mathbb{Z} $.	If $ E $ satisfies the hypotheses of Proposition 
\ref{surjective} it is surjective.
\vskip .4em
\subsubsection*{s-sections of spectral projectors}
Essentially hyperbolic operators coincides are the \textsl{admissible} 
operator defined in \cite{ZL99}, for which C.~Zhu and Y.~Long define
the spectral flow. In order to compute the spectral flow, we use a 
continuous path of projectors $ P $ such that $ P - P\sp + (A(t)) $
is compact for every $ t $. According to their \textsc{Definition} 2.1,
$ P $ is a \textsl{s-section} for $ P\sp + (A(t)) $ on $ [0,1] $.
In \textsc{Definition} 2.6 of \cite{ZL99}, in order to compute the spectral 
flow, they divide the unit interval in sub-intervals where a \textsl{s-section}
of \textsl{spectral projectors} exists. In fact, globally defined s-sections of
spectral projectors do not exist in general. Consider, for instance
\[
A(t) = (2P - 1) + (2t - 1) P_m
\]
where $ P,P_m $ are projectors such that $ P P_m = P_m P = 0 $ and
$ P_m $ has finite rank $ m $ and $ P $ has infinite-dimensional kernel
and image. In conclusion, if we do not put restrictions on the choice
of an s-section, we always have \textsl{globally defined} s-sections.
\section{The Fredholm index and the spectral flow}
Given an asymptotically hyperbolic path $ A $ in $ e\al{H}(E) $ the spectral
flow can be defined as follows: since $ \mathcal{H}(E) $ is an open subset
of $ \mathcal{L}(E) $ there exists $ \delta > 0 $ such that
$ A((-\infty,-\delta]\cup[\delta,+\infty))\subset\mathcal{H}(E) $. Then
define
\begin{equation}
\label{sf_hyp}
\Sf(A) = \Sf(A,[-\delta,\delta]).
\end{equation}
That the definition does not depend on the choice of $ \delta $ follows from
ii) of Proposition \ref{prop:sf-properties}.
\begin{definition}
A splitting $ E = E_1 \oplus E_2 $ is called \emph{essential} for an operator
$ T $ if there exists a compact perturbation $ T_0 $ of $ T $ such that 
$ T_0 (E_i)\subset E_i $.
\end{definition}
In fact it is easy to check that the above splitting is essential for an 
operator $ T $ if and only if $ [T,P(E_1,E_2)] $ is a compact operator.
Given an asymptotically hyperbolic path $ A $ we denote by $ E^+ (+\infty) $
and $ E^- (+\infty) $ the images of the spectral projectors of $ A(+\infty) $.
Similarly we define $ E^+ (-\infty) $ and $ E^- (-\infty) $.
\begin{definition}
An asymptotically hyperbolic path is called \emph{essentially splitting} if 
and only if the following conditions hold:
\begin{enumerate}
\item the
splittings $ E = E^+ (+\infty)\oplus E^- (+\infty) $ and 
$ E = E^+ (-\infty)\oplus E^- (-\infty) $ are essential for $ A(t) $, 
$ t > 0 $ and $ t \leq 0 $ respectively;
\item $ E^- (-\infty) $ is compact perturbation of $ E^- (+\infty) $.
\end{enumerate}
\end{definition}
We can prove the following
\begin{theorem}\rm{(cf. \textsc{Theorem} 6.3, \cite{AM03b}).}
\label{essential_ind}
If $ A $ is asymptotically hyperbolic and essentially splitting, the
operator $ F_A $ is Fredholm and %
\glsadd{labFA}%
$ \ind F_A = \dim(E^- (A(+\infty)),E^- (A(-\infty))) $.
\end{theorem}
\begin{proof}
Denote by $ P^{\pm} (+\infty) $ and $ P^{\pm} (-\infty) $ the spectral 
projectors of $ A(\pm\infty) $. The following paths
\begin{align*}
A_+ (t) = A(t) - [A(t),P^- (+\infty)] & \ \text{ if } t > 0 \\
A_- (t) = A(t) - [A(t),P^+ (-\infty)] & \ \text{ if } t \leq 0
\end{align*}
are compact perturbations of $ A $ and leave respectively 
$ E^{\pm} (+\infty) $ and $ E^{\pm} (-\infty) $ invariant. Since 
$ A_+ (+\infty) = A(+\infty) $ by Lemma \ref{invariance} we have 
\begin{gather*}
W_{A_+} ^s =  E^- (+\infty), \ W_{A_-} ^u = E^+ (-\infty).
\end{gather*}
By Theorem \ref{compact} $ W_A ^s $ and $ W_A ^u $ are compact perturbation of
$ E^- (+\infty) $ and $ E^+ (-\infty) $. respectively. By hypothesis 
$ (E^- (+\infty),E^+ (-\infty)) $ is a Fredholm pair. By Proposition
\ref{transitivity_of_dimension}, the pair $ (W_A ^s,W_A ^u) $ is Fredholm, 
hence, by Theorem \ref{facts}, $ F_A $ is Fredholm and
\begin{equation*}
\begin{split}
\ind F_A &= \dim (W_A ^s,W_A ^u) = \dim(W_A ^s,E^- (+\infty)) + 
\ind (E^- (+\infty),E^+ (-\infty)) \\ 
&+ \dim(E^+ (-\infty),W_A ^u) = \dim(E^- (+\infty),E^- (-\infty)).
\end{split}
\end{equation*}
\end{proof}
For essential splitting path we are able to compute the spectral flow. 
First we need the following
\begin{lemma}
\label{essential_splitting}
Let $ A $ be an asymptotically hyperbolic and essentially hyperbolic path. 
It is essentially splitting also if and only if the set 
$ \set{P^+ (A(t))}{t\in\R} $ is contained in the same class of compact 
perturbation.
\end{lemma}
\begin{proof}
Suppose $ A $ is essentially splitting and consider the 
restriction on half line $ \R^+ $; hence, using the decomposition 
$ E = E^+ \oplus E^- $, we can write
\begin{gather*}
A(t) = 
\begin{pmatrix}
A_+ & K_{\pm} \\
K_{\mp} & A_{-}
\end{pmatrix}
\end{gather*} 
where $ K_{\pm} $ and $ K_{\mp} $ are compact operators because $ A $ is
essentially splitting. Since $ A_+ (+\infty) $ is hyperbolic there exists 
$ t_+ > 0 $ such that $ A_+ ([t_+,+\infty))\subset \mathcal{H}(E^+) $ and 
\begin{gather*}
\no{P^+ (A_+ (t)) - P^+ (A_+ (+\infty))} < 1.
\end{gather*}
But $ A_+ (+\infty) $ has positive spectrum, hence 
$ P^+ (A_+ (+\infty)) = I $. Since at distance smaller than $ 1 $ from the
identity there are not projectors other than the identity, 
$ P^+ (A_+ (t)) $ is the identity too on $ E^+ $ if $ t\in [t_+,+\infty) $. 
Since $ A $ is essentially hyperbolic on $ E $, $ A_+ $ is also essentially 
hyperbolic on $ E^+ $ and we have a path in $ [0,t^+] $
\begin{gather*}
A_+\colon [0,t_+]\rightarrow e\al{H}(E^+),\ A_+ (t_+)\in e\al{H}_+ (E^+);
\end{gather*}
since $ e\al{H}_+ (E^+) $ is a connected component $ A_+ ([0,t^+]) $ is 
contained in $ e\mathcal{H}_+ (E^+) $. Thus the positive
eigenspaces have finite co-dimension for every $ t > 0 $. It is easy to
check that two projectors $ P^+ (A_+ (s)) $ and $ P^+ (A_+ (s')) $
with ranges of finite co-dimension have compact difference: the operator
\begin{gather*}
P^+ (A_+ (s)) - P^+ (A(s')) = (P^+ (A_+ (s)) - I) + I - P^+ (A(s'))
\end{gather*}
is sum of finite rank operators. Similarly $ P^+ (A_- (+\infty)) = 0 $ and
there exists $ t_- < 0 $ such that the positive projector of 
$ A_- (t) $ is zero for $ t\leq t_- $. Thus $ A_- (t) $ for $ 0\geq t\geq t_- $
is a path of continuous essentially hyperbolic operators that intersect
a connected component, that is $ e\mathcal{H}_- (E^-) $; by continuity
of $ A $ the whole path lies $ e\mathcal{H}_- (E^-) $. 
If $ t_0\geq\max\{t_+,-t_-\} $ we can write for every $ t\geq 0 $
\begin{equation*}
\begin{split}
P^+ (A(t))&\sim_c P^+ (A_+ (t)) + P^+ (A_- (t))\sim_c 
P^+ (A_+ (t_0)) + P^+ (A_- (t_0)) \\
&= I_{E^+} \oplus 0_{E^-} = P^+ (+\infty)
\end{split}
\end{equation*}
where $ \sim_c $ denotes the relation of compact perturbation. Similarly, we 
can prove that $ P^+ (A(t)) $ is compact perturbation of $ P^+ (-\infty) $ 
for every $ t\leq 0 $. By hypothesis, $ P^+ (+\infty) - P^+ (-\infty) $ is 
compact, hence all the positive projectors (and thus the negative) are 
compact perturbation one of each other. Conversely, if 
$ \set{P^+ (A(t)}{t\in\mathbb{R}} $ is in the same class of compact 
perturbation, we have
\[
[A(t),P^- (+\infty)] = [A(t),P^- (A(t))] - [A(t),P^- (A(t)) - P^- (+\infty)]
\]
for $ t > 0 $. The first term of the second member is $ 0 $, the last is 
compact by hypothesis. The proof for $ t\leq 0 $ is similar.
\end{proof}
We conclude the chapter with the proof that for an asymptotically 
hyperbolic path which is essentially splitting and essentially hyperbolic 
there holds $ \Sf(A) = -\ind F_A $.
\begin{theorem}
\label{sf_of_ess_split}
Let $ A $ be an asymptotically hyperbolic path and essentially hyperbolic
such that $ \set{P^+ (A(t))}{t\in\R} $ are compact perturbation of
each other. Then
\begin{equation}
\label{eq:sf_of_ess_split}
\Sf(A) = - \dim (E^- (A(+(\infty))), E^- (A(-\infty)))
\end{equation}
\end{theorem}
\begin{proof}
Let $ \delta > 0 $ such that 
$ A((-\infty,-\delta]\cup [\delta,+\infty))\subset\mathcal{H}(E) $. Since
all the projectors are compact perturbation of each other, the spectral
flow can be computed by using $ P\equiv P\sp + (A(\delta)) $. Hence
\[
\Sf(A) = [P\sp + (A(\delta)) - P\sp + (A(-\delta))].
\]
Since $ A $ is hyperbolic in $ (-\infty-\delta]\cup [\delta,+\infty) $ the 
path $ P^+ (A(t)) $ is continuous on this subset. 
By Theorem \ref{stability_dimension},
\[
\dim(E^- (A(+\infty)),E^- (A(-\infty))) = 
- [P^+ (A(\delta)) - P^+ (A(-\delta))] = -\Sf(A).
\]
\end{proof}
Thus, Theorems \ref{sf_of_ess_split} and \ref{essential_splitting} give for
essentially splitting paths in $ e\mathcal{H}(E) $ the equality 
$ \ind F_A = - \Sf(A) $. If
$ A $ is not essentially splitting counterexamples are known even in
a Hilbert space; here we describe the \textsc{Example} 7 of 
\cite{AM03b}, Ch. 7.
\begin{example}
In Proposition \ref{the_patching_argument} we showed how to patch a 
discontinuity of a path $ A $ without changing the stable space of $ A_+ $
and the unstable space of $ A_- $. Here we describe another method; let 
$ X $ and $ Y $ be closed isomorphic subspaces that admit isomorphic 
topological complements $ X' $ and $ Y' $. Define $ P = P(X,X') $ and 
$ Q = P(Y,Y') $. We have a piecewise continuous path
\begin{gather*}
A(t) = 
\left\{
\begin{array}{ll}
2P - I  & t \geq 1 \\
2Q - I  & t \leq -1
\end{array}
\right.;
\end{gather*}
call $ A^s $ and $ A^u $ the restrictions of $ A $ to the positive and negative
half-line; by Proposition \ref{the_patching_argument}, we know that 
$ W_{A^s} ^s = X $, $ W_{A^u} ^u = Y $. There exists an invertible operator
$ T $ such that $ T Q T^{-1} = P $ which means, in particular, that
$ T Y = X $. If $ GL(E) $ is connected, there also exists a path $ U $ that 
$ U(-1) = I $ and $ U(1) = T $. Define
\begin{gather*}
A_U (t) = 
\left\{
\begin{array}{lc}
2P - I  & t \geq 1 \\
U(t) (2P - I) U(t)^{-1}   & -1\leq t\leq 1 \\ 
2Q - I  & t \leq -1
\end{array}
\right.;
\end{gather*}
the path $ A_U $ is continuous and hyperbolic, hence, by ii) of
Proposition \ref{prop:sf-properties}, $ \Sf(A_U) $ is zero. 
By iii) of Theorem \ref{facts}, the operator $ F_A $ is Fredholm if and only 
if the pair $ (X,Y) $ is Fredholm. Thus,
\begin{gather*}
\Sf(A_U)\neq - \ind F_{A_U}
\end{gather*}
if $ (X,Y) $ is a Fredholm pair of index $ k \neq 0 $.
\end{example}
The result of Theorem \ref{sf_of_ess_split} is meaningful for Hilbert spaces
too. It is interesting detecting a class of paths of essentially hyperbolic 
operators such that (\ref{eq:sf_of_ess_split}) holds and the spectral
flow does not depend on the endpoints alone, but also on the homotopy class
of the path. 
%%% TEXEXPAND: END FILE ./chapter5-v2.tex
\fancyhead[CE]{}
\renewcommand{\chaptermark}[1]{%
\markboth{\footnotesize\Alph{chapter}.\, #1}{}}
\numberwithin{remark}{chapter}
%%% TEXEXPAND: INCLUDED FILE MARKER ./appendixA-v2.tex
\appendix
%\addcontentsline{toc}{chapter}{Appendices}
\chapter{The Cauchy problem}
\label{app:cauchy-problem}
Let $ E $ be a Banach space and let
$ f $ be a function defined on a open subset 
$ \Omega\subseteq \R\times E $ with values in $ E $. We denote by
$ \Omega_t = \set{u\in E}{(t,u)\in\Omega} $. We require $ f $ to have these 
properties:
\begin{enumerate}
  \item $ f $ is continuous
  \item for any $ t\in\R $ such that $ \Omega_t \neq\emptyset $ 
    there exists an open subset $ \R\supseteq U_t\ni t $ and a constant
    $ M $ such that $ f(t',\cdot) $ is a 
    Lipschitz function with constant $ M $ for every $ t' $ in $ U_t $.
\end{enumerate}
\begin{theorem}[Cauchy]
\label{Cau}
  Let $ f $ and $ \Omega $ be as above. Then for every 
  $ (t_0, u_0) \in\Omega $ there exists an open ball $ B(t_0,r) $ and 
  $ u\in C^1 (B(t_0,r), E) $ such that $ (t,u(t))\in\Omega $ for every
  $ t\in B(t_0,r) $ and	
  \begin{gather*}
    \left\{
    \begin{array}{l}
      u'(t) = f(t,u(t)) \\
      u(t_0) = u_0
    \end{array}
    \right.;
  \end{gather*}
  moreover, if there exists an open interval $ J\ni t_0 $ and 
  $ v\in C^1 (J,E) $ satisfying the same conditions as $ (u,B(t_0,r)) $ $ u $
  and $ v $ coincide in the intersection $ B(t_0,r)\cap J $. 
\end{theorem}
\begin{proof}
Set $ z_0 = (t_0,u_0) $. There exists an open neighbourhood of $ z_0 $,
$ D(t_0,a)\times B(u_0,b')\subseteq\Omega $. By compactness of $ D(t_0,a) $ 
we can find a open ball $ B(u_0,b) $ such that 
$ f(D(t_0,a)\times B(u_0, b)) $ is bounded, call $ m $ its bound. 
For any $ r\leq a $ let $ E_r $ be the space $ C(J_r, B(u_0,b)) $ endowed with the supremum topology. If $ v\in E_r $ 
$ (t,v(t))\in J_r \times B(u_0, b)\subseteq\Omega $, thus we can define
\begin{gather*}
\Phi_f (v) = u_0 + \int_{t_0} ^t f(s,v(s)) ds.
\end{gather*}
Since $ \left|\int_{t_0} ^t f(s,v(s)) ds \right| \leq rM $ for every 
$ t\in J_r $ we have
\begin{gather*}
\Phi_f (v) (t)\in B(u_0,mr).
\end{gather*}
Still by compactness of $ D(t_0,a) $, by property iii), there 
exists $ k\in\R^+ $ such that for every $ t \in D(t_0,a) $ 
the function $ f(t,\cdot) $ is Lipschitz with constant $ k $ in $ \Omega_t $. 
Let $ v,w\in E_r $. Hence
\begin{gather*}
\no{\Phi_f (v) - \Phi_f (w)}\leq kr\no{v - w}.
\end{gather*}
If we choose $ rm < b $ and $ kr < 1 $ we make $ \Phi_f $ a contraction
of $ E_r $ into itself. Hence $ \Phi_f $ has a unique fixed point 
$ u $. Then $ (u, B(t_0,r)) $ fulfills the requirements.
\end{proof}
\begin{proposition}
Suppose $ f $ and $ \Omega $ as in the theorem. If $ u $ and $ v $ are
two solutions defined on a connected open interval $ J $ and coincide in
$ t_0 \in J $ then $ u $ and $ v $ coincide in $ J $.
\end{proposition}
\begin{proof}
Let $ A = \set{t\in J}{u(t) = v(t)} $. Since $ u $ and $ v $ are
continuous $ A $ is a closed subset of $ J $. By hypothesis we know 
that is nonempty. We prove that $ A $ is also open (hence $ A = J $).
Let $ t'\in A $, $ u_0 = u(t') = v(t') $. By Theorem \ref{Cau} there exists a
solution $ w\in C^1 (B(t',r_0),E) $ such that $ w(t') = u_0 $. By uniqueness
of local solutions $ B(t',r_0)\subseteq A $.
\end{proof}
\begin{definition}
  Let $ (u, J) $ be a solution. Then $ (v, J' ) $ is a 
  \emph{prolongation} of $ (u,J) $ if $ J\supseteq I $ and 
$ v(t) = u(t) $ for every $ t\in J $.
\end{definition}
Using Zorn's Lemma it is easy to prove that for a solution $ (u,J) $ there 
exists a unique maximal prolongation $ (v,J') $. There many criterions to 
establish when a solution $ (u,J) $ can be extended to a bigger interval 
$ J' $. Here's an example:
\begin{lemma}\label{to_bound}
Let $ (u,B(t_0,r)) $ be a solution of $ (f,\Omega) $ and suppose that the
set $ \{f(t,u(t))\} $ is bounded in $ E $ and $ I_{t_0 + r} $ and 
$ I_{t_0 - r} $ are nonempty. Then there is a prolongation 
$ (w, B(t_0,r')) $, $ r' > r $.
\end{lemma}
The Lemma can be used to prove the existence of global maximal solution in
some particular case. First we need the 
\begin{lemma}[Gronwall]\label{gron}
Let $ w,\phi,\psi $ be continuous real valued functions on the compact
interval $ [a,b] $ such that the estimate
\begin{gather*}
w(t)\leq\phi(t) + \int_a ^t w(s)\psi(s) ds;
\end{gather*}
for every $ a\leq t\leq b $. Then for every $ t $ in the interval
the estimate
\begin{gather*}
w(t)\leq\phi(t) + 
\int_a ^t \phi(s)\left(\exp\int_s ^t \psi(\xi)d\xi\right) ds
\end{gather*}
also holds.
\end{lemma}
Using Gronwall's Lemma we can prove the following statement.
\begin{proposition}\label{whole}
Suppose $ \Omega $ is the product $ J\times E $ where $ J $ an open connected 
interval of $ \R $. If for every $ t_0\in J $ there exists a function 
$ k\in C(J,E) $ such that
\begin{gather*}
|f(t,u) - f(t,v)|\leq k(|t - t_0|) |u - v|, \ \ t_0\in J;
\end{gather*}
then every solution admits a prolongation to the whole 
interval $ J $.
\end{proposition}
It is easy to check that the pair $ (f,\Omega) $ satisfies the three conditions
of the Theorem \ref{Cau}. Thus, given $ (t_0,u_0) $, there exists a maximal 
solution $ (u,B(t_0,r)) $. Since the domain $ \Omega $ is a product 
the sets $ I_{t_0 + r} $ and $ I_{t_0 - r} $ are
nonempty. Moreover, for every $ t\in B(t_0,r) $ we have the estimate
\begin{gather*}
|u(t)|\leq |u_0| + \int_{t_0} ^t k(|s - t_0|) |u(s) - u_0| ds;
\end{gather*}
applying the Gronwall's Lemma we can conclude that $ u $ is bounded, hence
admits a prolongation by Lemma \ref{to_bound}.
\vskip .2em
The Proposition \ref{whole} applies to the particular case: let 
$ \Omega = J\times E $ be the domain of $ f $ and 
$ A\in C(J,\mathcal{L}(E)) $, $ b\in C(J,E) $ be two continuous functions. The
Cauchy problem
\begin{gather*}
f(t,u) = A(t) u + b(t), \ \ \Omega = J\times E
\end{gather*}
admits unique global solutions defined on $ J $. We conclude by remarking that
the theorems of existence, prolongation and the related results can be
restated in a more general setting: by \textsl{step function} we mean
a finite sum of characteristic functions. Let $ \scr{C}(J,E) $ be the vector 
space of step function. As a subset of $ L^{\infty} (J,E) $ we can consider 
the closure $ \orl{\scr{C}} $.
\begin{definition}
\label{regulated_function}
An element of $ \orl{\scr{C}} $ is called \emph{regulated function}.
\end{definition}
Here are the hypotheses of the Theorem \ref{Cau} for regulated functions:
we $ f $ and $ \Omega $ to solve the conditions
\begin{enumerate}
\item for every $ w\in C(J,E) $ such that $ \{(t,w(t))\}\subseteq\Omega $
$ f(t,w(t)) $ is regulated,
\item for any point $ (t,u)\in\Omega $ there are an open neighbourhood
$ B(t,r)\times B(u,b) $ and $ M\in\R^+ $ such that 
$ f $ is bounded $ B(t,r)\times B(u,b) $, and $ f(s,\cdot) $ is Lipschitz 
with constant $ M $.
\end{enumerate}
For the proofs and more details see \cite{Die87}.
\newpage
%%% TEXEXPAND: END FILE ./appendixA-v2.tex
%%% TEXEXPAND: INCLUDED FILE MARKER ./appendixB-v2.tex
\chapter{Fredholm operators}
\label{app:fredholm}
Given an operator 
$ T\colon E\rightarrow F $ we can consider the spaces $ \ker T $ and 
$ E/\ran T $. The latter is called co-kernel and is denoted by 
$ \coker T $.
\begin{definition}
An operator $ T \in \mathcal{L}(E,F) $ is called semi-Fredholm if $ \ker T $ and
$ \ran T $ are closed and at least one of $ \ker T $ and $ \coker T $ 
has finite dimension. It is said Fredholm if both have finite dimension.
\end{definition}
The \textsl{Fredholm index} of a (semi)Fredholm operator 
is $ \ind T = \dim\ker T - \dim\coker T $. We denote by $ \al{F}(E,F) $ 
the set of Fredholm operators.
\begin{proposition}
\label{t+k_is_fredholm}
If $ T\colon E\rightarrow F $ is a Fredholm operator and $ K $ a compact
operator then $ T + K $ is Fredholm operator and $ \ind (T + K) = \ind T $.
\end{proposition}
\begin{proposition}
  \label{essential_inverse}
An operator $ T\in\mathcal{L}(E,F) $ is Fredholm if and only if is 
essentially invertible, that is, there exists $ S\in\mathcal{L}(F,E) $ such that 
\begin{align*}
    S T &= I + K\\
    T S &= I + H
\end{align*}
where $ K $ and $ H $ are compact operators on $ E $ and $ F $ respectively.
\end{proposition}
\begin{proof}
Since $ \ker T $ and $ \ran T $ are complemented subspaces of $ E $ and
$ F $ respectively there are $ X\subset E $ and $ Y\subset F $ such that
$ E = \ker T\oplus X $ and $ F = Y\oplus\ran T $. The restriction of $ T $
to $ X $ maps isomorphically $ X $ onto $ \ran T $, let $ \sigma $ be its
inverse. Hence, given a pair $ (y,r) $ in $ F $ we have
\begin{gather*}
T\circ (0\oplus\sigma) (y,r) = r;
\end{gather*}
hence
\begin{gather*}
T\circ (0\oplus\sigma) = P(\ran T,Y) = I - P(Y,\ran T)
\end{gather*}
where the last term denotes the projector onto $ Y $
along $ \ran T $. Since $ Y $ has finite dimension it is a perturbation of
the identity by a finite-rank operator, hence compact. Similarly
\begin{gather*}
(0\oplus\sigma)\circ T = P(X,\ker T) = I - P(\ker T,X)
\end{gather*}
is a compact perturbation of the identity. Hence we can choose 
$ S = 0\oplus\sigma $. In order to prove the converse observe that if $ S $
is an essential inverse of $ T $ we have the inclusions
\begin{align*}
  \ker T &\subset\ker S\circ T = \ker (I + K),\\
  \ran T &\supset\ran T\circ S = \ran (I + H)
\end{align*}
where the right members have finite dimension and finite co-dimension because
by Proposition \ref{t+k_is_fredholm} a compact perturbation of the identity
is Fredholm.
\end{proof}
\begin{proposition}
\label{sum-of-index}
  Let $ A\in\mathcal{L} (E,F) $ and $ B\in\mathcal{L}(F,G) $ be two
  Fredholm operators. Then $ BA $ is Fredholm and its index is
  $ \ind B + \ind A $.
\end{proposition}
\begin{proof}
For the sake of simplicity we denote by $ k $ and $ c $ the dimension of
the kernel and the co-kernel respectively. Set $ T = BA $.
Since $ A $ is Fredholm there exists a finite-dimensional subspace 
$ X\subset E $ such that
\begin{gather*}
\ker T = \ker A \oplus X;
\end{gather*}
the restriction of $ A $ to $ X $ is an isomorphism with $ \ker B\cap\ran A $.
Thus
\begin{gather}
\label{kt}
k(T) = k(A) + \dim\ker B\cap\ran A.
\end{gather}
The image of $ T $ is $ B(\ran A) $. Consider the inclusion of subspaces
\begin{gather*}
B(\ran A)\subset\ran B\subset G;
\end{gather*}
the co-dimension of $ B(\ran A) $ in $ \ran B $ can be computed as the 
co-dimension of $ \ran A + \ker B $ in $ F $, hence
\begin{equation}
\label{ct}
\begin{split}
c(T) &= c(B) + \codim(\ran A + \ker B) \\
&= c(B) + \codim\ran A - (k(B) + \dim\ran A\cap\ker B).
\end{split}
\end{equation}
Thus adding the results of (\ref{kt}) and (\ref{ct}) we obtain
\begin{equation*}
\begin{split}
\ind T &= k(T) - c(T) = k(A) + \dim\ker B\cap\ran A - c(B) - c(A) \\
&- k(B) - \dim\ran A\cap\ker B = \ind A + \ind B.
\end{split}
\end{equation*}
\end{proof}
\begin{proposition}
\label{index-is-constant}
The subset $ \al{F}(E,F)\subset\mathcal{L}(E,F) $ is open and the Fredholm
index is a locally constant function with values in $ \mathbb{Z} $.
\end{proposition}
\begin{proof}
We use the Proposition \ref{essential_inverse}. Let $ T $ be a Fredholm
operator and $ S $ be an essential inverse, that is
$ TS - I $ is a compact operator. For every operator $ H $ such that
$ \no{H} < \no{S}^{-1} $ we have
\begin{equation*}
   (T + H) S = TS + HS = I + K + HS = (I + HS) + K
\end{equation*}
where $ K $ is a compact operator; since $ I + HS $ is invertible we can
multiply both terms by its inverse and obtain
\begin{gather*}
    (T + H) S(I + HS)^{-1} = I + K(I + HS)^{-1}
\end{gather*}
hence $ S(I_F + HS)^{-1} $ is an essential right inverse for $ T + H $.
Similarly we can write $ S(T + H) = I + SH + K' $ where $ K' $ is compact.
Since $ I + SH $ is invertible we obtain
\begin{equation*} 
     (I + SH)^{-1} S (T + H) = I + (I + SH)^{-1} K' 
  \end{equation*}
and prove that $ T + H $ has an essential left inverse also. Hence 
$ B(T,\no{S}^{-1})\subset\al{F}(E,F) $. We compute the index of
$ T + H $ using the Propositions \ref{sum-of-index} and \ref{t+k_is_fredholm}
\begin{equation*}
 \begin{split} 
   \ind(T + H) &= - \ind S(I + HS)^{-1} = - \ind S - 
\ind(I + HS)^{-1} \\
  &= - \ind S = \ind T.
 \end{split} 
\end{equation*}
\end{proof}
The preceding statement and the Proposition \ref{t+k_is_fredholm} say that 
the index of a Fredholm operator is stable under small or compact 
perturbations. Here we state a more specific result regarding the dimension 
of the kernel and the co-kernel
\begin{theorem}%
{\rm (cf. \textsc{Theorem} 5.31, ch. IV \S 5.5 of \cite{Kat95}.)}
\label{perturbation_of_spaces}
Let $ T $ be a semi-Fredholm operator from $ E $ to $ F $ and $ A $ bounded. 
There exists $ \delta > 0 $ such that, for every $ 0 < |\lambda| < \delta $ 
the quantities
\begin{gather*}
\dim\ker(T + \lambda A), \ \ \dim\coker(T + \lambda A)
\end{gather*}
are constants.
\end{theorem}
In order to prove the theorem we need the following lemma.
\begin{lemma}
Let $ T $ be an operator with finite-dimensional kernel from $ E $ to $ F $ 
and $ X\subset E $ a closed subspace. Then $ T(X) \subset F $ is closed.
\end{lemma}
\begin{proof}
We use the fact that an open linear operator maps closed subspaces containing
the kernel in closed subspaces. The purpose is to show that there exists
$ Y\subset E $ closed such that $ T(Y) = T(X) $ and $ Y\supset\ker T $.
Such space can be taken as $ Y = \ker T + X $ which is closed because the
kernel has finite dimension.
\end{proof}
We are now able to prove the theorem. First we show that the theorem cannot
be extended to a neighbourhood of zero. Let $ P $ be a projector of
finite co-dimension non surjective, hence it is a Fredholm operator and
let $ A = I - P $. Let $ x\in\ker (P + \lambda A) $ with $ \lambda\neq 0 $.
We can write
\begin{gather*}
Px = -\lambda (I - P)x
\end{gather*}
hence both $ -\lambda (I - P)x $ and $ Px $ are zero. Since $ \lambda\neq 0 $ 
we also have $ (I - P)x = 0 $ thus $ x = Px + (I - P)x = 0 $. We have
proved that $ P + \lambda A $ is injective, but $ P $ is not injective.\par
Suppose first that $ \ker T $ has finite dimension. Using induction we can 
build two decreasing sequences of closed subspaces 
$ \{E_n\} $, $ \{F_n\} $ of $ E $ and $ F $ respectively as follows
\begin{align*}
\left\{
\begin{array}{l}
E_0 = E \\
E_{n + 1} =  A^{-1} (T E_n)
\end{array}
\right. \ \ \ 
\left\{
\begin{array}{l}
F_0 = F \\
F_{n + 1} = T E_n
\end{array}
\right.
\end{align*}
these are all closed spaces by the previous lemma. We have 
$ A E_n \subset F_n $ and $ T E_n = F_{n + 1} $ for any $ n\in\mb{N} $. Let 
\begin{align*}
E_{\omega} &= \bigcap_{n\geq 0} E_n \\  
F_{\omega} &= \bigcap_{n\geq 0} F_n ;
\end{align*}
If $ x\in\ker(T + \lambda I) $ and $ \lambda\neq 0 $ using induction on
the equality $ \lambda^{-1} T x = - A x $ it is easy to check that 
$ x\in E_{\omega} $. It is clear that $ T(E_{\omega})\subset F_{\omega} $;
we prove now that $ T(E_{\omega}) = F_{\omega} $. Given $ y\in E_\omega $
\[
T^{-1} (\{y\})\cap E_{\omega} = T^{-1} (\{y\}) \cap
    \big (\bigcap_{n\geq 1} E_n \big) = 
    \bigcap_{n\geq 1} \big (T^{-1} (\{y\}) \cap E_n );
\]
since $ F_{n + 1} = T(E_n) $ for $ n\geq 1 $ the last member is a
decreasing intersection of finite-dimensional, since $ \ker T $ has
finite dimension, of affine subspaces. Hence the intersection is nonempty. 
Call $ T_{\omega} $ the restriction of $ T $ to $ E_{\omega} $. We 
proved that $ T_{\omega} $ is surjective and Fredholm. By 
Proposition \ref{index-is-constant} there exists $ \delta > 0 $ such that the 
operator $ T_{\omega} + \lambda A_{\omega} $ is Fredholm, of constant index, 
and surjective. If $ |\lambda| < \delta $ and $ \lambda\neq 0 $ 
\begin{gather*}
\ind (T_{\omega} + \lambda A_{\omega}) = \dim\ker (T + \lambda A).
\end{gather*}
and is still constant as long as $ \lambda\neq 0 $. If $ \coker T $ has
finite dimension the same steps can be repeated for $ T^* $.
%%% TEXEXPAND: END FILE ./appendixB-v2.tex
%%% TEXEXPAND: INCLUDED FILE MARKER ./appendixC-v2.tex
\chapter{Spectral decomposition}
\label{app:spectrum}
We recall some basic definitions and results on spectral theory. Given
a Banach algebra $ \mathcal{B} $ with unit $ 1 $, the spectrum of an element 
$ x\in\mathcal{B} $ is the set
\[
\set{\lambda\in\mathbb{C}}{x - \lambda\cdot 1\not\in G(\mathcal{B})}
\glsadd{labGb}
\] 
where $ G(\mathcal{B}) $ is set of invertible elements of the algebra; since
this is an open subset of the algebra, the spectrum is a closed subset of
the complex plane. It is usually denoted by $ \sigma(x) $ or 
$ \sigma_{\mathcal{B}} (x) $. Moreover, the following properties hold:
\begin{enumerate}
\item $ \sigma(x) $ is compact;
\item $ \sigma(tx) = t\sigma(x) $ and $ \sigma(x + t) = \sigma(x) + t $;
\item given $ p $ and $ q $ idempotents elements of $ \mathcal{B} $ such
that $ p + q = 1 $, we define two sub-algebras
\[
\mathcal{B}_p = \{pyp:y\in\mathcal{B}\},\ \ 
\mathcal{B}_q = \{qyq:y\in\mathcal{B}\}.
\glsadd{labBp}
\]
Each of the elements $ pxp $ and $ qxq $ has a spectrum in the respective
algebra it belongs and
\[
\sigma(x) = \sigma_{\mathcal{B}_p} (pxp) \cup \sigma_{\mathcal{B}_q} (qxq).
\]
\end{enumerate}
\begin{definition}
\label{defn:surround}
Let $ \Omega\subset\C $ be an open subset of the complex plane and 
$ K\subset\Omega $, compact. Let $ \Gamma $ be a collection of continuous 
curves $ \gamma_i:[a,b]\ra\C $ such that $ \gamma_i\cap K = \emptyset $.
We say that $ \Gamma $ \textsl{surrounds} $ K $ in $ \Omega $ if
\[
\Ind_{\Gamma} (\zeta) = 
\frac{1}{2\pi i}\int_{\Gamma} \frac{d\lambda}{\lambda - \zeta} = 
\left\{
\begin{array}{ll}
1 & \text{ if } \zeta\in K \\
0 & \text{ if } \zeta\notin\Omega
\end{array}
\right.
\]
where $ \Ind_{\Gamma} (\zeta) $ is the sum of $ \ind_{\gamma_i} (\zeta) $.
\end{definition}
\begin{lemma}
\label{surr}
Suppose $ \mathcal{B} $ is a Banach algebra, $ x\in \mathcal{B} $, 
$ \alpha\in\mathbb{C} $, $ \alpha\notin\sigma(x) $ and $ \Gamma $ surrounds 
$ \sigma(x) $ in $ \Omega $. Then
\begin{gather*}
\frac{1}{2\pi i}\int_{\Gamma} (\alpha - \lambda)^n (\lambda - x)^{-1}
d\lambda = (\alpha - x)^n.
\end{gather*}
for every $ n\in\mathbb{Z} $.
\end{lemma}
The proof is made by induction on $ n $. The case $ n = 0 $ is provided
by the Neumann series (see \cite{Rud91}, \textsc{Lemma} 10.24).

Let $ x\in\mathcal{B} $ and $ \sigma_+ $ and $ \sigma_- $ closed subsets 
of $ \sigma(x) $ such that $ \sigma(x) = \sigma_- \cup \sigma_+ $ and 
$ \sigma_-\cap \sigma_+ $. There is a pair of open subsets 
\begin{gather*}
\sigma_+ \subset \Omega_+,\ \sigma_-\subset\Omega_-\\
\p\Omega_+ = \gamma_+,\ \p\Omega_- = \gamma_-
\end{gather*}
where $ \gamma_\pm $ are continuous curves and $ \gamma_\pm $ surrounds
$ \sigma_\pm $ in $ \Omega_\pm $. 
\begin{theorem}
\label{decomposition_of_spectrum}
Let $ x $ and $ \gamma_\pm $ as above. Then, the integrations
\begin{gather*}
p^+ (x) = \frac{1}{2\pi\img}\int_{\gamma^+} (\lambda - x) ^{-1} d\lambda\\
p^- (x) = \frac{1}{2\pi\img}\int_{\gamma^-} (\mu - x) ^{-1} d\mu .
\end{gather*}
are projectors of $ \mathcal{B} $, called \emph{spectral projectors}. 
In the Banach algebras $ p^+\mathcal{B} p^+ $ and $ p^- \mathcal{B} p^- $ the
elements $ x p^+ $ and $ x p^- $ have spectrum $ \sigma_+ $ and $ \sigma_- $ 
respectively.
\end{theorem}
Using the \textsl{Fubini-Tonelli} theorem it can be checked that 
$ p^+ p^- = p^- p^+ = 0 $. Applying the previous lemma with $ n = 0 $ we also
have $ p^+ (x) + p^{-} (x) = 1 $. Hence 
\[
{p^+}^2 = p^+,\ \ {p^-}^2 = p^-.
\]

\begin{theorem}
\rm
\label{thm:functional-calculus}
Let $ \Omega\subset\C $ and $ x\in\mathcal{B} $ such that 
$ \sigma(x)\subset\Omega $. Let $ f $ be a holomorphic function on $ \Omega $
and $ \Gamma $ surrounding $ \sigma(x) $ in $ \Omega $. Thus, the integration
\[
\widehat{f}(x) = \frac{1}{2\pi i}\int_{\Gamma} f(z) (x - z)\sp{-1} dz
\]
defines an element of $ \mathcal{B} $. The following properties hold:
\begin{enumerate}
\item $ \widehat{fg} (x) = \widehat{f}(x)\widehat{g}(x) $;
\item $ \widehat{g\circ f}(x) = \widehat{g}(\widehat{f}(x)) $;
\item $ \sigma(\widehat{f}(x)) = f(\sigma(x)) $;
\item on the subset $ \{x:\sigma(x)\subset\Omega\} $, $ \widehat{f} $
is continuous.
\end{enumerate}
\end{theorem}
\begin{example}
\label{ex:square-root}
Let $ A $ be a bounded operator such that $ \|A\| < 1 $. There exists $ R $
such that $ R\sp 2 = I + A $. We consider the power series expansion in
a neighbourhood of the origin of $ f(z) = \sqrt{1 + z} $. Thus,
$ R = \widehat{f}(A) $ is a solution of the equation. Moreover, the
path
\[
t\mapsto\widehat{f}(tA)
\]
is continuous and connects the operator $ R $ to the identity.
\end{example}

%%% TEXEXPAND: END FILE ./appendixC-v2.tex
%%% TEXEXPAND: INCLUDED FILE MARKER ./appendixD-v2.tex
\chapter{Continuous sections of linear maps}
\label{app:sections}
We recall some classical theorem that regards continuous selection
We begin with the result of Bartle and Graves. Let $ X $ and $ Y $ be
Banach spaces and let $ L\colon E\rightarrow F $ be a linear surjective 
application. We do not require $ L $ to be bounded. Define
\begin{gather*}
I(L) = \sup_{|y| = 1} \inf_{Lx = y} |x|.
\end{gather*}
It is easy to check that if $ L $ is injective also and $ L^{-1} $ is 
bounded $ I(L) = \no{L^{-1}} $. Let $ T $ be a paracompact Hausd\"orff space. 
The conditions of the theorem are the following: for every $ t\in T $ 
we are given a bounded surjective operator $ S(t)\in\mathcal{L}(X,Y) $ which
is \textsl{strongly continuous}. Define
\begin{gather*}
M_0 (S) = \sup_{t\in T} \no{S(t)}, \ N_0 (S) = \sup_{t\in T} I(S(t))
\end{gather*}
the map $ s\colon C(T,X)\rightarrow C(T,Y), \ x\mapsto sx (t) = S(t) x(t) $
is well defined. Structures of Banach space on $ C(T,X) $ and $ C(T,Y) $ are
not required. 
\begin{theorem}
\label{bartle_graves}
Suppose both $ M_0 $ and $ N_0 $ are finite. Fix $ N > N_0 $ and 
$ \var > 0 $. For every $ y\in C(T,Y) $ there exists $ x\in C(T,X) $ such
that $ sx = y $ and
\begin{gather}
|x(t)|\leq N |y(t)| + \var.
\end{gather}
for every $ t\in T $.
\end{theorem}
For the proof see \cite{BK73}, \textsc{Theorem} 4. As application of this 
results consider the situation of two Banach spaces $ E,F $. Let $ T $ be a
topological space and $ y\in C(T,F) $ and $ x\in C(T,E) $ such that 
$ x(t)\neq 0 $ for every $ t\in T $. Let $ \hat{x} = x/|x| $
\begin{corollary}
\label{cont_sel_of_op}
For every $ \delta,\var > 0 $ there exists 
$ U_{\delta} ^{\var}\in C(T,\mathcal{L}(E,F)) $ such that 
$ U_{\delta} ^{\var}(t) x(t) = y(t) $ and
\begin{gather*}
\no{U_{\delta} ^{\var} (t)}\leq (1 + \delta) \frac{y(t)}{|x(t)|} + \var
\end{gather*}
\end{corollary}
\begin{proof}
We briefly check that the conditions of the theorem are fulfilled. As Banach
spaces we choose $ X = \mathcal{L}(E,F) $ and $ Y = F $. Since 
$ x(t)\neq 0 $ for every $ t\in T $ we have a map
\begin{gather*}
S\colon C(T,\mathcal{L}(E,F))\rightarrow C(T,F), \ U\mapsto 
U\cdot (x/|x|).
\end{gather*}
Strong continuity is trivial. Let $ t\in T $ and $ y\in F $. By Hahn-Banach
there exists $ \xi\in E^* $ such that $ \bra\xi,\hat{x}(t)\ket = 1 $, 
$ |\xi| = 1 $. Then the operator
\begin{gather*}
U\cdot z = \bra\xi,z\ket y
\end{gather*}
maps $ \hat{x}(t) $ in $ y $ and $ \no{U} = |y| $. On the other side there 
can be no operator $ U $ such that $ U\hat{x}(t) = y $ and $ \no{U} < | y | $.
This proves that $ s(t) $ is surjective and $ I(s(t)) = 1 $. Thus 
$ N_0 (S) = 1 $ and clearly $ M_0 (S) = 1 $. Fix $ \delta,\var > 0 $. Let
$ y\in C(T,F) $ be a continuous function. 
Since $ 1 + \delta > N_0 $ there exists $ U\in C(T,\mathcal{L}(E,F)) $ such 
that
\begin{gather*}
U(t)\hat{x}(t) = y(t)/|x(t)|, \ 
\no{U(t)}\leq (1 + \delta) \frac{|y(t)|}{|x(t)|} + \var.
\end{gather*}
Thus $ U(t) x(t) = y(t) $ for every $ t\in T $.
\end{proof}
\begin{proposition}
\label{cont_glb_sect}
Let $ E, F $ Banach spaces and $ f\in\mathcal{L}(E,F) $ a bounded surjective
operator. There exists a continuous map $ s\in C(F,E) $ such that 
$ f\circ s = id $.
\end{proposition}
\begin{proof}
The Theorem \ref{bartle_graves} can be applied as follows: since $ F $ is 
metric is a paracompact space. For every $ x\in F $ we define
\begin{gather*}
L(x)\colon C(F,E)\rightarrow C(F,F), \ s\mapsto f\circ s.
\end{gather*}
Since $ L $ is constant on $ F $ is clearly strongly continuous, in fact
is bounded. Then there exists $ s\in C(F,E) $ such that $ Ls = id $, thus 
$ f\circ s = id $.
\end{proof}
\begin{proposition}
\label{alg_hom}
Let $ \mathcal{A} $ and $ \mathcal{B} $ Banach algebras, 
$ \varphi\colon\mathcal{A}\rightarrow\mathcal{B} $ a surjective homomorphism.
There are local section of 
$ \varphi\colon G(\mathcal{A})\rightarrow\varphi(G(\mathcal{A})) $.
\end{proposition}
\begin{proof}
First let $ s $ be a continuous right inverse of 
$ \varphi\colon\mathcal{A}\rightarrow\mathcal{B} $. Such a section exists
by Proposition \ref{cont_glb_sect}. Let $ y_0 $ in $ \varphi(G(\mathcal{A})) $ 
and $ x_0\in G(\mathcal{A}) $ such that $ \varphi(x_0) = y_0 $. We can define 
another right inverse of $ \varphi $ such that
\[
S (y) = s(y) + x_0 - s(y_0), \ S(y_0) = x_0.
\]
Since $ G(\mathcal{A}) \subset\mathcal{A} $ is open, there exists 
$ \delta > 0 $ such that $ B(x_0,\delta)\subset G(\mathcal{A}) $. Thus 
$ S^{-1} (B(x_0,\delta))\subset\varphi(G(\mathcal{A})) $ and the restriction
of $ S $ to $ S^{-1} (B(x_0,\delta)) $ is a local section
on a neighbourhood of $ y_0 $.
\end{proof}
%%% TEXEXPAND: END FILE ./appendixD-v2.tex
\pagestyle{empty}
%\fancyhead{}
\nocite{*}
\clearpage
%\addcontentsline{toc}{chapter}{Bibliography}
%\bibliographystyle{plain}
%\bibliography{main}
\def\polhk#1{\setbox0=\hbox{#1}{\ooalign{\hidewidth
  \lower1.5ex\hbox{`}\hidewidth\crcr\unhbox0}}}

\printglossary
%\addcontentsline{toc}{chapter}{Glossary}
\end{document}